\documentclass[11pt]{article}

\usepackage{times}
\usepackage{amsmath,amsfonts,amstext,amssymb,amsbsy,amsopn,amsthm,eucal}
\usepackage{txfonts}
\usepackage{dsfont}
\usepackage{graphicx}   
\usepackage{hyperref}
\usepackage{enumerate}
\usepackage{cancel}

\hypersetup{
  pdftitle={Volume estimates on the critical sets of solutions to elliptic PDEs},
  pdfauthor={Aaron Naber and Daniele Valtorta},
  pdfsubject={Estimates on the critical sets of solutions to elliptic PDEs with Lipschitz leading coefficients, bounded first order coefficient and no zero order coefficient},
  pdfkeywords={Aaron, Naber, Daniele, Valtorta, critical set, volume estimates, Minkowski content, Lipschitz, elliptic, PDE, n-2},
  pdfpagelayout=SinglePage,
  pdfpagemode=UseOutlines,
  colorlinks,
  bookmarksopen,
  linkcolor=[rgb]{0,0,0.7},
  urlcolor=[rgb]{0,0,0.4},
  citecolor=[rgb]{0.4,0.1,0}
}

\usepackage{tocbibind}

\newcommand{\Vol}{\operatorname{Vol}}

\newcommand{\dN}{\mathds{N}}
\newcommand{\dR}{\mathds{R}}

\newcommand{\N}{\mathbb{N}}
\newcommand{\R}{\mathbb{R}}

\newcommand{\cC}{\mathcal{C}}
\newcommand{\cCS}{\mathcal{CS}}

\newcommand{\cH}{\mathcal{H}}

\newcommand{\cP}{\mathcal{P}}

\newcommand{\cS}{\mathcal{S}}

\newcommand{\cV}{\mathcal{V}}

\newcommand{\norm}[1]{\left\|#1\right\|}

\newcommand{\ps}[2]{\left\langle#1\middle\vert#2\right\rangle}

\newcommand{\ton}[1]{\left(#1\right)}
\newcommand{\qua}[1]{\left[#1\right]}
\newcommand{\cur}[1]{\left\{#1\right\}}
\newcommand{\abs}[1]{\left|#1\right|}

\renewcommand{\L}{\mathcal{L}}

\newcommand{\Cr}{\mathcal{C}}
\newcommand{\Cru}{\mathcal{C}(u)}

\begin{document}

\newtheorem{theorem}{Theorem}[section]
\newtheorem*{theorem*}{Theorem}
\newtheorem{ctheorem}[theorem]{Conjectural Theorem}

\newtheorem{proposition}[theorem]{Proposition}

\newtheorem{lemma}[theorem]{Lemma}
\newtheorem{clemma}[theorem]{Conjectural Lemma}

\newtheorem{corollary}[theorem]{Corollary}

\theoremstyle{definition}
\newtheorem{definition}[theorem]{Definition}

\theoremstyle{remark}
\newtheorem{remark}[theorem]{Remark}

\theoremstyle{remark}
\newtheorem{example}[theorem]{Example}

\theoremstyle{remark}
\newtheorem{note}[theorem]{Note}

\theoremstyle{definition}
\newtheorem{notation}[theorem]{Notation}

\theoremstyle{remark}
\newtheorem{question}[theorem]{Question}

\theoremstyle{remark}
\newtheorem{conjecture}[theorem]{Conjecture}

\title{Volume estimates on the critical sets of solutions to elliptic PDEs}

\author{Aaron Naber and Daniele Valtorta \footnote{supported by SNSF grant 149539 and GNAMPA}}

\date{\today}
\maketitle
\begin{abstract}
In this paper we study solutions to elliptic linear equations $L(u)=\partial_i(a^{ij}(x)\partial_j u) + b^i(x) \partial _i u + c(x) u=0$, either on $\dR^n$ or a Riemannian manifold, under the assumption that the coefficient functions  $a^{ij}$ are Lipschitz bounded.  We focus our attention on the critical set $\Cr(u)\equiv\{x:|\nabla u|=0\}$ and the singular set $\cS(u)\equiv\{x:u=|\nabla u|=0\}$, and more importantly on effective versions of these.  Currently, with just the Lipschitz regularity of the coefficients, the strongest results in the literature say that the singular set is $n-2$-dimensional, however at this point it has not even been shown that $\cH^{n-2}(\cS)<\infty$ unless the coefficients are smooth.  Fundamentally, this is due to the need of an $\epsilon$-regularity theorem which requires higher smoothness of the coefficients as the frequency increases.  We introduce new techniques for estimating the critical and singular set, which avoids the need of any such $\epsilon$-regularity.  Consequently, we prove that if the frequency of $u$ is bounded by $\Lambda$, then we have the estimates $\cH^{n-2}(\cC(u))\leq C^{\Lambda^2}$, $\cH^{n-2}(\cS(u))\leq C^{\Lambda^2}$, depending on whether the equation is critical or not.  More importantly, we prove corresponding estimates for the {\it effective} critical and singular sets.  Even under the assumption of real analytic coefficients these results are much sharper than those currently in the literature.  We also give applications of the technique to give estimates on the volume of the nodal set of solutions and estimates for the corresponding eigenvalue problem.
\end{abstract}

\tableofcontents

\section{Introduction:}
In this paper, we study solutions $u$ to second order linear homogeneous elliptic equations with Lipschitz leading coefficients.  That is, we will study on $\dR^n$ solutions $u$ to the equation
\begin{gather}\label{eq_Lu}
\L(u)=\partial_i(a^{ij}(x)\partial_j u) + b^i(x) \partial _i u + c(x) u=0\, ,
\end{gather}
where the coefficients $a$ are Lipschitz and the coefficients $b,c$ are bounded and measurable. For effective estimates we assume the bounds

\begin{align}\label{e:coefficient_estimates}
&(1+\lambda)^{-1}\delta^{ij}\leq a^{ij}\leq (1+\lambda)\delta^{ij}, \,\notag\\
&\text{Lip}(a^{ij}), \,\, \abs{b^i},\,\,\abs{c}\leq \lambda\, .
\end{align}

We will call the equation {\it critical} if $c\equiv 0$.  Given the local nature of the estimates and the techniques involved in their proof, it is not restrictive to assume for simplicity that $u$ is defined on the ball $B_1(0)\subset \dR^n$. With simple modifications the results are easily extensible also to more general domains in Riemannian manifolds.  Using a new covering argument (see Section \ref{ss:volume_estimates_harmonic}), we prove new $n-2$-{\it Minkowski} estimates on the critical and singular sets
\begin{align}
&\Cr(u)\equiv\{x\in B_{1/2}(0):|\nabla u|=0\}\, ,\,\,\,\,\, \cS(u)\equiv\{x\in B_{1/2}(0):u=|\nabla u|=0\}\, .
\end{align}
In principal, for general equations we will prove estimates on $\cS(u)$, while for critical equations we will prove estimates on $\cC(u)$.  It will be convenient to denote by $\cCS(u)$ either the critical set $\cC(u)$ or singular set $\cS(u)$, depending on whether the equation (\ref{eq_Lu}) is critical or not, respectively.  Note that if we wish to control the $n-2$-measure of critical or singular sets, then the assumption of Lipschitz coefficients is a sharp assumption, since if the coefficients are only H\"older one can find nontrivial solutions which vanish on open subsets, see \cite{plis}. 

\vspace{2mm}

\subsection{Effective estimates}
More importantly we will prove estimates on the {\it effective} critical and singular sets $\cC_r(u)$, $\cS_r(u)$.  The effective critical and singular sets were first introduced by the authors in \cite{ChNaVa}.  In essence, $x\not \in \cC_r(u)$ if on the definite size ball $B_r(x)$ we have that $|\nabla u|$ has a definite size relative to $u$, more precisely we have:

\begin{align}\label{e:effective_critical}
&\cC_r(u)\equiv\cur{x\in B_{1/2}(0): \inf_{B_r(x)} r^2 |\nabla u|^2 < \frac{n}{16}\fint_{\partial B_{2r}(x)} |u-u(x)|^2}\, ,\notag\\
&\cS_r(u)\equiv\cur{x\in B_{1/2}(0): \inf_{B_r(x)}\ton{\abs{u}^2 + \frac{r^2}{n} |\nabla u|^2 } < \frac{1}{16}\fint_{\partial B_{2r}(x)} |u|^2}\, .
\end{align}
Again we will denote by $\cCS_r(u)$ the effective critical or singular set, depending on whether the equation is critical or not.  Notice that for every $r>0$ we have that $\cCS(u)\subseteq \cCS_r(u)$, and more effectively that points of $\cCS_r(u)$ are those points which have a definite amount of gradient on a ball of definite size.

\subsection{Background}  To control the critical and singular sets of a solution to (\ref{eq_Lu}) more information about the solution is needed.  For instance, one could just take the solution $u=0$, which by all regards is a great solution but there is no control on the critical and singular set.  It has been understood for some time that being a constant or close to a constant is all that can really go wrong, and hence what is important is to control how far away $u$ is from a constant solution.  The right measurement for this are the frequencies
\begin{align}\label{d:frequency}
N^u_{\cC}(x,r)\equiv \frac{r\int_{B_r(x)} |\nabla u|^2}{\int_{\partial B_r(x)} (u-u(x))^2}\, ,\,\, N^u_{\cS}(x,r)\equiv \frac{r \int_{B_r(x)} |\nabla u|^2}{\int_{\partial B_r(x)} u^2}\, ,
\end{align}
and their generalizations (see Section \ref{ss:generalized_frequency}), where we denote by $N^u(x,r)$ either $N^u_{\cC}(x,r)$ or $N^u_{\cS}(x,r)$, depending on whether (\ref{eq_Lu}) is critical or not, respectively.  By unique continuation and the maximum principle, if $u$ is not constant, then $N^u$ is well defined for positive $r$.  For a fixed solution $u$ of (\ref{eq_Lu}) we then denote by
\begin{align}
\Lambda \equiv N^u(0,1)\, ,
\end{align}
the frequency bound of $u$.  The main conjecture in the area goes back to Lin \cite{lin}, which predicts that for some constant $C(n,\lambda)$ we have that
\begin{align}
&\cH^{n-2}(\cCS(u))< C\Lambda^2\, .
\end{align}
The best that has been proved in the literature at this point goes back to \cite{hanlin,HON,hanhardtlin,hoste,HLrank}, which proves, under the assumption of {\it smooth} coefficients, that there exists constants $C(n,a,b,c,\Lambda)$ such that
\begin{align}
&\cH^{n-2}(\cCS(u))< C(n,a,b,c,\Lambda)\, .
\end{align}
In particular, $C$ depends on upper bounds on the coefficients $a$, $b$, $c$ and their higher order derivatives.  If one drops the assumption of smoothness on the coefficients, even if one assumes control over a large number of derivatives but not all, then the situation becomes drastically worse.  In this case the best that has been proven is in \cite{hanlin,han_sing}, where it was shown that Hausdorff dimension satisfies
\begin{align}
\dim\cCS(u) = n-2\, ,
\end{align}
however it was not even shown that $\cH^{n-2}(\cCS(u))<\infty$.  There is a fundamental reason for this, as the results of \cite{hanlin,han_sing} rely on an $\epsilon$-regularity theorem which requires additional smoothness as the frequency increases.  One of the main goals of this paper is to improve on these estimates by removing the need for such an $\epsilon$-regularity theorem.

In another direction there are more recent results from \cite{ChNaVa} that attempt to prove more effective estimates on the critical and singular sets.  Namely, even a Hausdorff dimension bound has limited application.  In short, the Hausdorff dimension of a set can be small while still being dense.  On the other hand, Minkowski estimates bound not only the set in question, but the tubular neighborhood of that set, providing a much more analytically effective notion of {\it size}.  For example, we recall that the set of rational numbers in $\dR^n$ has Hausdorff dimension $0$ and Minkowski dimension $n$.  What is needed for applications to nonlinear equations are control over the critical and singular sets on balls of definite size.  That is, it would be better to estimate $\Vol(B_r(\cCS(u)))$, and even better to make the statement that if $x\not\in B_r(\cCS(u))$, then the gradient of $u$ on $B_r(x)$ has some definite size.  The first results in this direction were proven in \cite{ChNaVa}, where by using the ideas of quantitative stratification it was shown under only Lipschitz coefficients that for every $\epsilon>0$:
\begin{align}
\Vol(B_r(\cCS_r(u))) < C(n,\lambda,\epsilon) r^{2-\epsilon}\, .
\end{align}
While such a Minkowski estimate on the critical set is stronger than simple Hausdorff estimates, the existence of the $\epsilon$ unfortunately still prevents one from obtaining finiteness of the $n-2$-measure.

It is worth mentioning that in the very special case of harmonic functions in $\R^2$, a sharp bound on the number of singular points (sharp as a function of the frequency $N$) is proved in \cite[theorem 3.3]{han}.

\paragraph{Nodal sets} Although in this paper the estimates on nodal sets are not the main estimates but rather secondary results, it is worth making a brief overview of the results available in literature in this context. Also in this case, \cite{hanlin} provides a suitable overview of the literature. Here we briefly cite the best results available in literature.

For nodal sets, better bounds are available in terms of the frequency $N_{\cS}^u$. The primary conjecture in the area goes back to Yau, which predicts that there exists a constant $C(n,\lambda)>0$ such that
\begin{align}
\cH^{n-1}(Z(u))< C\Lambda\, .
\end{align}
Yau's conjecture has been proven in \cite{DonFef} for analytic coefficients.  For Lipschitz coefficients the best result known are given in \cite{HardtSimon} which give the estimate
\begin{align}\label{eq_hardtsimon}
\cH^{n-1}(Z(u))< \Lambda^{C\Lambda}\, .
\end{align}
This result is stated in a more general and technical way in the paper in question, see \cite[theorem 1.7]{HardtSimon}.

The techniques of this paper, which are quite different from that of \cite{HardtSimon}, can recapture the result \eqref{eq_hardtsimon}, as well as improve it to the stronger Minkowski version.

\subsection{Main Results}

Now we briefly describe our main new results.

\paragraph{Main Result for Critical and Singular Sets:}  In this paper we have developed a new method for controlling the critical and singular sets, distinct from the techniques of either \cite{ChNaVa} or \cite{hanlin}.  Before discussing the methods, let us state our main results.

\begin{theorem}\label{t:main_critical}
Let $u:B_1(0)\to \dR$ solve (\ref{eq_Lu}) with Lipschitz coefficients satisfying (\ref{e:coefficient_estimates}). There exists $r_0=r_0(n,\lambda)>0$ and $C=C(n,\lambda)$ such that if $\Lambda\equiv N^u(0,2s)$ with $s\leq r_0$, then the following Minkowski estimates hold:
\begin{align}\label{e:mink_est}
\Vol(B_r(\cCS(u))\cap B_{s}(0)) \leq \Vol(B_r(\cCS_r(u))\cap B_{s}(0)) \leq C^{\Lambda^2}\, (r/s)^2\, .
\end{align}
\end{theorem}

\begin{remark}\label{rem_HvsM_crit}
 As a corollary, we obtain the Hausdorff measure estimate 
\begin{align}\label{e:haus_est}
\cH^{n-2}(\cCS(u)\cap B_{s}(0))\leq C^{\Lambda^2}s^{n-2}\, .
\end{align}
Note that this estimate is weaker than \eqref{e:mink_est} in two ways. First of all, uniform volume estimates on $B_r(\cCS(u))$ are stronger than Hausdorff estimates. As a guiding example, consider the set $R$ of rational points in $\R^n$. Although this set has Hausdorff dimension $0$, $B_r(R)$ covers the whole $\R^n$.

Moreover, as explained above, \eqref{e:mink_est} gives estimates not just on the critical set, but on the set $\cCS_r(u)$ defined in \eqref{e:effective_critical}. This set contains not just the critical points of $u$, but also the points where $\nabla u(x)$ is small relative to $u$ in a neighborhood of $x$.
\end{remark}

\begin{remark}
Since this statement is scale invariant, we will assume for convenience that $r_0\geq 1$. This can be obtained using a suitable blow-up of the domain of the function $u$, or, equivalently, by assuming $\lambda$ to be small enough.
\end{remark}

Before continuing let us make some remarks about Theorem \ref{t:main_critical}.  Even under the assumption of analytic coefficients the Hausdorff measure estimate of (\ref{e:haus_est}) is the first which gives an {\it effective} bound for the $n-2$ Hausdorff measure of the critical and singular sets, while of course the Minkowski estimate of (\ref{e:mink_est}) is in fact significantly stronger.  As was previously discussed, under the assumption of Lipschitz coefficients the Hausdorff estimate (\ref{e:haus_est}) is the first to give that the $n-2$-Hausdorff measure is even finite.  In fact, the techniques even show that the critical and singular sets are {\it finitely} rectifiable.  That is, away from a set of $n-2$-measure zero we have that $\cCS(u)$ is the finite union of bi-Lipschitz images of subsets of $\dR^{n-2}$.  On a manifold the constant $C$ should also depend on the sectional curvature bound of the manifold.

\paragraph{Main Results for Nodal Sets:}  By a simple adaptation of the arguments used for critical sets, we are able to also give estimates on the {\it nodal} set of solutions $u$ to \eqref{eq_Lu}.  In this case our effective Hausdorff estimates match those that are currently in the literature, however we do prove the significantly stronger Minkowski versions as well, which is quite new.  To state the results let us recall the definition of the nodal and {\it effective} nodal sets given by
\begin{align}
&Z(u)\equiv \{x\in B_{1/2}(0):u(x)=0\}\, ,\notag\\
&Z_r(u)\equiv \cur{x\in B_{1/2}(0): \inf_{B_r(x)}|u|^2(x)<\epsilon(n)\fint_{\partial B_{2r}(x)}|u|^2}\, .\label{eq_deph_Zr}
\end{align}
As with the effective critical and singular sets, the effective nodal set represents the set of points where $u$ has a definite size on a ball of definite size.  It is again the frequency which plays a key role in controlling the nodal set, though in this case it is the singular frequency $N^u_\cS$. 


Our main estimate for nodal sets is the following:

\begin{theorem}\label{t:main_nodal}
Let $u:B_1(0)\to \dR$ solve (\ref{eq_Lu}) with the coefficients satisfying (\ref{e:coefficient_estimates}).  There exists $r_0=r_0(n,\lambda)>0$ and $C=C(n,\lambda)$ such that if $\Lambda\equiv N^u(0,2s)$ for some $s\leq r_0$, then the following Minkowski estimates hold:
\begin{align}\label{e:mink_est_nod}
\Vol(B_r(Z(u))\cap B_{s}(0)) \leq \Vol(B_r(Z_r(u))\cap B_{s}(0)) \leq (C(n,\lambda)\Lambda)^{\Lambda}\, r/s\, .
\end{align}
\end{theorem}

\begin{remark}
 As for critical sets (see Remark \ref{rem_HvsM_crit}), this estimate immediately yields the Hausdorff measure bound
\begin{align}\label{e:haus_est_nod}
\cH^{n-1}(Z(u)\cap B_{s}(0))\leq (C(n,\lambda)\Lambda)^{\Lambda}s^{n-1}\, .
\end{align}
\end{remark}

\paragraph{Applications to Eigenvalue Equation on Manifolds:}

Let us now assume we are working on a compact Riemannian manifold $(M^n,g)$ with at least Lipschitz metric $g$.  In this context we are most interested in studying the Laplace-Beltrami operator $\Delta u \equiv \text{div} \nabla u$, though the results hold equally well for other second order operators.  It is well understood that the eigenvalues $0=\lambda_0<\lambda_1\leq \lambda_2\leq \cdots$ of $-\Delta$ are discrete with $\lambda_j\to \infty$.  As an application of Theorems \ref{t:main_critical} and \ref{t:main_nodal}, as well as the doubling estimate of \cite{DonFef} we have the following:

\begin{theorem}
For a compact Lipschitz Riemannian manifold $(M^n,g)$ there exists a constant $C(g)$ such that if $u$ solves the eigenfunction equation $-\Delta u = \lambda u$, then we have the Minkowski estimates
\begin{align}
&\Vol(B_r(\cS(u))) \leq \Vol(B_r(\cS_r(u))) \leq C^{\lambda}\,r^2\, ,\notag\\
&\Vol(B_r(Z(u))) \leq \Vol(B_r(Z_r(u))) \leq \lambda^{C\sqrt{\lambda}}\,r\, .
\end{align}
In particular, we have the much weaker estimate on the Hausdorff measure
\begin{align}
&\cH^{n-2}(\cS(u))\leq C^{\lambda}\, ,\notag\\
&\cH^{n-1}(Z(u))\leq \lambda^{C\sqrt{\lambda}}\, .
\end{align}

\end{theorem}

\section{Preliminaries and Outline of Proof}

In this section we concentrate on introducing the correct terminology for the paper, as well as giving an outline for the proof of the mains Theorems.  To keep the arguments as non convoluted as possible we will concentrate on proving Theorem \ref{t:main_critical} in the context where (\ref{eq_Lu}) is critical, as the other results are completely analogous.

The main new ingredient to the proof of Theorem \ref{t:main_critical} is a new covering argument, which itself relies on a new {\it effective} tangent map uniqueness statement.  In section \ref{ss:tangent_maps} we review the notion of a blow up and discuss the new results in this paper related to them.  In section \ref{ss:critical_radius} we discuss the notion of the critical radius.  In section \ref{ss:proof_outline} we outline the proof of Theorem \ref{t:main_critical}, and in particular the new covering argument. 

\subsection{Generalized Frequency}

For solutions of (\ref{eq_Lu}) it is more natural and convenient on small scales to work with a generalization of the frequency function (\ref{d:frequency}) which takes into account better the coefficients of the equation.  Among other things this allows one to preserve the essential property of {\it monotonicity} for the frequency.  Such a generalized frequency was first introduced in \cite{galin1,galin2,hanlin}, and further expanded in \cite{ChNaVa}.  We will follow the mild extensions given in \cite{ChNaVa}, which are discussed in Section \ref{ss:generalized_frequency}.  For now, we simply wish to remark that the frequency $N$ in the next subsections will refer to the generalized frequency.

\subsection{Tangent Maps and Blow Ups}\label{ss:tangent_maps}

In this subsection we define the notion of a blow up and discuss both new and old results related to it.  To discuss this with precision, let us define for $x\in B_1(0)$ the linear transformation
\begin{gather}\label{eq_Q}
 Q_{x} (y) = q_{ij}(x)  \ y^i e^j\, ,
\end{gather}
where $q^{ij}(x)$ is the square root of the matrix $a^{ij}(x)$.  For instance if we consider just the Laplacian then $Q\equiv I$ is just the identity map.  It is evident that $Q_x$ is a bi-Lipschitz equivalence from $\R^n$ to itself with Lipschitz constant $\leq (1+\lambda)^{1/2}$.  Thus if $u$ solves (\ref{eq_Lu}) with $x\in B_{1/2}(0)$ and $r<\frac{1}{2(1+\lambda)^{1/2}}$, then we can define the blow up by 
\begin{definition}[{\bf Tangent map for $u$}]\label{d:tangent}
\begin{enumerate}
\item For $x\in B_{1/2}(0)$ and $r<\frac{1}{2(1+\lambda)^{1/2}}$ we define $T^u_{x,r}:B_{r^{-1}}(0)\to\dR$ by
\begin{gather}\label{eq_dephT_Q}
 T_{x,r}^u(y) = \frac{u(x+rQ_x(y))-u(x)}{\ton{\fint_{\partial B_1(0)} \abs{u(x+rQ_x (y))-u(x)}^2 }^{1/2}}\, .
\end{gather}
\item For $x\in B_1(0)$ we define
\begin{gather}\label{eq_T}
 T_x^u(y)=\lim_{r\to 0} T_{x,r}^u (y)\, .
\end{gather}
\end{enumerate}
\end{definition}
By unique continuation and the maximum principle, $T^u_{x,r}$ is defined for all positive $r$ sufficiently small. The existence of the limit is a different matter. If the coefficients of the equation are smooth, its existence is an easy consequence of Taylor's theorem and the unique continuation principle. In this case, the limit is unique and, up to rescaling, $T_x^u$ is just the leading order polynomial of the Taylor expansion of $u-u(x)$ at $x$.  In the general case, the existence of the limit has been proved in \cite{han_sing} and is a deeply important property of solutions to \eqref{eq_Lu}. It is worth underlying that not only the limit in \eqref{eq_T} exists pointwise in $y$, but \cite{han_sing} proves a definite rate of convergence in $r$ related to the frequency $N$. 

\vspace{3mm}
Using a simple change of variables, it is easy to see that the function $T$ satisfies an 
elliptic PDE of the form:
\begin{gather}\label{eq_LT}
 \tilde \L(u)=\partial_i\ton{\tilde a^{ij} \partial _j T} + \tilde b^i \partial_i T +\tilde c T =0\, ,
\end{gather}
with $\tilde a^{ij}(0)=\delta^{ij}$. Moreover, the conditions \eqref{e:coefficient_estimates} imply similar estimates for the coefficients $\tilde a^{ij}, \ \tilde b^i$:
\begin{gather}\label{eq_aT}
 (1+\lambda r)^{-1}\delta^{ij}\leq \tilde a^{ij}\leq (1+\lambda r)\delta^{ij}, \, \text{Lip}(\tilde a^{ij})\leq \lambda r\, , \,
\abs{\tilde b^i},\abs{\tilde c} \leq \lambda r\, .
\end{gather}

An important property of the blow ups is that they are controlled by the frequency.  We say the frequency at $x$ is $\delta$-pinched on the scales $[r_2,r_1]$ if $|N(x,r_2)-N(x,r_1)|<\delta$.  It is known, see \cite{ChNaVa} for instance, that for every $\epsilon>0$ there is a $\delta>0$ such that if the frequency is $\delta$-pinched at $r$, then $T^u_{x,r}$ is $\epsilon$ close to some homogeneous harmonic polynomial $P_d$.  The primary weakness of this result from \cite{ChNaVa}, besides its lack of effectiveness, is that if the frequency is $\delta$-pinched over a potentially large number of scales $[r_2,r_1]$, then the homogeneous harmonic polynomial $P^{(r)}_d$ which $T^u_{x,r}$ is close to might conceivably depend on $r$.  

In this paper, using arguments which extends both those of \cite{ChNaVa} and the tangent map uniqueness result of \cite{han_sing}, we prove a result which strengthens both of these into a more quantitative statement.  Namely, we see that for every $\epsilon>0$ there is a $\delta>0$, which is in fact given explicitly and sharply, such that if the frequency is $\delta$ pinched on scales $[r_2,r_1]$ with $0\leq r_2< r_1/(c\Lambda)$, then there exists a {\it unique} homogeneous harmonic polynomial $P_d$ such that $T_{x,r}$ is $\epsilon$-close to $P_d$ for all $r\in [r_2,r_1]$.  See Theorem \ref{t:eff_tan_con_uniq_harm} for the harmonic case and Theorem \ref{t:eff_tan_con_uniq} for the general case.  Both the uniqueness and the sharp bounds of the constants play an important role in Theorem \ref{t:main_critical}.

\subsection{The Critical Radius}\label{ss:critical_radius}

Let us begin with the following definition of the critical radius:
\begin{definition}
 Given $x\in B_{1/2}(0)$ we define
 \begin{gather}\label{eq_deph_rc}
  r_c(x)\equiv r_x = \sup\cur{0\leq s\leq r_0:  N(x,s)<\frac{3}{2}} \, ,
 \end{gather}
 where $r_0(n,\lambda)>0$ is a small constant chosen from Lemma \ref{l:frequency_comparison}.
\end{definition}

Let us quickly remark on the following, which is easy to prove, see for instance \cite{ChNaVa} and Lemma \ref{lemma_ell3/2}:

\begin{lemma}\label{l:frequency_comparison}
There exists $C(n,\lambda)>0$ and $r_0(n,\lambda)>0$ such that if $r_x<C^{-1}r$ with $r\leq r_0$, then $\inf_{B_r(x)}|\nabla u|^2 > \frac{n}{4r^2}\fint_{\partial B_{2r}(x)} |u-u(x)|^2$.  In particular, we have the inclusion
\begin{align}
\cC_r(u)\subseteq\{x\in B_{1/2}(0):\ r_x\leq C^{-1}r\}\, .
\end{align}
\end{lemma}

The above Lemma allows us in the proof of Theorem \ref{t:main_critical} to prove a volume estimate on the set $\{r_x<r\}$, which will be more natural in the context of the frequency.

Let us now generalize the above definition in \eqref{eq_deph_rc}, as it will play an important role in our covering argument.  For $d\in\dN$ a fixed integer let us define the following $d$-critical radius:

\begin{definition}
 Given $x\in B_1(0)$ in its domain, we define the $d$-critical radius
 \begin{gather}
  r^d_x = \sup\cur{0\leq s\leq r_0: \forall y\in B_s(x) \text{ we have that } N(y,s)< d+\epsilon_0} \, .
 \end{gather}
\end{definition}
\begin{remark}
Though not supremely important at this stage, the constant $\epsilon_0(n,\lambda)>0$ is chosen from Corollary \ref{c:cone_splitting} by setting $\tau=10^{-6}$, and the radius $r_0(n,\lambda, d)$ is chosen from Theorem \ref{t:eff_tan_con_uniq}. Note that, as $d\to \infty$, $r_0(n,\lambda,d)\to 0$.
\end{remark}

For the sake of the outline all that is important is that the constants $\tau, \epsilon_0,r_0$ in the above definition are chosen small enough in such a way that one has the effective cone splitting of Corollary \ref{c:cone_splitting}.

\subsection{Outline of Proof}\label{ss:proof_outline}

The proof comes by working inductively on $d$-critical balls.  In this short subsection we will try to outline the general idea of the construction without worrying about precision or technical details.  

Let $u$ be a solution of (\ref{eq_Lu}) such that the frequency is bounded by $\Lambda$ as in the statement of Theorem \ref{t:main_critical}.  Using the results of Sections \ref{s:harmonic} and \ref{s:general_elliptic} it is not hard to see that there exists a constant $C(n,\lambda)$ and an integer $d\leq C\Lambda$ such that for each $x\in B_{1/2}(0)$ we have the $d$-critical radius bound $r^d_x\geq r_0$, where $r_0\geq \epsilon(n,\lambda)^{-d}$.  We can cover $B_{1/2}(0)$ by at most $C^\Lambda$ such balls, thus there is no harm in estimating each such ball individually and adding up the error.  Note that on each such ball that after rescaling $r_x\to 1$ and translating $x\to 0$ we can assume we are working on a ball $B_1(0)$ such that $r^d_0\geq 1$.  

Now let us fix $r>0$ and assume $B_1(0)$ is such that $r^d_0\geq 1$ as above.  The rough goal is to find a collection of balls $\{B_{r_i}(x_i)\}\subseteq B_1(0)$ with the following properties:
\begin{enumerate}
\item \label{i1}For each $i$ either $r_i=r$ or $r_i = r_{x_i}^{d-1}$.
\item \label{i2} If $x\not\in \cup B_{r_i}(x_i)$, then $r_x<r$, which is to say $x\not\in \cC_r(u)$.
\item \label{i3} We have the estimate $\sum r_i^{n-2} < C(n,\lambda)^d$.
\end{enumerate}
Ignoring the construction of the balls $B_{r_i}(x_i)$ for a moment, let us remark that we are done if we can always find such a collection.  Indeed, in this case we can then consider each ball $B_{r_i}(x_i)$ independently.  If $r_i=r$ we leave the ball alone, otherwise by rescaling $r_i\to 1$ and translating $x_i\to 0$ we now have a ball $B_1(0)$ such that $r^{d-1}_0\geq 1$, and hence we can find a $d-1$-covering as above for the new ball.  In particular, this means we can cover $B_{r_i}(x_i)$ by a collection of balls $B_{r_{ij}}(x_{ij})$ which satisfy the above conditions for $d-2$.  Summing over all $i$ and $j$ gives us a collection of balls $\{B_{r_{ij}}(x_{ij})\}\subseteq B_1(0)$ of our original ball which satisfy \eqref{i1} and \eqref{i2} above and for which
\begin{align}
\sum r_{ij}^{n-2}\leq C(n,\lambda)^{d}\cdot C(n,\lambda)^{d-1}\, .
\end{align}
 Continuing this $d$ times we arrive at a collection of balls $B_r(y_j)$ which satisfy \eqref{i2} and for which $\sum r^{n-2} \leq C^{\frac{1}{2}d(d-1)}\leq C^{\Lambda^2}$, which finishes the proof.

Hence, we are left with understanding the construction of the balls $\{B_{r_i}(x_i)\}\subseteq B_1(0)$ satisfying \eqref{i1},\eqref{i2},\eqref{i3} above under the assumption that $r_0^{d}\geq 1$.  Roughly, the construction proceeds as follows.  For every $x\in B_1(0)$ let us define
\begin{align}
r'_x \equiv \max\{r, r_x^{d-1}\}\, ,
\end{align}
where $r^{d-1}_x$ is the $(d-1)$-critical radius of Section \ref{ss:critical_radius}.  Let us separate $B_1(0)$ into subsets $S_1,S_2$ which are defined by
\begin{align}
&S_1\equiv\{x\in B_1(0): \not\exists\, y\in B_{10 r'_x}(x) \text{ s.t. } r'_y<10^{-2}r'_x\}\, \notag\\
&S_2\equiv B_1\setminus S_1\, .
\end{align}
We can interpret $S_1$ as the set of points with locally minimizing $d-1$-critical radii.  We let $\{B_{r_i}(x_i)\}$ be a Vitali subcovering of the collection $\{B_{r'_x}(x)\}_{x\in S_1}$.  Clearly the collection satisfies \eqref{i1}, and hence we need to show this collection of balls satisfies \eqref{i2},\eqref{i3}.

Now standard arguments as in \cite[theorem 2.8]{ChNaVa} give us, roughly, that for every $x_i$ and $s\in [r_i,1]$ that there exists a homogeneous harmonic polynomial $P_i^{(s)}$ of degree $d$ such that $T_{x,s}^u$ is close to $P_i^{(s)}$.  A key point is that the new effective argument discussed in Section \ref{ss:tangent_maps} will allow us to take $P_i^{(s)}\equiv P^d$ to be independent of both $i$ and $s$.  For the sake of the outline let us make the assumption that $P^d$ is $n-2$ symmetric, which is to say that $P^d$ depends on only two variables.  Up to some technical details this will turn out to be the important case, in that one can always handle the other cases by even simpler methods.  So in this case there is an $n-2$-plane $V\subseteq \dR^n$ such that if $x\not\in V$, then $P^d$ is not critical at $x$.  

There are two steps needed to complete the proof.  First, if $x\in S_2$, then by assumption there exists a point $x_i$ centering a ball in our covering which is not too far from $x$ relative to $r'_{x}$.  In particular, since $u$ is close to $P^d$ centered at $x_i$ this is roughly equivalent to the statement that $d(x,V)>r'_x\geq r$.  We have already mentioned that $P^d$ is therefore not critical at $x$, and with a little work, since $u$ is close to $P^d$, one can use this to show $u$ is not critical at $x$.  More effectively, we even have that $r_x\geq r$, which proves \eqref{i2} for the covering.  Making this precise will turn out to require an effective cone-splitting argument (see Sections \ref{ss:almost_conesplitting} and \ref{ss:almost_conesplitting_ell}).

Finally, let us consider the projection map $P^V:\dR^n\to V$.  Since $u$ is close to $P^d$ in all balls $B_{s}(x_i)$ with $s\in [r_i,1]$, one can use this to prove the projection map, when restricted to the centers of the balls $\{x_i\}$, is $(1+\epsilon)$-bi-Lipschitz.  Slightly more precisely, if $x_i,x_j\in \{x_k\}$ are two centers in the covering, then we know that the blow up of $u$ centered at $x_i$ at the radius $d_{ij}\equiv d(x_i,x_j)$ looks close to $P^d$.  In particular, since both $x_i$ and $x_j$ are 'good' points relative to the frequency pinching, by construction, we have that $x_i$ and $x_j$ must be close to the plane $V$ relative to $d_{ij}$.  Making this precise is exactly the statement that $P^V$ restricted to $\{x_i\}$ is $(1+\epsilon)$-bi-Lipschitz.  In particular, the Vitali covering $B_{r_i}(x_i)$ induces a Vitali covering $\{B_{r_i/2}(P^V(x_i))\}$ of the $n-2$-ball $B_1\cap V$.  Thus we get from this the estimate \eqref{i3}.

\clearpage

\section{Harmonic functions}\label{s:harmonic}
In this Section, we concentrate on harmonic functions in $\R^n$, and will first prove Theorem \ref{t:main_critical} in this simplified case.  This will allow us to illustrate the main ideas of the proof without the confusion of the added technical complications needed for the general case.  More than that, many of the tools we will need for general solutions of (\ref{eq_Lu}) will follow by appropriate approximation arguments with harmonic functions, and thus many of the results of this Section are directly relevant.

We start in section \ref{ss:hhp} by recalling some basic properties of homogeneous harmonic polynomials, hhP's in short. For a more complete overview on the subject, we refer the reader to \cite{HFT}.  In section \ref{ss:frequency_hhp} and \ref{ss:hpinch} we discuss the frequency function and its relation to homogeneous harmonic polynomials.  Although much of this is known, the estimates of these sections are much more refined than those previously in the literature, and we will need these results.  In particular in Theorem \ref{t:eff_tan_con_uniq_harm} we will prove an effective tangent cone uniqueness statement, which will play an important role in our estimates.  In section \ref{ss:almost_conesplitting} we revisit the idea of cone splitting, introduced in this context first in \cite{ChNaVa}.  The results of \cite{ChNaVa} are based on contradiction arguments, and we again prove much more refined estimates.  Sections \ref{ss:almost_codim2_invariant} and \ref{ss:symmetric_criticalpoints} discuss the relationship of critical points to the symmetry of a harmonic function.  Finally in section \ref{ss:volume_estimates_harmonic} we prove Theorem \ref{t:main_critical} for harmonic functions.

\subsection{Homogeneous harmonic polynomials}\label{ss:hhp}
Let $D\subset \R^n$ be any domain, and denote for convenience $\cH(D)$ the space of harmonic functions $u:D\to \R$, $u\in W^{1,2}(D)$. Most of the times, we will consider $B_1(0)$ as our domain, thus we define for simplicity $B=B_1(0)\subset \R^n$. We recall that a polynomial $P$ is said to be homogeneous of degree $d$ if $P(\lambda x)=\lambda^d P(x)$ for all $\lambda\in \R$ and $x\in \R^n$, or equivalently if $P$ is the sum of monomials with the same degree $d$. 
\begin{definition}
 Set $\cP_{d}$ to be the vector space of homogeneous harmonic polynomial of degree $d$ defined on $\R^n$. For $d\geq 2$ and $n\geq 3$, its dimension is given by
 \begin{gather}
  \operatorname{dim}(\cP_d) = \binom{n+d-1}{n-1}-\binom{n+d-3}{n-1} \leq C(n) d^{n-2}\, .
 \end{gather}
\end{definition}

By the standard theory of spherical harmonics (see for example \cite[chapter 5]{HFT}), one can characterize any such hhP by its restriction to the unit sphere $\partial B_1(0)$ and one finds that
\begin{gather}
 L^2(\partial B_1(0)) = \bigoplus_{d=0}^\infty \cP_d\, ,
\end{gather}
where $L^2(\partial B_1(0))$ is the real Hilbert space generated by the scalar product
\begin{gather}
 \ps{f}{g}= \fint_{\partial B_1(0)} fg\, .
\end{gather}
The space $\cH(B_1(0))$ inherits the Hilbert structure of $L^2(\partial B_1(0))$. Indeed, this product is well defined for all functions in $W^{1,2}(B_1(0))$, however only gives a Hilbert space structure on the harmonic functions as $\norm u =0 \Rightarrow u=0$ is true only on harmonic functions. 

Thus, we can write
\begin{gather}\label{eq_expu}
 u(y)= \sum_{d=0}^\infty a_d P_d(y)\, ,
\end{gather}
where each $P_d$ is a hhP of degree $d$ normalized with $\norm{P_d}=1$, and $a_d = \ps{u}{P_d}$. This expansion of course will depend on the base point chosen for the expansion. When needed, we will make this dependence explicit by writing
\begin{gather}
 u(y+x)=\sum_{d=0}^\infty a_d(x) P_{d,x} (y)\, .
\end{gather}

It is clear that if $P$ is a hhP of degree $d$, then $\partial_i P$ is either zero or a hhP of degree $d-1$. An important relation between the norm of a hhP and the norm of its gradient is given by the following lemma.
\begin{lemma}\label{lemma_pdp}\cite[lemma 5.13]{HFT}
Let $P,\ Q:\R^n\to \R$ be two homogeneous harmonic polynomials of degree $d$. Then
\begin{gather}
 \ps P Q = \frac{1}{d(2d+n-2)} \ps{\nabla P}{\nabla Q} =\frac{1}{d(2d+n-2)} \sum_{i=1}^n \ps{\partial_i P}{\partial_i Q}\, .
\end{gather}
\end{lemma}

Moreover, it is possible to prove a simple bound on the sup norm on $\partial B_1(0)$ of a hhP $P$ given its degree and its $L^2(\partial B_1(0))$ norm.
\begin{lemma}\label{lemma_dest}
 Given a $P_d\in \cP_d$, we have the sharp upper bound
\begin{gather}
 \norm{P_d}_{C^0(B_1(0))}\leq \sqrt{\operatorname{dim}(\cH_d)} \ton{\fint_{\partial B_1(0)} P_d^2 }^{1/2}\leq C(n) d^{\frac n 2 -1 }\ton{\fint_{\partial B_1(0)} P_d^2 }^{1/2}\, .
\end{gather}

\end{lemma}

\subsubsection{Two Variables} A special case which deserves to be studied on its own is the case of hhP's defined in $\R^2$.  The following is a standard but useful point:
\begin{proposition}
 Let $n=2$, then $\cH_{d,2}$ is a $2$ dimensional space for every $d\geq 2$, and an orthonormal basis for this space is given by
\begin{gather}
 P_d(r,\theta) = 2r^{d} \sin\ton{d\theta}\, , \quad C_d(r,\theta)= 2r^{d}\cos\ton{d\theta} = P_d(r,\theta+\pi/(2d))\, .
\end{gather}
Moreover, by direct computation one has
\begin{gather}\label{eq_dxP}
 \partial_x P_d(r,\theta) = 2d r^{d-1} \sin((d-1)\theta)\, , \quad \partial_y P_d(r,\theta) = 2d r^{d-1} \cos((d-1)\theta)\, ,\\
 \abs{\nabla P_d(r,\theta)}=\abs{\nabla C_d(r,\theta)} = 2d r^{d-1}\, \label{eq_nabla2vars} ,\\
 \frac{\partial^k}{\partial x^k} P_d(r,\theta)|_{(0,0)} = 2 \binom{d}{k} k! r^{d-k}\sin((d-k)\theta)= \binom d k k! P_{d-k}(r,\theta)\, ,\\
  P_d((t,0)+(r,\theta)) = P_d(r,\theta) + \sum_{k=1}^d \binom{d}{k} t^k P_{d-k}(r,\theta)\, .
\end{gather}
\end{proposition}

\subsubsection{Cone Splitting}  By simple algebra, it is easy to see that the set of points wrt which $P$ is homogeneous forms a vector subspace of $\R^n$. Indeed, let $P$ be homogeneous wrt $0$ and wrt $x\neq 0$, and pick any $y\in \R^n$. By Euler's formula
\begin{gather}
 d\cdot P(y) = \nabla P|_y \cdot y = \nabla P|_y \cdot (y-x)\quad \Longrightarrow \quad \nabla P|_y \cdot x =0\, ,
\end{gather}
and thus the partial derivative of $P$ in the $x$ direction vanishes at every point, making $P$ invariant wrt the subspace spanned by $x$. In other words, if $P$ is a harmonic polynomial of degree $d\geq 1$, the set
\begin{gather}
 V\equiv\cur{v\in \R^n \ \ s.t.  \ \ \nabla P\cdot v =0}
\end{gather}
is a vector subspace of $\R^n$, and it is the invariant subspace of $P$ in the sense that $P(x+v)=P(x)$ for all $x\in \R^n$ and $v\in V$. 

In the next proposition, we prove an extremely important (and simple) relation between the degree of $P$ and the dimension of $V$.
\begin{proposition}\label{prop_pn-2}
 Let $P$ be a nonconstant homogeneous harmonic polynomial. Then $P$ is of degree $1$ if and only if $V$ has dimension $n-1$.
\end{proposition}
\begin{proof}
 The direct implication is evident. As for the reverse, if $V$ has dimension $n-1$, then $P$ is a nonconstant homogeneous harmonic function of one variable, and thus it is linear.
\end{proof}

\subsubsection{Space of invariant polynomials}
\begin{definition}
 Given $x\in \R^n\setminus \{0\}$, we set $\cP_d(\cancel x)$ to be the subspace of $\cP_d$ of polynomials invariant with respect to the coordinate $x$. It is clear that this subspace has a uniquely defined orthogonal complement such that
 \begin{gather}
  \cP_d = \cP_d(\cancel x) \oplus \cP_d(\cancel x)^\perp\, ,
 \end{gather}
 where the direct sum is in the sense of $L^2(\partial B)$.
\end{definition}

In the next proposition, we characterize the elements of $\cP_d(\cancel x_1)^\perp$.
\begin{proposition}
 The linear function $K:\cP_{d-1}\to\cP_d$ defined by
 \begin{gather}
  K[p]\equiv x_1p - \frac 1 {2d+n-4} \abs x^2 \partial_1 p \equiv x_1p +c_{d,n} \abs x^2 \partial_1 p 
 \end{gather}
provides a vector space isomorphism between $\cP_{d-1}$ and $\cP_{d}(\cancel x_1)$. Moreover
\begin{gather}\label{eq_K-}
 \norm{K[p]}\leq \norm{x_1 p}_{L^2(\partial B_1(0)}\leq \norm{p}\, .
\end{gather}

\end{proposition}
\begin{remark}
 Note that $K$ is, up to multiplicative constants, the \textit{Kelvin transform} defined in \cite{HFT}.
\end{remark}
\begin{proof}
 Let $p\in \cP_{d-1}$, and let $q\in \cP_d(\cancel x_1)$. We will prove that $\ps{K[p]}{q}=0$ by induction on $d$. 
 
 It is clear that the statement is true for $d=1$.  By Lemma \ref{lemma_pdp} we have
 \begin{gather}
 d(2d+n-2) \ps{K[p]}{q}= \sum_{i=2}^n\ps{\partial_i \ton{x_1p +c \abs x^2 \partial_1 p} }{\partial_i q}=\sum_{i=2}^n \ps{x_1\partial_i p +2cx_i \partial_1p +c\abs x^2 \partial_1\partial_i p  }{\partial_i q}=\\
 =\sum_{i=2}^n \ps{x_1\partial_i p +c_{d-1,n}\abs x^2 \partial_1\partial_i p  }{\partial_i q} + \sum_{i=2}^n \ps{\ton{c_{d,n}-c_{d-1,n}}\abs x^2 \partial_1\partial_i p  }{\partial_i q} + 2c\sum_{i=2}^n \ps{\partial_1p}{x_i\partial_i q}=\\
 =\sum_{i=2}^n \ps{x_1\partial_i p +c_{d-1,n}\abs x^2 \partial_1\partial_i p  }{\partial_i q} + \ton{c_{d,n}-c_{d-1,n}}\sum_{i=2}^n \ps{ \partial_1\partial_i p  }{\partial_i q} + 2cd\sum_{i=2}^n \ps{\partial_1p}{q}\, .
 \end{gather}
The first sum is null by induction, while the second and third sum are null since they are scalar products of spherical harmonics (hhP's) of different degrees.  Hence we see that $K$ maps $\cP_{d-1}$ into $\cP_d(\cancel x_1)$.

Note now that $K[\cdot]$ in injective. Indeed, let $p\in \cP_{d-1}$ be such that $K[p]=0$. Then necessarily $\partial_1 p$ must be divisible by $x_1$, and thus $p$ is a harmonic polynomial proportional to $\abs x^2$, which is necessarily zero by \cite[corollary 5.3]{HFT}.

Surjectivity is easily proved by a dimension argument. Indeed
\begin{gather}
 \dim\ton{\cP_d} = \dim\ton{\cP_{d}(\cancel x_1)} + \dim\ton{\cP_{d-1}}\, .
\end{gather}

Finally, since $K[p]$ and $\partial_1 p$ are hhP's of different degrees,
\begin{gather}
 \ps{K[p]}{c_{d,n}\abs{x}^2 \partial_1 p }= c_{d,n}\ps{K[p]}{\partial_1 p }=0\, .
\end{gather}
This immediately implies the estimate on $\norm{K[p]}$.
\end{proof}

This characterization allows us to prove the following important property.
\begin{proposition}\label{prop_perpest}
 Let $h\in \cP_d(\cancel x_1)^\perp$, then
 \begin{gather}
  \norm{h}\leq \norm{\partial_1 h}\, .
 \end{gather}
\end{proposition}

\begin{proof}
If $d=1$, the proposition is easily proved by direct computation.

If $d\geq 2$, we write
\begin{gather}
 h = x_1 p + c_{d,n}\abs x ^2 \partial _1 p\, ,\\
 \partial_1 h = p + (2c_{d,n}+1) x_1 \partial_1 p + c_{d,n} \abs x ^2 \partial^2_{1^2} p\, .
\end{gather}
Note that for every $d,n\geq 2$, $2c_{d,n}+1\geq 0$. We estimate the norm of $\partial_1 h$ by
\begin{gather}
 \norm{\partial_1 h}^2 = \norm p ^2 + \norm{(2c_{d,n}+1) x_1 \partial_1 p + c_{d,n} \abs x ^2 \partial^2_{1^2} p}^2 + 2 \ps{p}{(2c_{d,n}+1) x_1 \partial_1 p + c_{d,n} \abs x ^2 \partial^2_{1^2} p}\geq\\
 \geq \norm p ^2 +  2 \ps{p}{(2c_{d,n}+1)x_1 \partial_1 p + c_{d,n}\abs x ^2 \partial^2_{1^2} p}=\norm p ^2 +  2 (2c_{d,n}+1) \ps{p}{x_1 \partial_1 p}\, .
\end{gather}
The last scalar product is nonnegative, as will be shown in the next Lemma (\ref{lemma_+}). This and equation \eqref{eq_K-} conclude the proof.
\end{proof}

\begin{lemma}\label{lemma_+}
 Let $p\in \cH_{n,d}$, then 
 \begin{gather}
  \ps{p}{x_1 \partial_1 p}_{L^2(\partial B_1(0)} = \frac{1}{2d+n-2}\norm{\partial_1 p}^2\, .
 \end{gather}
\end{lemma}
\begin{proof}
 We prove this proposition by induction on $d$. For $d=1$, the proposition is easily proved. Indeed, $p= \sum_i p_i x_i$ and by direct computation 
 \begin{gather}
  \ps{p}{x_1\partial_1 p} = \sum_i p_1p_i\fint_{\partial B_1(0)}x_1 x_i = p_1^2 \fint_{\partial B_1(0)} x_1 ^2 =\frac{1}{n} \norm{\partial_1 p}^2\, . 
 \end{gather}
Suppose by induction that the statement is true for $d-1$, and let $p\in \cH_{n,d}$. Note that the function $x_1 \partial_1 p$ is not harmonic, and its projection $\cP_d$ is
\begin{gather}
 x_1 \partial_1 p +c_{d,n} \abs x ^2 \partial^2_{1^2} p\, .
\end{gather}
Since $p$ and $\partial^2_{1^2} p$ are spherical harmonics of different degrees, 
\begin{gather}
 \ps{p}{c_{d,n} \abs x ^2 \partial^2_{1^2} p} =0\, .
\end{gather}
Thus, by Lemma \ref{lemma_pdp}, we can write
\begin{gather}
 \ps{p}{x_1 \partial_1 p } = \ps{p}{x_1 \partial_1 p + c_{d,n} \abs x ^2 \partial^2_{1^2} p}=\frac 1 {d(2d+n-2)}\ps{\nabla p}{\nabla \ton{x_1 \partial_1 p + c_{d,n} \abs x ^2 \partial^2_{1^2} p}}\, .
\end{gather}
On the right hand side we have
\begin{gather}
 \ps{\nabla p}{\nabla \ton{x_1 \partial_1 p + c_{d,n}\abs x ^2 \partial^2_{1^2} p}} = \norm{\partial_1 p}^2 + \ps{\nabla p}{x_1 \partial_1 \nabla p} + 2c_{d,n} \ps{x\cdot\nabla p}{\partial^2_{1^2} p} + c_{d,n} \ps{\nabla p}{\abs x ^2 \nabla \partial^2_{1^2} p}\, .
\end{gather}
The last two scalar products are null by the orthogonality of spherical harmonics of different degree. Induction and Lemma \ref{lemma_pdp} allow us to conclude
\begin{gather}
 \ps{\nabla p}{x_1 \partial_1 \nabla p} = (d-1)\frac{2d+n-4}{2d+n-4}=(d-1)\norm{\partial_1 p}^2\, .
\end{gather}
\end{proof}

We close this section with a consideration about invariant polynomials and their norm. Let $P:\R^n\to \R$ be a hhP of degree $d$, and suppose that $P$ is $x_1$ invariant. Then $P$ naturally induces a hhP $\hat P:\mathbb R^{n-1}\to\mathbb R$.  The following proposition gives the relation between the $L^2$ norms of $P$ and $\hat P$.  

\begin{lemma}\label{lemma_norm_n}
  Let $\hat P:\R^{n-1}\to \R$ be a hhP of degree $d$, and denote by $P:\R^{n}\to \R$ the polynomial $P(x_1,y)\equiv \hat P (y)$. Then
  \begin{gather}
   \norm{\hat P}^2 = \fint_{\partial B_1(0)\subset \R^{n-1}} \abs{\hat P}^2 = \norm{P}^2\prod_{k=1}^d \frac{n+2k-2}{n+2k-3}\leq \norm{P}^2_{n}e^2\sqrt{1+\frac{2d-2}{n-1}} \, .
  \end{gather}
 \end{lemma}
 \begin{proof}
  This lemma can be proved by direct computation, or as a corollary of \cite[Theorem 5.14]{HFT}.  As for the last estimate, it is easy to see that
  \begin{gather}
   \ln\ton{\prod_{k=1}^d \frac{n+2k-2}{n+2k-3}} \leq \ln\ton{1+\frac 1 {n-1}} + \sum_{k=2}^d \frac 1 {2k+n-3} \leq \ln\ton{2} + \frac 1 2 \ln\ton{1+\frac{2d-2}{n-1}}\, .
  \end{gather}
 
 \end{proof}

%

\subsection{Almgren's frequency and homogeneous harmonic polynomials}\label{ss:frequency_hhp}

In this subsection we recall the classic frequency function and some basic results about it.  Because we will focus on the critical set and not the singular set in our proofs, we will focus on the normalized frequency function:

\begin{definition}
 Given a nonconstant $u\in \cH(B_1(0))$, $x\in B_1(0)$ and $r\leq 1-\abs x$, Almgren's frequency is defined by
 \begin{gather*}
  N^u_{\cS}(x,r)= \frac{r \int_{B_r(x)} \abs{\nabla u}^2 dV }{\int_{\partial B_r(x)} \abs{u}^2 dS }\, .
 \end{gather*}
 This function is suitable for studying nodal and singular sets of harmonic functions. If we want to focus on the whole critical set of $u$, instead of just the singular set of $u$, a \textit{normalized} frequency is better suited for this job. For this reason, we set as in \cite{HLrank}
 \begin{gather}
  N^u_{\cC}(x,r)=N(x,r) = \frac{r \int_{B_r(x)} \abs{\nabla u}^2 dV }{\int_{\partial B_r(x)} \abs{u-u(x)}^2 dS }\, .
 \end{gather}
\end{definition}

\vspace{3mm}

Note that $N$ is invariant under blow-up, rescaling, and adding a constant.  In particular, if we define
\begin{gather}\label{eq_dephT}
 T_{x,r}^u(y) = T_{x,r}(y)= \frac{u(x+ry)-u(x)}{\ton{\fint_{\partial B} \abs{u(x+ry)-u(x)}^2  }^{1/2}}\, ,
\end{gather}
then $N^u_{\cC}(x,rs)=N^T_{\cC} (0,s)$.  The monotonicity of $N$ wrt $r$ is standard in literature (see for example \cite{hanlin}). Moreover, its proof is a simple matter of calculus.
\begin{proposition}
 For every $u\in \cH(B)$ and $x\in B$, $N(x,r)$ is monotone non decreasing wrt $r$. Moreover, if for some $0<r<s$, $N(x,r)=N(x,s)$, then $u$ is a harmonic polynomial homogeneous wrt $x$, and $N(x,t)$ is constant in $t$ and equal to the degree of the polynomial $u$.
\end{proposition}
Since $N$ is monotone, one can define $N(x,0)=\lim_{r\to 0} N(x,r)$ for all $x\in B$. As it is easily seen, $N(x,0)$ is the vanishing order of the function $u-u(x)$ at the point $x$. In particular
\begin{gather}
 1\leq N(x,0)=d \quad \Longleftrightarrow \quad \forall 1\leq k <d\, , \ \nabla ^{(k)} u|_x =0  \quad \text{  and  } \quad \nabla ^{(d)}u|_x\neq 0\, .
\end{gather}

\subsubsection{Polynomials and Frequency} \label{ss:freq_poly_exp}
Let $P$ be a harmonic polynomial of degree $d\geq 1$. Since $N$ is invariant under blow-up and rescaling, it is easy to see that $N(x,r)\leq d$ for all $x$ and $r$. Indeed, consider the function $P_{x,r}(y) = \frac{P(x+ry)}{\ton{\fint_{\partial B_1(0) } P(x+ry)^2  }^{1/2} } $. As $r\to \infty$, this function converges in the smooth sense to the normalized homogeneous component of $P$ with the highest degree, and thus, $\lim_{r\to \infty} N(x,r)=d$ for all $x$. As an easy corollary of the previous proposition, we get that $P$ is homogeneous wrt $x$ if and only if $N(x,0)=d$.

Sometimes it is convenient to exploit the polynomial expansion of $u$ given in \eqref{eq_expu} to express its frequency (see for example \cite[p 23]{hanlin}). Given that we can re-write
\begin{gather}
 \int_{B_r(0)}\abs{\nabla u}^2 = \int_{\partial B_r(0)} u \nabla_n u = \sum_{d=0}^\infty a_d^2 \int_{\partial B_r(0)} P_d \nabla P_d \cdot (r^{-1}x) = \sum_{d=0}^\infty r^{-1} d a_d^2\int_{\partial B_r(0)} P_d^2 \, ,
\end{gather}
we obtain
\begin{gather}
 N(0,r)= \frac{\sum_{d=0}^{\infty} d a_d^2 r^{2d} }{\sum_{d=0}^{\infty} a_d^2 r^{2d}}\, , \quad \quad N(0,r)= \frac{\sum_{d=1}^{\infty} d a_d^2 r^{2d} }{\sum_{d=1}^{\infty} a_d^2 r^{2d}}\, .
\end{gather}

\subsubsection{Growth estimates} Almgren's frequency can also be used to get growth estimates on the function $u$. Indeed, let
\begin{gather}
 h(x,r)\equiv \fint_{\partial B_r(x)} u^2 dx\, ,
\end{gather}
then by direct computation we get
\begin{gather}\label{eq_doth}
 \frac{d}{dt}\ln(h(x,t)) = \frac{2N(x,t)}{t}\, \quad \Longrightarrow \quad h(x,t) = h(x,r) \exp\ton{-2\int_t^r \frac{N(s)}{s}ds}\, .
\end{gather}

\subsubsection{Uniform control} An important property of the unnormalized frequency is that $N(x,r/2)\leq c(n,r) (N(0,1)+1)$ for $\abs x  \leq r <1$ (see for example \cite[theorem 2.2.8]{hanlin}). Thus, a bound on $N(0,1)$ implies an upper bound on the vanishing order of the function for all $x\in B$ away from the boundary. A similar statement with a similar proof holds also for the normalized version. For the sake of completeness, hereafter we sketch a proof of this result.

\begin{theorem}\label{th_cN}
 Let $u$ be a nonconstant harmonic function. Then for every $r,k<1$, there exists a constant $C(n,r,k)$ such that for all $\abs x \leq r$
 \begin{gather}
  N(x,k(1-r))\leq C N(0,1)\, .
 \end{gather}
\end{theorem}
\begin{proof}
We assume for simplicity that $u(0)=0$. First of all, we prove that there exists a radius $\beta(n)>0$ such that for all $\abs x \leq \beta(n)$,
 \begin{gather}\label{eq_alpha}
 u(x)^2 \leq \frac{1}{2} \fint_{\partial B_{1/2}(x)} u^2\, .
\end{gather}
We exploit the fact that $u$ vanishes with order at least $1$ at the origin. Suppose for convenience that $\fint_{\partial B_{2\beta} (0)} u^2 =1$. Since $N(0,2\beta)\geq 1$, we have
\begin{gather}
 \int_{B_{2\beta} (0)} u^2 \leq \int_0^{2\beta} ds \omega_n s^{n-1} \ton{\frac{s}{2\beta}}^2 = \frac{\omega_n}{n+2} (2\beta)^n\, .
\end{gather}
With a similar argument:
\begin{gather}
 \int_{B_{1/2-\beta} (0)} u^2 \geq \int_{2\beta}^{1/2-\beta} ds \omega_n s^{n-1} \ton{\frac{s}{2\beta}}^2 = \frac{\omega_n}{(2\beta)^2(n+2)} \qua{\ton{\frac 1 2 - \beta}^{n+2} -(2\beta)^{n+2}}\, .
\end{gather}
By geometric considerations, we have
\begin{gather}
 u(x)^2\leq \fint_{B_\beta (x)} u^2 = \frac{1}{\omega_n \beta^n} \int_{B_\beta (x)} u^2\leq \frac{1}{\omega_n \beta^n} \int_{B_{2\beta} (0)} u^2\leq \frac{2^n}{\omega_n} \alpha(n) \int_{B_{1/2-\beta} (0)} u^2\leq\\
 \leq \frac{2^n}{\omega_n} \alpha(n) \int_{B_{1/2}(x)} u^2= \alpha(n) \fint_{B_{1/2}(x)} u^2\leq  \alpha(n) \fint_{\partial B_{1/2}(x)} u^2 \, ,
\end{gather}
where we have set
\begin{gather}
 \alpha(n) \equiv \frac{(2\beta)^{n+2}}{(2\beta)^n \qua{\ton{\frac 1 2 - \beta}^{n+2} -(2\beta)^{n+2}}} = \qua{\ton{\frac 1 2 - \beta}^{n+2} -(2\beta)^{n+2}}^{-1} (2\beta)^2\, .
\end{gather}
It is evident that one can choose $\beta(n)$ sufficiently small in such a way that $\alpha(n)\leq 1/2$, which is all we need.
 
Now we are in a position to prove the theorem with $r=\beta(n)$ and $k$ generic. The general case is obtained by repeated applications of this estimate.

For simplicity of notation, we will assume $k=1/2$ and let $c(n)$ denote a constant that depends only on $n$, that may will change several times throughout the proof. Given the obvious inclusions
 \begin{gather}
  B_{1/4}(0)\subset B_{1/2}(x)\subset B_{3/4}( x)\subset B_1(0)\, ,
 \end{gather}
and the growth estimates
\begin{gather}
 \fint_{\partial B_1(0)} u^2 \leq 4^{2N(0,1)}\fint_{\partial B_{1/4}(0)} u^2 \quad \Longrightarrow \quad \fint_{B_1(0)} u^2 \leq c(n)4^{2N(0,1)}\fint_{B_{1/4}(0)} u^2\, ,
\end{gather}
one has
\begin{gather}
 \fint_{B_{3/4}( x)} u^2 \leq c(n)4^{2N(0,1)}\fint_{B_{1/2}(x)} u^2 \, .
\end{gather}
Since $\fint_{\partial B_r(0)} u^2$ is an increasing function of $r$ (if $u$ is harmonic), we can estimate
\begin{gather}
 \int_{B_{3/4}( x)} u^2 \geq \int_{B_{3/4}( x)\setminus B_{5/8 }(x )} u^2 \geq c(n) \fint_{\partial B_{5/8 }(x )} u^2\, ,\\
 \int_{B_{1/2}(x)}u^2 \leq c(n) \fint_{\partial B_{1/2}(x)} u^2\, .
\end{gather}
Thus we obtain that
\begin{gather}\label{eq_1}
 \fint_{\partial B_{5/8 }(x )} u^2 \leq c(n) 4^{2N(0,1)} \fint_{\partial B_{1/2}(x)} u^2\, .
\end{gather}
With some easy computations as in \cite{hanlin}, this implies $N_{\cS}(x,1/2)\leq C(n) (N(0,1)+1)$. In order to obtain a similar estimate for $N=N_{\cC}$, consider that
\begin{gather}
 u(x)^2 = \lim_{r\to 0} \fint_{\partial B_r(x)} u^2 = \gamma\fint_{\partial B_{5/8 }(x )} u^2\, ,
\end{gather}
where we set for convenience $\gamma = \exp\ton{-2\int_0^{5/8} \frac{N(x,s)}{s} ds }$. Note that, if $u$ is not constant, $0\leq \gamma<1$. Using \eqref{eq_1}, we obtain
\begin{gather}
\fint_{\partial B_{5/8 }(x )} u^2-u(x)^2 \leq c(n) 4^{2N(0,1)} \ton{1-\gamma}\fint_{\partial B_{1/2}(x)} u^2 = c(n) 4^{2N(0,1)} \ton{\fint_{\partial B_{1/2}(x)} u^2 - \gamma' u(x)^2}\, ,
\end{gather}
where $\gamma'= \exp\ton{-2\int_{1/2}^{5/8} \frac{N(x,s)}{s} ds }$, and again $0<\gamma'<1$.
Given Equation \eqref{eq_alpha}, we can estimate
\begin{gather}\label{eq_2}
\fint_{\partial B_{5/8 }(x )} u^2-u(x)^2 \leq 2c(n) 4^{2N(0,1)} \ton{\fint_{\partial B_{1/2}(x)} u^2 - u(x)^2}\, ,
\end{gather}
By the growth conditions related to $ N$, we obtain the inequality
\begin{gather}
 \ton{\frac 5 4}^{2 N(x,1/2)}\leq \exp\ton{2\int_{1/2}^{5/8} \frac{ N(x,s)} s ds} = \frac{\fint_{\partial B_{5/8 }(x )} u^2-u(x)^2 }{\fint_{\partial B_{1/2}(x)} u^2 - u(x)^2} \leq 2c(n) 4^{2N(0,1)}\, .
\end{gather}
By taking logs on both sides, we complete our estimate. Indeed we obtain
\begin{gather}
  N(x,1/2)\leq c(n) (N(0,1) +1) = c(n) ( N(0,1) +1) \leq c(n)  N(0,1)\, .
\end{gather}

\end{proof}

\subsection{Frequency pinching for harmonic functions}\label{ss:hpinch}
In the previous section, we have seen that if $N$ is constant, then the function $u$ is homogeneous. The aim of this section is to prove a quantitative ``almost'' version of this statement, with particular care on how the parameters involved depend on the frequency $N(0,1)$.  The results of this Section may be viewed as a generalization of the quantitative pinching in \cite[theorem 2.8]{ChNaVa}.

%
\begin{definition}
Given a nonconstant harmonic function $u$, we say that its frequency is $\delta$-pinched at $x$ on the scales $[r_2,r_1]$ if
\begin{gather}\label{eq_pinch1}
N(x,r_1)- N(x,r_2)\leq \delta\, .
\end{gather}
\end{definition}
As we have seen, if $\delta=0$, $u$ is, up to an additive constant, a homogeneous harmonic polynomial of degree $d$ and $N(x,r)=d$ for all $r$. Using a simple compactness argument (see \cite[theorem 2.8]{ChNaVa}), one can prove that if $\delta$ is small enough, then $u$ is close to a homogeneous harmonic polynomial. In particular, for every $\epsilon>0$ there exists $\delta(n,\Lambda,\epsilon)>0$ such that if \eqref{eq_pinch1} is satisfied, then for all $r\in [2r_2,r_1]$ there exists a hhP $P^{(r)}$ of degree $d$ such that
\begin{gather}
\norm{T^u_{x,r}-P^{(r)}}\leq \epsilon\quad \text{and}\, , \quad \abs{N^u(x,r)-d}\leq \epsilon\, .
\end{gather}

By exploiting some improved monotonicity properties of $N$, we make the previous argument effective. First of all, we prove that $N(r)$ can be pinched only when it is close to an integer. 
\begin{lemma}\label{lemma_Npinch}
 Let $\min_{k\in \N} \abs{N(r)-k}=\epsilon>0$. Then 
 \begin{gather}
  r \left.\frac{d N}{dt} \right\vert_{t=r} \geq 2\epsilon(1-\epsilon)\, .
 \end{gather}
As a corollary, if $N(r)\leq d-\epsilon$, then
\begin{gather}
 N\ton{\frac{\epsilon}{1-\epsilon} r}\leq d-1+\epsilon\, .
\end{gather}
\end{lemma}
\begin{proof}
By the scale invariance properties of $N$, we can assume for simplicity $r=1$ and $h(1)=1$. Let $d$ be the integral part of $ N(1)$, i.e., the largest integer $\leq  N(1)$. By hypothesis, $d\leq N(1)-\epsilon$.

 Define the following functions.
 \begin{gather}\label{eq_hpm}
 h_+(t) = \sum_{k\geq d+1} a_k^2 t^{2k}\, , \quad \quad h_-(t) = \sum_{k\leq d-1} a_k^2 t^{2k}\, ;\\
 N_+ (t) = \frac{\sum_{k\geq d+1} k a_k ^2 t^{2k}}{h_+(t)}\, , \quad \quad N_- (t) = \frac{\sum_{k\leq d-1} k a_k ^2 t^{2k}}{h_-(t)}\, ;\label{eq_Npm}\\
 f_+(t) = \frac{h_+(t)}{t^{2d+2}}\ton{N_+(t)-d}\, , \quad \quad f_-(t) = \frac{h_-(t)}{t^{2d+2}}\ton{d-N_-(t)}\, .
 \end{gather}
Note that $f_+(t)>0$ for $t>0$, with $\lim_{t\to 0} f_+(t)=0$. As for $f_-$, it is either a strictly positive function or it is zero. In the first case $\lim_{t\to 0} f_-(t)=\infty$. The derivatives of $f_\pm$ are easily computed directly. Indeed, we obtain
\begin{gather}
 \dot f_+(t) = 2\sum_{k\geq d+1}  (k-d-1)(k-d) a_k^2 t^{2(k-d-1)-1}\geq 0 \, ,\\
  \dot f_-(t) = 2\sum_{k\leq d-1}  (k-d-1)(d-k) a_k^2 t^{2(k-d-1)-1}\leq 0 \, .
\end{gather}

We rewrite the frequency $N(t)$ in the following convenient way.
\begin{gather}\label{eq_Nhpm}
  N(t) = \frac{N_-(t) h_-(t) + d a_d^2 t^{2d} + N_+(t) h_+(t) }{h(t)} = \frac{N_-(t) h_-(t) + d \qua{h(t)-h_-(t)-h_+(t)}+ N_+(t) h_+(t) }{h(t)}=\\
 \notag= \frac{h_-(t)}{h(t)} \ton{N_-(t)-d} + d + \frac{h_+(t)}{h(t)} \ton{N_+(t)-d}  = -\frac{h_-(t)}{h(t)} \ton{d-N_-(t)} + d + \frac{h_+(t)}{h(t)} \ton{N_+(t)-d}      \, .
\end{gather}
In particular, we obtain the simple formula
\begin{gather}
N(t)-d = \frac{t^{2d+2}}{h(t)}\ton{f_+(t) - f_-(t)}\, .
\end{gather}
By \eqref{eq_doth}, we obtain the equality
\begin{gather}
\frac{h(t)}{t^{2d+2}} = h(1)\exp\ton{- 2\int_t^1 \frac{ N(s)-d-1}{s}ds}\, .
\end{gather}
This and the fact that $\frac{d}{dt} \ton{f_+(t)-f_-(t)}\geq 0$ imply that
\begin{gather}
 0\leq \frac{d}{dt} \qua{\ton{N(t)-d} \exp\ton{- 2\int_t^1 \frac{ N(s)-d-1}{s}ds} }= \\
 \notag = \exp\ton{- 2\int_t^1 \frac{ N(s)-d-1}{s}ds} \qua{\frac d {dt} { N(t)} + 2(N(t)-d)\frac{ N(t)-d-1}{t}}\, .
\end{gather}
As a consequence we obtain
\begin{gather}\label{eq_dotN}
 \frac{d}{dt} N(t) \geq 2( N(t)-d)\frac{d+1- N(t)}{t}\, .
\end{gather}
As long as $ N(t)\in (d,d+1)$, the rhs is positive. Define for convenience $\rho(t)=\ln(t)$, and $t=e^{\rho}$. We have
\begin{gather}\label{eq_dNdrho}
 \frac {d}{d\rho} N \geq 2( N(\rho)-d)\ton{d+1-N(\rho)}\, .
\end{gather}
Let $\hat N(\rho)$ be the solution of the corresponding differential equality, i.e.,
\begin{gather}
 \hat N(\rho) = d + \frac{1}{c e^{-2\rho} +1} \, ,
\end{gather}
where $c>0$ is chosen in such a way that $\hat N(0)\geq N(\rho=0)$. Since $N(\rho=0)\leq d+1-\epsilon$, we can pick
\begin{gather}
 c=\frac{\epsilon}{1-\epsilon}\, .
\end{gather}
A standard comparison for ODE implies that $N(\rho)\leq \hat N(\rho)$ for all $\rho\leq 0$. Thus if $\bar \rho$ satisfies
\begin{gather}
 \frac{1}{c e^{-2\bar \rho}+1}\leq \epsilon\, \quad \Longrightarrow \quad \bar \rho \leq \log\ton{\frac{\epsilon}{1-\epsilon}}\, ,
\end{gather}
then $N(\bar \rho)\leq d+\epsilon$. This concludes the proof.
%
%
\end{proof}
As a corollary of the proof, we obtain the following
\begin{corollary}\label{cor_Npinch}
 Let $\operatorname{dist}\ton{N(r),\N}=2\epsilon >0$. Then $N(r)-N(r/e)\geq \epsilon$.
\end{corollary}
\begin{proof}
 Note that by definition $0<\epsilon<1/2$. By \eqref{eq_dNdrho}, as long as $\operatorname{dist}\ton{N(\rho),\N}\geq\epsilon$ we have the lower bound
 \begin{gather}
  \frac {d}{d\rho} N \geq 2\epsilon\, .
 \end{gather}
This and the monotonicity of $N$ immediately imply the thesis.
\end{proof}

Using a similar technique, we can prove that either $u$ is close in the $L^2$ sense to a homogeneous harmonic polynomial $P$ at a certain scale, or the frequency drops by a definite amount after some definite number of scales. 
\begin{lemma}\label{lemma_polypinch}
 Given a harmonic function $u:B_r(0)\to \R$, for every $\epsilon>0$ one of these two things can happen
\begin{enumerate}
 \item either there exists $d$ such that $a_d^2 r^{2d} \geq \ton{1-6\epsilon}h(r)$;
 \item or $N(0,r)-N(0,r/e) \geq \epsilon$.
\end{enumerate}
\end{lemma}

\begin{proof}
Suppose without loss of generality that $r=1$ and $h(1)=1$. Fix an index $d$, and define $h_\pm$ and $N_\pm$ as in \eqref{eq_hpm}, \eqref{eq_Npm}. By analogy with the usual frequency, both $N_\pm$ are monotone nondecreasing functions. Moreover, it is easily seen that
\begin{gather}
 N_+(t)-d\geq 1 \, \quad \quad d-N_-(t)\geq 1\, .
\end{gather}
Simple considerations on the definitions of $h$ and $h_\pm$ imply that
\begin{enumerate}
 \item if $h_+$ is not identically zero, $\frac{h_+(t)}{h(t)}$ is increasing with respect to $t$, and, if $h_-$ is not identically zero, it has limit $0$ for $t\to 0$,
 \item if $h_-$ is not identically zero, $\frac{h_-(t)}{h(t)}$ is decreasing with respect to $t$, and it has limit $1$ for $t\to 0$.
\end{enumerate}

By \eqref{eq_Nhpm} we have
\begin{gather}
 N(t) = -\frac{h_-(t)}{h(t)} \ton{d-N_-(t)} + d + \frac{h_+(t)}{h(t)} \ton{N_+(t)-d}      \, .
\end{gather}

If $a_k^2 \leq 1-6\epsilon$ for all $k$, then there exists an index $d$ such that either $h_+(1)\in [3\epsilon,1/2]$ or $h_-(1) \in [3\epsilon,1/2]$. Suppose that the first case is true, with a similar argument one can deal also with the second case.

By monotonicity of $N_-$, positivity of $d-N_-(t)$ and since $h_-(t)/h(t)$ is a decreasing function of $t$, we have for $t\leq 1$
\begin{gather}
 -\frac{h_-(1)}{h(1)} \ton{d-N_-(1)} + \frac{h_-(t)}{h(t)} \ton{d-N_-(t)} \geq 0\, .
\end{gather}
Thus
\begin{gather}
  N(1)-N(t)\geq  \frac{h_+(1)}{h(1)}(N_+(1)-d) -\frac{h_+(t)}{h(t)}(N_+(t)-d) \, .
\end{gather}
Note that for all $t\leq 1$:
\begin{gather}
1\leq N_+(t)-d \leq N_+(1)-d\, ,\\
\frac{h_+(t)}{h(t)} = \frac{t^{-2d}h_+(t)}{t^{-2d}h_-(t) + a_d^2 + t^{-2d}h_+(t)} \leq \frac{t^2 h_+(1)}{t^{-2d}h_-(t) + a_d^2 } \leq \frac{t^2 h_+(1)}{t^{-2}h_-(1) + a_d^2 } \leq 2t^2 h_+(1)\, ,
\end{gather}
where the last inequality follows from the assumptions $h(1)=1$ and $h_+(1)\leq 1/2$.

Thus we obtain:
\begin{gather}
 N(1)-N(t)\geq  h_+(1) \qua{(N_+(1)-d) -2t^2(N_+(t)-d)} \geq 3\epsilon (N_+(t)-d)(1-2t^2)\geq 3\epsilon (1-2t^2)\, .
\end{gather}
If we choose $t=e^{-1}$, we obtain
\begin{gather}
 N(1)- N(1/e)>  2\epsilon\, .
\end{gather}

In case $h_-(1)\in [3\epsilon,1/2]$, a similar computation holds. Indeed, by monotonicity of $N_+$, positivity of $N_+(t)-d$ and since $h_+(t)/h(t)$ is an increasing function of $t$, we have for $t\leq 1$
\begin{gather}
 \frac{h_+(1)}{h(1)} \ton{N_+(1)-d} - \frac{h_+(t)}{h(t)} \ton{N_+(t)-d} \geq 0\, .
\end{gather}
Thus
\begin{gather}
 N(1)- N(t)\geq  -\frac{h_-(1)}{h(1)}(d-N_-(1)) +\frac{h_-(t)}{h(t)}(d-N_-(t)) \, .
\end{gather}
Note that for all $t\leq 1$:
\begin{gather}
1\leq d-N_-(1)\leq d-N_-(t)\, ,\\
1\geq\frac{h_-(t)}{h(t)} = \frac{t^{-2d}h_-(t)}{t^{-2d}h_-(t) + a_d^2 + t^{-2d} h_+(t)} \geq \frac{t^{-2d} h_-(t)}{t^{-2d}h_-(t) + a_d^2 +t^2 h_+(1)} \geq \frac{t^{-2d} h_-(t)}{t^{-2d}h_-(t) + 1-h_-(1)} \, .
\end{gather}
Since the function $x/(1+x)$ is an increasing function for $x\geq 0$, and since $t^{-2d} h_-(t) \geq t^{-2} h_-(1)$, we have
\begin{gather}
\frac{h_-(t)}{h(t)} \geq \frac{t^{-2} h_-(1)}{1 + (t^{-2})h_-(1)} \, .
\end{gather}
Since $h_-(1)\leq 1/2$, for $t=e^{-1}$ we obtain
\begin{gather}
\frac{h_-(t)}{h(t)} >  \frac 4 3 h_-(1)\, .
\end{gather}

Thus we obtain:
\begin{gather}
 N(1)-N(1/e)> h_-(1) \qua{-(d-N_-(1)) +\frac 4 3 (d-N_-(1/2))}\geq \frac 1 3 \epsilon (d-N_-(1/2))\geq\epsilon \, .
\end{gather}
This concludes the proof.
\end{proof}
As a consequence of Lemma \ref{lemma_Npinch} and \ref{lemma_polypinch}, we see that if the nonnegative quantity $N(r)-N(r/e)$ is sufficiently small, then the function $u$ is close to a homogeneous harmonic polynomial at scale $r$, and $N$ is close to an integer.  We now generalize this point to our {\it effective} tangent cone uniqueness statement for harmonic functions, which is our main result for this subsection. 

\begin{theorem}\label{t:eff_tan_con_uniq_harm}
 Let $u:B_{r_1 }(0 )\to \R$ be harmonic, and assume that $\big|N(0,r_1)-N(0,r_2)\big| \leq \epsilon$ with $r_2\leq r_1/e^3$. There exists an absolute constant $\epsilon_0>0$ such that if $\epsilon \leq \epsilon_0$, then
 \begin{enumerate}
  \item[(i)] There exists an integer $d$ such that for all $t\in (r_2,r_1)$, $\abs{ N(t)-d}\leq 3\epsilon$,
  \item[(ii)] For all $t\in (er_2,r_1/e)$ we have $a_d^2 t^{2d} \geq (1-6\epsilon) h(t)$,
  \item[(iii)] For all $t\in (er_2,r_1/e)$ we have that $u$ is close in the $L^2$ sense to the homogeneous harmonic polynomial $P_d$. More precisely, for all $t\in (e r_2,r_1/e)$,
   \begin{gather}
    \fint_{\partial B_1(0)} \abs{T_{0,t}^u - P_d}^2 \leq 7\epsilon\, ,
   \end{gather}
  \item[(iv)] up to a factor $d$, $u$ and $P_d$ are also $W^{1,2}$ close. More precisely, for all $t\in (e r_2,r_1/e)$,
     \begin{gather}
    \int_{B_1(0)} \abs{\nabla T_{0,t}^u - \nabla P_d}^2\leq 7d\epsilon\, .
   \end{gather}
 \end{enumerate}
\end{theorem}
\begin{remark}
The key aspect of this Theorem is the sharpness of the closeness of $u$ to $P_d$ depending on the frequency drop.  That is, after dropping one scale either the frequency drops by $\epsilon$ or $u$ is $\epsilon$ close to a homogeneous harmonic polynomial, where $\epsilon$ is independent of $d$, compare for instance to \cite{ChNaVa}.
\end{remark}
\begin{remark}
The second key aspect of this Theorem is that if $u$ is pinched on many scales, then $u$ is automatically close to the {\it same} homogeneous harmonic polynomial on all scales.  This is a key point to the proof of the main Theorem.
\end{remark}

\begin{proof}
Let us begin with the observation that if we prove the Theorem for $e^3r_2 = r_1 \equiv r$, then the result is proved for any $r_1,r_2$.  Indeed, since we are proving that $u$ is close to the $d^{th}$-order part of its Taylor expansion, the radii are unimportant, and thus we will make this assumption in the rest of the proof.

 (i) is a direct consequence of Corollary \ref{cor_Npinch}. By Lemma \ref{lemma_polypinch} (ii) is valid with $d$ replaced by another integer $q$ which, a priori, might be different from $d$. We are left to prove that $q=d$. In order to do so, we will prove that $\abs{N(r/e)-d}$ cannot be small if $d\neq q$.
 
 For simplicity, we assume that $r=e$ and $h(1)=1$. A simple computation yields
 \begin{gather}\label{eq_N-d}
  N(1)-d = \sum_{k\neq d} (k-d) a_k^2 = (d-q) a_q^2 + \sum_{k\neq d,q} (k-d) a_k^2 \, .
 \end{gather}
By Cauchy inequality, we estimate the last sum as follows
\begin{gather}
\abs{ \sum_{k\neq q,d} (k-d) a_k^2} \leq \sum_{k\leq d-1, \ k\neq q} (d-k) a_k^2+\sum_{k\geq d+1, \ k\neq q} (k-d) a_k^2\leq \\
\leq \ton{\sum_{k\leq d-1, \ k\neq q} (d-k)^2 a_k^2}^{1/2}\ton{\sum_{k\leq d-1, \ k\neq q} a_k^2}^{1/2}+\ton{\sum_{k\geq d+1, \ k\neq q} (d-k)^2 a_k^2}^{1/2}\ton{\sum_{k\geq d+1, \ k\neq q} a_k^2}^{1/2}\leq \\
\leq \sqrt{6\epsilon} \qua{\ton{\sum_{k\leq d-1, \ k\neq q} (d-k)^2 a_k^2}^{1/2} + \ton{\sum_{k\geq d+1, \ k\neq q} (d-k)^2 a_k^2}^{1/2}}\, .
\end{gather}
In order to estimate the sums with $(d-k)^2$, we exploit the growth conditions on $h(t)$. Recall that, for all $t\in (e^{-1},e)$, $\sum_{k\neq q} a_k^2 t^{2k} \leq 6\epsilon h(t)$. Moreover, by (i) and \eqref{eq_doth}, we can estimate $h(t)$ by
\begin{gather}
 h(e)\leq e^{2d+6\epsilon}\leq e^{2d+1}\, \quad \Longrightarrow \quad \sum_{k\neq q} a_k^2 e^{2k-2d-1}\leq 6\epsilon         \, , \\
 h(e^{-1}) \leq e^{-2d+6\epsilon}\leq e^{-2d+1 }\, \quad \Longrightarrow \quad \sum_{k\neq q} a_k^2 e^{-2k+2d-1}\leq 6\epsilon   \, .
\end{gather}
It is evident that there exists a universal constant $C$ such that
\begin{gather}
 (k-d)^2 \leq C\begin{cases}
               e^{2k-2d-1} & \text{ for } k\geq d+1\, ,\\
               e^{-2k+2d-1}& \text{ for } k\leq d-1\, .
              \end{cases}
\end{gather}
Thus we obtain
\begin{gather}
 \sum_{k\geq d+1, k\ \neq q} (k-d)^2 a_k^2  \leq C \sum_{k\neq q} a_k^2 e^{2k-2d-1} \leq 6 C\epsilon \, ,\\
 \sum_{k\leq d-1, k\ \neq q} (k-d)^2 a_k^2  \leq C \sum_{k\neq q} a_k^2 e^{-2k+2d-1} \leq 6 C\epsilon \, .
\end{gather}
By \eqref{eq_N-d} and the triangle inequality, we obtain
\begin{gather}
 \abs{N(1)-d} \geq \abs{q-d} (1-6\epsilon) - 12 \epsilon \sqrt C\, ,
\end{gather}
and the conclusion follows immediately by (i) and the fact that $q$ and $d$ are both integers.

(iii) is a simple corollary of (ii). Indeed, for all $t\in (r/e^2,r)$
 \begin{gather}
  \fint_{\partial B_1(0)} \abs{T_{0,t}^u - \tilde a_d P_d}^2 \leq 6 \epsilon\, ,
 \end{gather}
and $\tilde a_d \geq \sqrt{1-6\epsilon}\geq 1-4\epsilon$ for $\epsilon\leq \epsilon_0$. This and the normalization of $P_k$ imply that 
 \begin{gather}
  \ton{\fint_{\partial B_1(0)} \abs{T_{0,t}^u - P_d}^2 }^{1/2}\leq \sqrt{6 \epsilon} + 4\epsilon\leq \sqrt{7\epsilon}\, .
 \end{gather}

 The $W^{1,2}$ estimates in $(iv)$ are an easy consequence of the harmonicity of $u$ and $P$. Indeed, we have
\begin{gather}
 \int_{B_1(0)} \abs{\nabla T_{0,t}^u - \nabla P_d}^2 = \int_{B_1(0)} \abs{\nabla T_{0,t}^u}^2 + \int_{B_1(0)} \abs{\nabla P}^2 - 2 \int_{B_1(0)} \ps{\nabla T_{0,t}^u}{\nabla P} =\\
 = N(t) +d - 2 \int_{\partial B_1(0)} T_{0,t}^u\nabla _n P\, .
\end{gather}
By homogeneity, $\nabla_n P (x)= d\abs{x}^{-1} P(x)$, thus
\begin{gather}
 \int_{B_1(0)} \abs{\nabla T_{0,t}^u - \nabla P_d}^2 \leq 2d + 3\epsilon - 2d\tilde a_d \leq 3\epsilon + 4d\epsilon\, .
\end{gather}

 \end{proof}

As it is clear from the proofs of this subsection, the same results proved here are valid also if we replace $N$ with the unnormalized frequency, the only difference is that in this second case the unnormalized frequency is bounded below by $0$, not $1$, and the integers appearing in the propositions can take the value $0$.

\subsection{Almost cone splitting}\label{ss:almost_conesplitting}
As we have seen before, a cone splitting theorem is valid for hhP's. In particular, if $P$ is homogeneous wrt $0$ and $x$, then $P$ is invariant wrt the line defined by $x$, and thus $\partial_x P=0$. An almost cone splitting holds for generic harmonic functions, where homogeneity is replaced by quantitative frequency pinching.

\begin{lemma}\label{lemma_epsd}
Let $u:B_{e^2 d+1}(0)\to \R$ be a harmonic function, and let $d\geq 1$ be an integer such that
 \begin{itemize}
  \item $N(0,e^2 d)-N(0,e^{-1})\leq \epsilon$ with $\abs{N(0,1)-d}\leq \epsilon$ as well,
  \item there exists $\bar x \in B_1(0)$ such that $N(\bar x,e^2 d)-N(\bar x,e^{-1})\leq \epsilon$ with $\abs{N(\bar x,1)-d}\leq \epsilon$.
 \end{itemize}
 After rotating we may assume without loss that $\bar x = (t,0,\cdots,0)$. If $\epsilon \leq \epsilon_0(n)$, then $u$ is almost $\bar x$ invariant, in the sense that:
 \begin{enumerate}
  \item The $d$-th degree part in its expansion is almost constant. In particular             	\begin{gather}
           \fint_{\partial B_1(0)} \abs{a_d(0) P_{d,0}(y) - a_d(\bar x) P_{d,\bar x}(y) }^2 dy \leq C(n)\epsilon t^2 \fint_{\partial B_1(0)} \abs{a_d(0) P_{d,0}(y) }^2 dy  =C(n)\epsilon t^2 a_d(0)^2\, ;
	  \end{gather}
  \item The function itself is almost invariant under translation with respect to $\bar x$. In particular
	  \begin{gather}
           \fint_{\partial B_1(0)} \abs{u(y) - u(\bar x +y) }^2 dy \leq C(n)\epsilon t^2\fint_{\partial B_1(0)} \abs{u(y)}^2 dy \, ;
	  \end{gather}
  \item The $\bar x$ derivative of $P_d$ is almost zero, more precisely
	   \begin{gather}
	    \norm{\partial_1 P_{d,0}}\leq C(n) t^{-1} \sqrt \epsilon \norm{\nabla P_{d,0}}=C(n) t^{-1} \sqrt \epsilon \sqrt{d(2d+n-2)}\norm{P_{d,0}}\, .
	   \end{gather}

 \end{enumerate}
\end{lemma}

\begin{proof} Let $u= \sum_k a_k P_k$ be the expansion at zero. By the pinching conditions and Theorem \ref{t:eff_tan_con_uniq_harm}, we know that for all $s\in [e^{-1},ed]$ we have
 \begin{gather}\label{eq_estepsd}
  \sum_{k\neq d} a_k^2 s^{2k}\leq \epsilon a_d^2 s^{2d}\quad \Longrightarrow \quad \sum_{k\geq d+1} a_k^2 (ed)^{2(k-d)}\leq \epsilon a_d^2\quad \Longrightarrow \quad \forall k\geq d+1\, , \, \,  a_k^2 \leq \epsilon a_d^2 (ed)^{2(d-k)}\, .
 \end{gather}
In order to compute the expansion $u$ at $\bar x$, we expand all the polynomials $P_k$ using Taylor's formula. 
\begin{gather}
 P_k(x+\bar x) = P_k(x) + \sum_{ i=1}^k \frac{t^i}{i!} \ton{\partial _1}^i P_k\, ,
\end{gather}
where $\ton{\partial _1}^i P_k$ is again a homogeneous harmonic polynomial of degree $k-i$. By an iterated use of Lemma \ref{lemma_pdp}, we can estimate 
\begin{gather}\label{eq_estdip}
 \frac{\norm{\ton{\partial _1}^i P_k}^2}{\norm{P_k}^2}\leq \binom k i i! \prod_{j=0}^{i-1} (2(k-j)+n-2)=\qua{\binom k i \ i!}^2 2^i  \prod_{j=0}^{i-1} \ton{1+\frac{n-2}{2(k-j)}}\leq c \qua{\binom k i \ i!}^2 2^i  \ton{\frac{k}{k-i+1}}^{n/2}\, .
\end{gather}

Now, when we re-expand, we obtain
\begin{gather}
 u(x+\bar x)=\sum_k a_k P_k(x+\bar x)\, .
\end{gather}
The degree $d$ part in this expansion is
\begin{gather}
a_d(\bar x) P_{d,\bar x} = a_d P_d + \sum_{k=1}^\infty \frac{t^k}{k!} a_{d+k}\ton{\partial _1}^k P_{d+k}\, .
\end{gather}
By \eqref{eq_estepsd} and \eqref{eq_estdip}
\begin{gather}
 a_d^{-1}\norm{\sum_{k=1}^\infty \frac{t^k}{k!} a_{d+k}\ton{\partial _1}^k P_{d+k}}\leq \sqrt \epsilon \sum_{k=1}^\infty t^k \binom{d+k}{d} d^{-k} e^{-k}\ton{1+\frac{k}{d}}^{n/2}\, .
\end{gather}
Simple and very rough algebraic manipulations lead to
\begin{gather}
 a_d^{-1}\norm{a_dP_d-a_d(\bar x) P_{d,\bar x}}\leq \sqrt \epsilon t \sum_{k=1}^\infty \frac{(d+k)(d+k-1)\cdots (d+1)}{d^k k!} e^{-k}\ton{1+\frac{k}{d}}^{n/2}\leq C(n)\sqrt\epsilon t\, ,
\end{gather}
which concludes the proof of point $1$.

\paragraph{Proof of (2)} This point is a simple corollary of the fact that pinching implies the dominance of the $d$-th term in the expansion. Thus, if the $d$-th term in the expansion is almost constant, the whole function is almost constant.
 
\paragraph{Proof of (3)} In order to prove this last point, we consider also the $d-1$ order part in the expansion of $u$ around $\bar x$.
 
 In detail, it is easy to see that 
 \begin{gather}
  a_{d-1}(\bar x) P_{d-1,\bar x} = a_{d-1} P_{d-1} + a_d t \ \partial_1 P_{d}+ \sum_{k=1}^\infty \frac{t^{k+1}}{(k+1)!} a_{d+k}\ton{\partial _1}^{k+1} P_{d+k} \, .
 \end{gather}
Given the pinching in the frequency around $\bar x$, we obtain
\begin{gather}
 \norm{a_{d-1}(\bar x) P_{d-1,\bar x}}^2\leq C\epsilon \norm{a_{d}(\bar x) P_{d,\bar x}}^2 =C\epsilon a_d^2\, .
\end{gather}
Moreover, the triangle inequality implies the easy estimate
\begin{gather}
 \norm{a_{d-1}(\bar x) P_{d-1,\bar x}} \geq \norm{a_d t \ \partial_1 P_d} - \norm{a_{d-1} P_{d-1}} - \sum_{k=1}^\infty \frac{t^{k+1}}{(k+1)!}  a_{d+k}\norm{\ton{\partial _1}^{k+1} P_{d+k}}\, .
\end{gather}
By computations similar to before, we obtain
\begin{gather}
 \norm{a_d t \ \partial_1 P_d} \leq C\abs {a_d} \sqrt \epsilon + Cd \abs {a_d} \sqrt \epsilon  \, .
\end{gather}
Given Lemma \ref{lemma_pdp}, we can conclude
\begin{gather}
 t \abs {a_d} \norm{\partial_1 P_d} \leq C(n)\abs{a_d} d \sqrt \epsilon \, \quad \Longrightarrow \quad \, \norm{\partial_1 P_d} \leq C(n) t^{-1} \sqrt \epsilon \norm{\nabla P_d}\, .
\end{gather}
\end{proof}

With this Lemma, we are in a position to prove a quantitative version of Proposition \ref{prop_pn-2}. In particular, we will see that the points where the frequency is almost pinched around $d$ are almost contained in a $n-2$ dimensional plane. In order to do so, we start by proving that if a harmonic polynomials has $n-1$ partial derivatives suitably close to zero, then it has to be linear.
\begin{lemma}
 Let $P_d:\R^n\to \R$ be a (nonconstant) hhP of degree $d$ such that
 \begin{gather}
  \norm{\partial_i P_d}\leq \sqrt\epsilon \norm{\nabla P_d}
 \end{gather}
for $i=1,\cdots,n-1$ and $\epsilon < \epsilon_0(n)=[2n(n-1)]^{-1}$. Then $P_d$ is linear, i.e., $d=1$.
\end{lemma}
\begin{proof}
 For simplicity we assume that $P_d$ is normalized. By Lemma \ref{lemma_pdp}
 \begin{gather}
  \norm{P_d}^2=1 \quad \Longrightarrow \quad \norm{\nabla P_d}^2 =\sum_{i=1}^n \norm{\partial_i P_d}^2 = d(2d+n-2) \quad \Longrightarrow \\
  \Longrightarrow \quad \norm{\nabla^2 P_d}^2 =\sum_{i,j=1}^n \norm{\partial_i \partial_j P_d}^2 = d(d-1)(2d+n-2)(2(d-1)+n-2)\, .
 \end{gather}
 Also, for each $i$,
 \begin{gather}
  \norm{\partial _i P_d }^2 = (d-1)(2(d-1)+n-2) \norm{\nabla \partial_i P_d}^2 = (d-1)(2(d-1)+n-2) \sum_{j=1}^n \norm{\partial_j \partial_i P_d}^2\, .
 \end{gather}
 Thus we obtain that for $i=1,\cdots,n-1$:
 \begin{gather}
  \norm{\nabla \partial_i P_d} = \norm{\nabla^2 (P_d)[e_i]}<\sqrt \epsilon \norm{\nabla^2 P_d}\, .
 \end{gather}
This in particular implies that $\nabla^2 P_d=(h_{ij})$ is a symmetric matrix (whose elements are hhP's) with $\norm{h_{ij}}^2 \leq \epsilon\norm{h}$ if either $i$ or $j$ are not $n$, so for all $(i,j)\neq (n,n)$. This last term is small as well since $P_d$ is harmonic. Indeed, $\sum_i h_{ii}(x) =0$ for all $x$ implies $\norm{h_{nn}}^2 \leq (n-1)^2 \epsilon \norm h$. Summing everything up we obtain
\begin{gather}
 \norm{h}^2 = \sum_{i,j=1}^n \norm{h_{ij}}^2 \leq \ton{(n^2-1)\epsilon + (n-1)^2 \epsilon } \norm{h}^2\, .
\end{gather}
If $\epsilon < c(n) = 2n(n-1)$, then $\norm{h}^2=0$, which implies $d=1$.
 
\end{proof}

One could rephrase this Lemma in the following way: the almost invariant directions of every nonlinear hhP $P_d$ are always contained in a neighborhood of a subspace $V$ of dimension $\leq n-2$. It is easy to see that this notion is in some sense stable wrt the polynomial $P_d$.
\begin{proposition}\label{prop_alcone}
 Let $d\geq2$ and $P_d, P'_d:\R^n\to \R$ be two nonlinear hhP's with $\norm{P_d}=\norm{P'_d}=1$. Let $I(\epsilon)\subset S\equiv\cur{\norm v =1}\subset \R^n$ be the set of almost invariant directions for $P_d$, i.e., the set of unit vectors $v$ such that
 \begin{gather}
  \norm{\partial_v P_d}\leq \sqrt \epsilon \norm{\nabla P_d}\, , 
 \end{gather}
and let $I'(\epsilon)$ be the corresponding set for $P'_d$. Then for every $\tau>0$, there exists $\epsilon_0(n,\tau)>0$ such that if $0<\epsilon<\epsilon_0$ and
\begin{gather}
 \norm{P_d-P'_d}\leq \sqrt{\epsilon}\, ,
\end{gather}
then there exists a subspace $V\leq \R^n$ of dimension $\leq n-2$ such that $I(\epsilon)\cup I'(\epsilon)\subset B_\tau(V)$. We say that this subspace $V$ as is the almost invariant subspace of $P_d$ and $P'_d$.
\end{proposition}

\begin{remark}
 Note that this proposition is a quantitative version of Proposition \ref{prop_pn-2} which is also stable wrt the $L^2$ norm of the polynomial $P_d$.  Note also that $V$ is only well defined up to an $\epsilon$-perturbation.
\end{remark}

\begin{proof}
 Recall that, by Lemma \ref{lemma_pdp} and the normalization of the polynomials, the following equality holds
 \begin{gather}
  \norm{\nabla P_d}^2 = \norm{\nabla P'_d}^2 = d(2d+n-2)\, .
 \end{gather}
Thus it is easy to see that
\begin{gather}
 \norm{\partial_v P_d - \partial_v P'_d} \leq \sqrt{d(2d+n-2)} \norm{P_d-P'_d}= \sqrt{\epsilon}\norm{\nabla P_d}=\sqrt \epsilon \norm{\nabla{P'_d}}\, ,
\end{gather}
which means $I(\epsilon/4)\subset I'(\epsilon)\subset I(4\epsilon)$. Thus, up to an inconsequential change in $\epsilon_0$, it is sufficient to prove the statement for $I$.

The inclusion $I(\epsilon)\subset B_\tau(V)$ is an easy corollary of the previous Proposition. Suppose by contradiction that for every $n-2$ dimensional plane $V$, $I\setminus B_\tau(V)\neq \emptyset$. Then there exists $n-1$ unit vectors $v_i$ with 
\begin{gather}
\norm{\partial_{v_i} P_d}\leq \sqrt \epsilon \norm{\nabla P_d}\quad \text{  and  }\quad  v_i \not \in B_\tau \ton{\operatorname{span}(v_1,\cdots,v_{i-1})}\, .
\end{gather}
By a simple orthonormalization argument, it is easy to see that there exists $n-1$ \textit{orthonormal} vectors $w_i$ satisfying $\norm{\partial_{w_i} P_d}\leq c(n,\tau)\sqrt\epsilon \norm{\nabla P_d}$. The previous Lemma concludes the proof.
\end{proof}

As a Corollary of this Proposition and Lemma \ref{lemma_epsd}, we obtain that the points with pinched frequency have to be close to an $n-2$ dimensional plane.
\begin{corollary}\label{cor_alcone}
Let $u:B_{e^2 d +1}(0)\to \R$ be a harmonic function, fix some $\tau>0$ and set $\cV$ to be set of points $x\in B_1(0)$ such that
\begin{gather}
 N(x,e^2 d)-N(x,e^{-1})\leq \epsilon \quad \text{ with } \quad \abs{N(x,1)-d}\leq \epsilon\, .
\end{gather}
If $d\geq 2$, there exists $\epsilon_0(n,\tau)$ such that if $\epsilon<\epsilon_0$, then there exists a subspace $V$ of dimension at most $n-2$ such that for all $x\in \cV$
\begin{gather}
 \cV\cap B_1(0)\subset x+B_\tau (V)\, .
\end{gather}

Note that the subspace $V$ may be chosen independently of $x$. Moreover, if $N(x,e^2 d)-N(x,r_x)\leq \epsilon $ with $0\leq r_x\leq e^{-1}$, the subspace $V$ may also be chosen independently of $r_x$.
\end{corollary}

\subsection[Almost n-2 invariant hhP's]{Almost $n-2$ invariant hhP's}\label{ss:almost_codim2_invariant}
In the previous section, we have seen that if some hhP is almost $n-1$ invariant, then it depends only on $1$ variable and thus it is linear. Here we will explore the properties of almost $n-2$ invariant polynomials. Although an almost $n-2$ invariant polynomial is not necessarily a polynomial of $2$ variables, such a function has to be \textit{close} to an hhP of $2$ variables. Exploiting the properties of hhP's in dimension $2$, we will then use this statement to get some control over the critical and almost critical sets of such functions.

\begin{lemma}
 Let $P:\R^n\to \R$ be a hhP of degree $d$ such that
 \begin{gather}
  \norm{\partial_1 P}^2\leq \epsilon \norm{\nabla P}^2 = \epsilon d(2d+n-2) \norm{P}^2\, .
 \end{gather}
There exist constants $\epsilon_0$ such that if $\epsilon\leq \epsilon_0/d^2$, then
\begin{gather}
 P = \sqrt{1-\delta} Q + \sqrt \delta R\, ,
\end{gather}
where $Q$ and $R$ are hhP's of degree $d$ with $\norm Q = \norm R = \norm P$, $Q$ is $x_1$-invariant and $\delta\leq n\epsilon_0$.
\end{lemma}
\begin{proof}
 We will assume for simplicity $\norm P=1$. Let $P=Q+R$, where $Q\in \cP_d$ is independent of $x_1$, and $R$ is orthogonal to all $x_1$ invariant polynomials. Then $\partial_1 P = \partial_1 R$, and so, by Proposition \ref{prop_perpest}, $\norm{R}\leq \norm{\partial_1 R}\leq \sqrt{n\epsilon} d \norm{P}\leq \sqrt{n\epsilon_0} \norm P$.
\end{proof}

Proceeding by successive steps, one can prove the following.
\begin{proposition}\label{prop_poly2}
 Let $P:\R^n\to \R$ be a hhP of degree $d$ such that for $i=1,\cdots,k\leq n-2$
 \begin{gather}
  \norm{\partial_i P}^2\leq \epsilon \norm{\nabla P}^2 = \epsilon d(2d+n-2) \norm{P}^2\, .
 \end{gather}
There exist constants $\epsilon_0$ such that if $\epsilon\leq \epsilon_0/d^{2n-4}$, then
\begin{gather}
 P = \sqrt{1-\delta^2} Q + \delta R\, ,
\end{gather}
where $Q$ and $R$ are hhP's of degree $d$ with $\norm Q = \norm R = \norm P$, $Q$ is $x_1,\cdots,x_k$ invariant and $\delta\leq \sqrt {c(n)\epsilon_0}$.
\end{proposition}
\begin{proof}
 For $k=1$, this is exactly the content of the previous lemma. Thus we can write $P=\sqrt{1-\delta_1}Q_1+\sqrt \delta_1 R_1$, where $Q_1$ is invariant wrt $x_1$ and $\delta_1\leq \epsilon_0/d^{2n-6}$. This in particular implies that
 \begin{gather}
  \sqrt{1-\delta_1}\norm{\partial_2 Q_1} \leq \norm{\partial_2 P} + \sqrt{\delta_1} \norm{\partial_2 R_1}\leq \sqrt{n}d\sqrt{\epsilon}  \norm P + nd^2\sqrt{\epsilon} \norm{R_1} = \sqrt{\epsilon}\ton{\sqrt{n}d  + nd^2}\norm{Q_1}\, . 
 \end{gather}
Given the hypothesis on $\delta_1$, we have the rough estimate
\begin{gather}\label{eq_ere}
 \norm{\partial_2 Q_1} \leq \sqrt 2 \sqrt{\epsilon}\ton{\sqrt{n}d  + nd^2}\norm{Q_1}\, .
\end{gather}
Note that $Q_1$ (and thus also $\partial_2 Q_1$) is independent of $x_1$. As in Lemma \ref{lemma_norm_n}, let $\hat {Q_1}$ and $\hat {\partial_2 Q_1}$ be induced hhP on $\R^{n-1}$. By Lemma \ref{lemma_norm_n}, \eqref{eq_ere} is equivalent to
\begin{gather}
 \norm{\hat{\partial_2 Q_1}} \leq \sqrt 2 \frac{n+2d-2}{n+2d-3} \sqrt{\epsilon}\ton{\sqrt{n}d  + nd^2}\norm{\hat Q_1}\leq 2\sqrt{\epsilon}\ton{\sqrt{n}d  + nd^2}\norm{\hat Q_1}\, .
\end{gather}
Thus we can apply again the previous Lemma and obtain that
\begin{gather}
 Q_1=\sqrt{1-\delta_2}Q_2+\sqrt{\delta_2} R_2\, ,
\end{gather}
with $\delta_2\leq c(n) \epsilon_0 /d^{2n-8}$. Moreover, $Q_2$ is both $x_1$ and $x_2$ invariant.

By induction, we obtain the thesis.
\end{proof}

\subsection{Symmetry and Critical Points}\label{ss:symmetric_criticalpoints}
In this section, we study the properties of functions close to $n-2$ symmetric hhP's and obtain estimates on the critical radius $r_c(x)$ for suitable $x$.

Let $P$ be a hhP of degree $d$ depending only on two variables, where for simplicity we choose the coordinates $(x,y)\in \R^2\times \R^{n-2} =\R^n$ in such a way that $P$ depends only on $x$. As we have seen in \eqref{eq_nabla2vars}, the gradient of $P$ has absolute value $\abs{\nabla P (x,y)}= 2d \norm{P}_{L^2(\partial B)} \abs x ^{d-1}$, thus $P$ has no critical points outside its $n-2$ dimensional invariant plane $V$. The aim of this section is to obtain a quantitative version of this property.

\subsubsection{Harmonic functions in $\R^2$} First of all, we restrict ourselves to harmonic functions in $\R^2$, since in this situation the statements are stronger and easier to prove. In the previous sections, we have seen how the pinching on Almgren's frequency affects the expansion of a harmonic function at a point. Here we prove an important connection between pinching and the critical points (or better, the lack thereof).
\begin{proposition}\label{prop_crit2}
 Let $u:B_{e^2}(0)\subset \R^2\to \R$ be a harmonic function. There exists an $\epsilon_0$ independent of $d$ such that if
 \begin{gather}
   N(0,e^2)-  N(0,e^{-2})\leq \epsilon
 \end{gather}
with $\epsilon \leq \epsilon_0$, then $u$ does not have critical points on $\partial B_1(0)$.
\end{proposition}
\begin{proof}
As done previously, we consider the Taylor expansion of $u$
\begin{gather}
 u= \sum_{k=1}^\infty a_k P_k\, ,
\end{gather}
where without loss of generality, we assume $u(0)=0$ and $\norm{u}=h(1)=1$.

By Theorem \ref{t:eff_tan_con_uniq_harm}, there exists an integer $d$ such that
 \begin{enumerate}
  \item for all $t\in [e^{-2}, e^2]$, $\abs{ N(t)-d}\leq 3\epsilon$
  \item for all $t\in [e^{-1}, e^1]$,
   \begin{gather}\label{eq_kdsum}
    \sum_{k\neq d} a_k^2 t^{2k} \leq 6\epsilon h(t)\, ,
   \end{gather}
  or equivalently
   \begin{gather}
    a_d^2 t^{2d} \geq (1-6\epsilon )h(t)\, .
   \end{gather}
 \end{enumerate}
With this relation we can compare the gradient of $u$ with the gradient of its leading term, $P_d$:
\begin{gather}
 \delta \equiv u-a_d P_d = \sum_{k\neq d} a_k P_k \, \Longrightarrow \, \abs{\nabla \delta} \leq \sum_{k\neq d} \abs{a_k} \abs{\nabla P_k} \, ,
\end{gather}
In particular, for $x\in \partial B_1(0)$, we have
\begin{gather}
 \abs{\nabla P_d}(x) = 2d\, , \quad \abs{\nabla \delta}(x) \leq 2 \sum_{k\neq d} k\abs{a_k}\, . 
\end{gather}
In order to estimate the last sum, we split it in two parts: the sum from $d+1$ to infinity, and the sum up to $d-1$. We can estimate
\begin{gather}
 \sum_{k\geq d+1} k\abs{a_k} \leq \ton{\sum_{k\geq d+1} \abs{a_k}^2 e^{2k-2d} }^{1/2}\ton{\sum_{k\geq d+1} k^2e^{2d-2k}  }^{1/2}\, .
\end{gather}
The first term on the rhs can be estimated using \eqref{eq_kdsum}. Indeed by this equation, \eqref{eq_doth} and the pinching on the frequency we have
\begin{gather}
 \sum_{k\geq d+1} \abs{a_k}^2 e^{2k} \leq 6\epsilon e^{2d+6\epsilon} \, \Longrightarrow   \sum_{k\geq d+1} \abs{a_k}^2 e^{2k-2d} \leq 6e\epsilon\, ,
\end{gather}
where we assumed $\epsilon_0\leq 6^{-1}$. As for the second term, we can use the comparison with integrals.
\begin{gather}
 \sum_{k\geq d+1} k^2e^{-2k}  \leq \int_d^\infty x^2 e^{-2x} dx = -\frac 1 2 \qua{ \ton{x^2+x-\frac 1 2 } e^{-2x} }_d^\infty = -\frac 1 2 \ton{d^2+d-\frac 1 2 } e^{-2d} \leq cd^2 e^{-2d}\, .
\end{gather}
In a similar way, we can deal with the sum up to $d-1$. As before, we use Cauchy inequality to split the sum and get
\begin{gather}
 \sum_{k\leq d-1} k\abs{a_k} \leq \ton{\sum_{k\leq d-1} \abs{a_k}^2 e^{2d-2k} }^{1/2}\ton{\sum_{k\leq d-1} k^2e^{2k-2d}  }^{1/2}\, .
\end{gather}
The first term on the rhs can be estimated using \eqref{eq_kdsum}. Indeed by this equation, \eqref{eq_doth} and the pinching on the frequency we have
\begin{gather}
 \sum_{k\leq d-1} \abs{a_k}^2 e^{-2k} \leq 6\epsilon e^{-2d+6\epsilon} \, \Longrightarrow   \sum_{k\leq d-1} \abs{a_k}^2 e^{-2k+2d} \leq 6e\epsilon\, .
\end{gather}
As for the second term, we can use again the comparison with integrals.
\begin{gather}
 \sum_{k\leq d-1} k^2e^{2k}  \leq \int_0^d x^2 e^{2x} dx = \frac 1 2 \qua{ \ton{x^2-x+\frac 1 2 } e^{2x} }_d^\infty = \frac 1 2 \ton{d^2-d+\frac 1 2 } e^{2d} \leq cd^2 e^{2d}\, .
\end{gather}
Summing up, we obtain
\begin{gather}
 \sum_{k\neq d } k\abs {a_k} \leq cd\epsilon\, .
\end{gather}
If $\epsilon_0\leq (4c)^{-1}$, we have for all $x\in \partial B_1(0)$:
\begin{gather}
 \abs{\nabla u}(x) \geq a_d \abs{\nabla P_d}(x) - \abs{\nabla \delta}(x) \geq d -2cd\epsilon >0\, .
\end{gather}
\end{proof}

It is possible to improve the previous theorem to obtain information not only on the gradient of $u$ at $x\in \partial B_1(0)$, but also on its critical radius $r_c(x)$.
\begin{proposition}\label{prop_h2effcrit}
 Let $u:B_{e^2}(0)\subset \R^2\to \R$ be a harmonic function. There exist $\epsilon_0,r_0$ independent of $d$ such that if
 \begin{gather}
   N(0,e^2)-  N(0,e^{-2})\leq \epsilon
 \end{gather}
with $\epsilon \leq \epsilon_0$ and $\abs{ N(0,1) -d }<1/2$, then for all $x\in \partial B_1(0)$, $r_c(x)\geq r_0 d^{-1}$.
\end{proposition}
\begin{remark}
 By studying hhP of two variables, it is easy to realize that the lower bound on $r_c(x)$ cannot be independent of $d$. however, in our computations this only adds a polynomial error to the final estimate.
\end{remark}
\begin{proof}
 The proof of this theorem is very similar in spirit to the proof of the previous proposition. However, in order to get estimates on $r_c(x)$, it is not sufficient to concentrate on the gradient of the function $u$. We need to estimate all the terms in the Taylor expansion of the function $u$ at $x$.
 
 As before, we start by writing the expansion of $u$ at the origin and at some point $x\in \partial B_1(0)$:
 \begin{gather}
  u(y)=\sum_k a_k P_k(y)\, , \quad \quad u(x+y)=\sum_k a_k(x) P_{k,x}(y)\, .
 \end{gather}
 The pinching condition implies that
 \begin{gather}
  \sum_{k\neq d} a_k^2 e^{4\abs{k-d}} \leq c\epsilon a_d^2\, .
 \end{gather}
For simplicity, we will assume that $a_d^2=1$, and that $x=(t,0,\cdots,0)$. By re-expanding $u$ at $x$, we get
\begin{gather}
 a_k(x) P_{k,x}(y)= \sum_{s=0}^\infty a_{k+s} \frac{t^s}{s!} \partial_1^s P_{k+s}\, ,\\
 \abs{a_k(x)}\leq \sum_{s=0}^\infty \abs{a_{k+s}} t^s \binom{k+s}{s}\, .
\end{gather}
Since $t=1$, for $k\geq d+1$ we obtain
\begin{gather}
 \abs{a_k(x)}\leq c\sqrt \epsilon \sum_{s=0}^\infty e^{-2(k-d)} e^{-2s} \frac{(k+s)^s}{s!}\leq c\sqrt \epsilon \sum_{s=0}^\infty e^{-2(k-d)} e^{-s} \ton{1+\frac k s}^s\leq c\sqrt \epsilon e^{-k+2d}\sum_{s=0}^\infty e^{-s}\leq c\sqrt \epsilon e^{-k+2d}\, .
\end{gather}
For $k\leq d$, it is convenient to separate the contribution coming from the expansion of the degrees $\leq d-1$, $=d$ and $\geq d+1$. In such a way we obtain
\begin{gather}
 a_{k}(x)P_{k,x}(y)= \sum_{s=0}^{d-k-1} a_{k+s} \frac{t^s}{s!}\partial_1^s P_{k+s} \, +\, t^{d-k}{(d-k)!} \partial_1^{d-k} P_d \, +\, \sum_{s=d-k+1}^\infty a_{k+s} \frac{t^s}{s!}\partial_1^s P_{k+s}\, ,\\
 \abs{a_k}\leq \sum_{s=1}^{d-k} a_{d-s} \binom{d-s}{k} \, + \, \binom d k \,  + \,  \sum_{s=1}^\infty a_{d+s} \binom{d+s} k \leq \binom d k + c\sqrt\epsilon\qua{\sum_{s=1}^{d-k} e^{-2s} \frac{(d-s)^k}{k!} +  \sum_{s=1}^\infty e^{-2s} \frac{d^k}{k!} \ton{1+\frac{s}{d}}^k }\leq\\
 \leq \binom{d}{k} +c\sqrt\epsilon \frac{d^k}{k!}\qua{c+\sum_{s=1}^\infty s^k e^{-2s}}\leq \frac{d^k}{k!} +c\sqrt\epsilon d^k\qua{\frac 1 {k!} + 2^{-k}} \, .
\end{gather}
 By the previous proposition, it is easy to see that
 \begin{gather}
  (1-c\epsilon)d \leq \abs {a_1}\leq (1+c\epsilon) d\, .
 \end{gather}
Putting together these estimates, we obtain the following very rough bound on the frequency $ N(x,r_0 d^{-1})$:
\begin{gather}
  N(x,r)= \frac{\sum_{k=1}^\infty k a_k(x)^2 r^{2k}}{\sum_{k=1}^\infty a_k(x)^2 r^{2k}}\leq 1+\frac{\sum_{k=2}^\infty k a_k(x)^2 r^{2k}}{a_1(x)^2 r^2 }\, ,\\
  N(x,r_0 d^{-1})\leq 1+cr_0^2 \sum_{k=2}^d\ton{k \frac{d^k}{k!} \frac 1 {d^k}} + c\sqrt \epsilon r_0^2 \qua{\sum_{k=2}^d \frac k {k!} + k2^{-k}  + \sum_{k=d+1}^\infty ke^{-k+2d} d^{-k} }\, .
\end{gather}
It is clear that $r_0$ can be chosen in such a way that $ N\ton{x,r_0d^{-1}}\leq 1+3/2$, and this proves the thesis.

\end{proof}

\subsubsection{Harmonic functions in $\R^n$} For general $n$, with similar computations we can obtain similar results. However, in this case the results we obtain are somewhat weaker, in particular the constant $\epsilon_0$ will not be independent of the degree $d$ of the polynomial.

\begin{proposition}\label{prop_n-2}
 Let $u:B_1(0)\subset \R^n\to \R$ be a harmonic function which can be written as
 \begin{gather}
  u=Q_d + \sum_k a_k P_k\, ,
 \end{gather}
where $P_k$ are normalized hhP's of degree $k$, and $Q_d$ is a normalized hhP of degree $d$ invariant wrt the $n-2$ dimensional plane $V$. For $0<\tau\leq 1$, there exists a constant $c(n)$ such that if
\begin{gather}
 \sum_k \abs{a_k}^2 e^{2\abs {k-d}}\leq \epsilon\, 
\end{gather}
with $\epsilon \leq (c(n)\tau )^{2d-2}$, then $u$ does not have critical points in $B_1(0)\setminus B_\tau (V)$.
\end{proposition}
\begin{proof}
 The proof of this Lemma follows closely the proof of Proposition \ref{prop_crit2}. Indeed, define $\delta = \sum_k a_k P_k$, and consider that
 \begin{gather}
  \abs{\nabla \delta}\leq \sum_k \abs{a_k} \abs{\nabla P_k}\, .
 \end{gather}
The normalization on $P_k$, along with Lemmas \ref{lemma_pdp} and \ref{lemma_dest}, imply that
\begin{gather}
 \norm{\nabla P_k}^2 \leq k(2k+n-2)\leq nk^2\, \quad \Longrightarrow \quad \, \abs{\nabla P_k(x)} \leq c(n) k^{n/2} \abs x ^{k-1}\, .
\end{gather}
Thus for all $x\in B_1(0)$:
\begin{gather}
 \abs{\nabla \delta(x)}\leq c(n)\sqrt{\epsilon} \sum_k e^{-\abs{k-d}} k ^{n/2}\leq c(n) \sqrt{\epsilon} d^{n/2}\, .
\end{gather}
On the other hand, let $\abs{x-V}$ be the distance from $x$ to $V$. By the properties of hhP's of two variables
\begin{gather}
 \abs{\nabla Q_d(x)}\geq 2 d \abs{x-V}^{d-1}\, ,
\end{gather}
thus, if $x\in B_1(0)\setminus B_\tau (V)$, 
\begin{gather}
 \abs{\nabla u (x)} \geq \abs{x-V}^{d-1} \ton{2d - \tau^{1-d}c(n) d^{n/2} \sqrt{\epsilon}  }\, ,
\end{gather}
which implies the thesis.
\end{proof}

As for the $n=2$ case, also for general dimension it is not difficult to improve the previous statement in order to get estimates on the \textit{effective} critical set.

\begin{proposition}\label{prop_heffcrit}
  Let $u:B_1(0)\subset \R^n\to \R$ be a harmonic function which can be written as
 \begin{gather}
  u=Q_d + \sum_k a_k P_k\, ,
 \end{gather}
where $P_k$ are normalized hhP's of degree $k$, and $Q_d$ is a normalized hhP of degree $d$ invariant wrt the $n-2$ dimensional plane $V$. For $0<\tau\leq 1$, there exists a constant $c(n)$ such that if
\begin{gather}
 \sum_k \abs{a_k}^2 e^{2\abs {k-d}}\leq \epsilon\, 
\end{gather}
with $\epsilon \leq (c(n)\tau )^{2d-2}$, then for all $x\in B_1(0)\setminus B_\tau(V)$, $r_c(x)\geq c(n) \tau^{d}$. 
\end{proposition}
\begin{proof}
The proof follows as in the $n=2$ case using the results of the previous proposition.
\end{proof}

\subsection{Volume estimates on the effective critical sets}\label{ss:volume_estimates_harmonic}
In this section, we prove the main volume estimates on the effective critical set. The proof is obtained by successive covering of ``good'' and ``bad scales''.

We start with the definition of a good scale for the function $u$ relative to the degree $d$. As we will see, on these scales we will have nice covering arguments for the set $\cS_r(u)$.

We fix $\tau=1/100$, and set $\epsilon=\epsilon(n,d)$ given by the minimum of $\epsilon_0(n)/2$ in Theorem \ref{t:eff_tan_con_uniq_harm}, $\epsilon_0(n)/2$ in Lemma \ref{lemma_epsd}, $\epsilon_0(n,\tau)/2$ in Proposition \ref{prop_alcone} and $\epsilon(n,d)=(c(n)\tau)^{2d-2}$
from Proposition \ref{prop_heffcrit}.

\begin{definition}
 Let $u$ be a harmonic function defined on some domain $D$ with $B_{2t}(x)\subset D$. We say that $(x,t)$ is a good scale for $u$ (or equivalently we say that $B_t (x)$ is a good scale ball for $u$) relative to the degree $d$ if $N(y,t)\leq d+\epsilon$ for all $y\in B_t (x)$.
\end{definition}

\begin{definition}
 Fix some positive $r$ and suppose that $B_t (x)$ is a good scale ball for $u$ relative to $d$. Then we define
 \begin{gather}
  r'_x = \sup\cur{s\geq 0 \ \ s.t. \ \ N(x,s)\geq d-\epsilon}\, , \quad r_x = \max\cur{r'_x,r}\, ,
 \end{gather}
where as a convention we set $r'_x=\infty$ if $N(x,s)$ is never $\geq d-\epsilon$ on the domain of $u$. Moreover, for any positive $r$, we also set
 \begin{gather}
  \cS_r(u)=\cS= \cur{x\in B_1(0) \ \ s.t. \ \  N(x,r)\geq 3/2}\, ,\\
  \cS_g(u) =\cS_g= \cur{x\in \cS \ \ s.t. \ \ \forall y \in \cS\cap B_{5r_x}(x), \ \ r_y \geq r_x/7}\, ,\\
  \cS_b(u) =\cS_b= \cS(u)\setminus \cS_g(u) =\cur{y\in \cS \ \ s.t. \ \ \exists x \in \cS\cap B_{5r_y}(y), \ \ r_x < r_y/7}\, .
 \end{gather}
\end{definition}

The following proposition gives us a covering of the set $\cS(u)$ on a good scale. Later on, we will deal with bad scales. 
\begin{proposition}
 Let $B_1(0)$ be a good scale ball for $u$ relative to the degree $d$. Then there exists $x_i\in \cS_r(u)$ and $s_i>0$ such that
 \begin{gather}
  \cS_{c(n)\tau^d r}(u)\subset \bigcup_{i} B_{s_i}(x_i)\, , \quad \sum_i s_i^{n-2}\leq r^{n-2} C(n)d^n
 \end{gather}
and such that for every $y\in B_{s_i}(x_i)$, either $s_i\leq r$ or $ N(y,7^{-1}\epsilon s_i)\leq d-1+\epsilon$.
\end{proposition}
\begin{proof}
We can assume that $r\leq (e^3d)^{-1}$, otherwise a simple Vitali covering of $\cS_r\cap B_1(0)$ will do the trick.
 
 Consider the collection of balls $B_{r_x}(x)$ with centers in $x\in \cS_g(u)$, and pick a Vitali subcovering of $\cS_g(u)$, i.e., a finite collection of balls such that
 \begin{gather}
  \cS_g(u)\subset \bigcup_{i} B_{5r_i}(x_i)\quad \text{  and  }\quad B_{r_i}(x_i)\cap B_{r_j}(x_j)=\emptyset\, ,
 \end{gather}
where $r_i=r_{x_i}$. For each $r_i$ we have two options, either this radius is smaller than $(e^3d)^{-1}$ or not. In the first case, we say that $i\in G_g$, in the second, $i\in G_b$.

An immediate volume argument allow us to estimate
\begin{gather}
 \sum_{i\in G_b} r_i^{n-2}\leq C(n) d^n\, .
\end{gather}
As for the indexes in $G_g$, we can partition this set further in subfamilies, in such a way that for each $i$ and $j$ in each subfamily, $d(x_i,x_j)\leq (e^3d)^{-1}$. Again, the number of such subfamilies is bounded above by $C(n)d^n$.

Now, pick $i,j$ in the same subfamily, and consider that
\begin{gather}
 r_i+r_j\leq d(x_i,x_j)\leq (e^3d)^{-1}\, .
\end{gather}
By definition of good scale and $r_i$, $N(x_i,1)- N(x_i,r_i)\leq 2\epsilon$, and the same holds for $x_j$. Thus we can apply Theorem \ref{t:eff_tan_con_uniq_harm} to obtain the existence of a unique normalized hhP $P_d$ such that for all $s\in [r_i,(e^3d)^{-1}]$, $\norm{T_{x_i,s}-P_d}\leq \sqrt{7\epsilon}$. A similar statement is true for $x_j$, and we denote $P'_d$ the approximating polynomial in this case.

By the almost cone splitting proved in Corollary \ref{cor_alcone}, there exists a common almost invariant subspace $V\leq \R^n$ of dimension at most $n-2$ for $P_d$ and $P'_d$, and $x_j$ is effectively close to $x_i+V$, in the sense that
\begin{gather}
 d(x_j-x_i,V) \leq \tau d(x_i,x_j)=100^{-1} d(x_i,x_j)\, .
\end{gather}

Since this argument holds for any $i,j$ in the same subfamily, by the Lipschitz extension theorem there exists a Lipschitz function $f:V\to V^\perp$ with Lipschitz constant $\leq 10^{-1}$ such that all $x_i$ in the same subfamily belong to the graph of $f$, which we denote by $\Gamma(f)$.

This allow us to estimate the sum $\sum r_i^{n-2}$, where $i$ belong to the same subfamily. Indeed, this sum is bounded above by a constant depending only on the Lipschitz constant of $f$ and on the $n-2$ Lebesgue measure of an $n-2$ dimensional ball of radius $(e^3d)^{-1}$. Summing over all subfamilies we obtain
\begin{gather}
 \sum_{i\in G_g} r_i^{n-2}\leq C(n) d^2\, .
\end{gather}
In the end, we have
\begin{gather}
 \sum_{i\in G_g\cup G_b} r_i^{n-2} \leq C(n) d^n\, .
\end{gather}

As for the drop in the frequency, let $z\in B_{5r_i}(x_i)$. By definition of $\cS_g$, $r_z\geq r_i/7$, which with Lemma \ref{lemma_Npinch} proves the frequency drop.

\paragraph{Covering of $\cS_b$} Now we turn our attention to the set $\cS_b$. We divide this argument in two subcases.

\paragraph{If $V$ has dimension $\leq n-3$}. It is easy to see that for each $y\in \cS_b$, there exists a point $x\in \cS_g$ such that
\begin{gather}
 d(x,y)\leq \ton{5\sum_{k=0}^\infty 7^{-k}}r_y \leq 6 r_y \quad \text{  and  } \quad r_x<7^{-1} r_y\, .
\end{gather}
In turn, there exists some $B_{5r_i}(x_i)$ in the covering of $\cS_g$ such that $x\in B_{5r_i}(x_i)$ and $r_i\leq 7^{-1}r_x$. This implies that
\begin{gather}\label{eq_rj}
 r_i\leq r_y \quad \text{  and  } \quad d(x_i,y)\leq 11r_y\, .
\end{gather}
For all $y\in \cS_b$, define $t_y=55^{-1}\min_i d(y,x_i)\leq r_y/5$, and consider the covering of $\cS_b$ given by $\cup_{y\in \cS_b} B_{t_y}(y)$. A Vitali subcovering has the property that
\begin{gather}
 \cS_b\subset \bigcup_j B_{5t_j}(y_j) \quad \text{  and  }\quad B_{t_j}(y_j)\cap B_{t_k}(y_k)=\emptyset\, .
\end{gather}
We partition the index set $J$ into $J_k$, with $k=0,\cdots,\infty$ such that $j\in J_k$ only if $t_j\in (2^{-k-1},2^{-k}]$. Denote by $\Gamma(f)$ the union of all the graphs of the functions $f$ in the subfamilies described above. By \eqref{eq_rj}, $d(y_j,\Gamma(f))\leq 55t_j$, and so for every $k$ we can estimate the number of balls elements in $J_k$ by
\begin{gather}
 \#(J_k)\leq \frac{\Vol\ton{B_{55\times 2^k} \Gamma(f)} }{\Vol B_{2^{-k-1}}(0)}\leq c(n) 2^{k(3-n)} d^2\, .
\end{gather}
Thus we obtain
\begin{gather}
 \sum_{j\in J} t_j^{n-2} \leq \sum _k 2^{k(n-2)} \# (J_k) \leq c(n)d^2\, .
\end{gather}

Now consider any $z\in B_{5t_j}(y_j)\cap \cS_2$. Evidently for all $i$: $d(z,x_i)\geq d(y_j,x_i)-5t_j \geq 50 t_j $. Moreover, since $r_z\geq d(z,x_i)/11$, we have that $r_z\geq 5 t_j$, thus proving the frequency drop.

\paragraph{If $V$ has dimension $=n-2$.} In this case, we see that all the hypothesis of Proposition \ref{prop_heffcrit} are satisfied. 

Thus in particular, if $z\in B_1(0)$ is such that there exists $i$ with
\begin{gather}
 5r_i<d(z,x_i)< (e^3d)^{-1}\, \quad \text{ and } \quad d(z,x_i+V) > \tau d(z,x_i)\, ,
\end{gather}
then $r_c(z)\geq (c(n)\tau)^{d}r_i$, which means that, by definition, $z\not \in \cS_{(c(n)\tau)^d r}$.

As for the points such that $d(z,x_i+V) \leq \tau d(z,x_i)$, we can cover them as we covered $\cS_b$ in the previous case and obtain easily the $n-2$ Minkowski estimate on them. Indeed, these points are effectively close to an $n-2$ dimensional subspace.
\end{proof}

With this proposition, we are ready to prove our main theorem.

\begin{theorem}\label{th_vol_harm}
 Let $u:B_1(0)\to \R$ be a harmonic function with $N(0,1)\leq \Lambda$. There exists a constant $C(n)$ such that
 \begin{gather}
  \Vol\ton{B_r\ton{\cS_r(u)} \cap B_{1/2}(0)}\leq C(n)^{\Lambda^2} r^2\, .
 \end{gather}
\end{theorem}

\begin{proof}
We are going to prove the theorem by successive covering of the set $S_r\cap B_{1/2}(0)$, in such a way that in each step we will cover a ball of radius $s$ in the previous step with balls of radia $s_i\geq r$ with
\begin{gather}
 \sum_i s_i^{n-2}\leq C(n)^\Lambda s^{n-2}\, .
\end{gather}
As we will see, the number of steps in the induction will be bounded above by $C(n)\Lambda$, thus the estimate follows. Define for convenience ${\tilde r}=c(n)\tau^{d} r$. 

First of all, observe that, by Theorem \ref{th_cN}, $N(x,1/3)\leq C(n)\Lambda$ for all $x\in B_{1/2}(0)$. Let $d^\star$ be the integral part of $C(n)\Lambda$, then by Lemma \ref{lemma_Npinch} $N(x,\epsilon/4)\leq d^\star + \epsilon$. This in particular implies that for every $x$, $B_{\epsilon/4}(x)$ is a good scale ball relative to the degree $d^\star$. We cover $\cS_{\tilde r}\cap B_{1/2}(0)$ with $C(n)\epsilon^{-n}$ such balls, say $B_{s_{0,i}}(x_{0,i})$ such that
\begin{gather}
 \sum_i s_{0,i}^{n-2}\leq C(n)\epsilon^{-2}\leq C(n)^{\Lambda}\, .
\end{gather}
Fix one $i$, the previous proposition gives us a covering of $\cS_{\tilde r}\cap B_{s_{0,i}}(x_{0,i})$ by balls $B_{t_{0,j}}(y_{0,j})$ such that
\begin{gather}
 \sum_{j} t_{0,j}^{n-2} \leq C(n)^{\Lambda} s_{0,i}^{n-2}\, .
\end{gather}
Evidently, if we consider all the balls $B_{t_{0,i}}(y_{0,i})$ in the coverings of all $B_{s_{0,i}}(x_{0,i})$ we obtain
\begin{gather}
 \sum_i t_{0,i}^{n-2}\leq C(n)^{2\Lambda}\, .
\end{gather}
Moreover, for each $x\in B_{t_{0,i}}(y_{0,i})$, either $t_{0,i}=r$ (and in such a case we keep this ball untouched in the successive steps), or the ball of radius $t_{0,i}\epsilon/7$ is a good scale ball relative to $d^\star -1$. Since we can cover each $B_{t_{0,i}}(y_{0,i})$ by $C(n) \epsilon^{n}$ such balls, we obtain a covering of $\cS_{\tilde r}\cap B_{1/2}(0)$ by balls $B_{s_{1,i}}(x_{1,i})$ which are good with respect to $d^\star -1$ and such that
\begin{gather}
  \sum_i s_{1,i}^{n-2}\leq C(n)^{3\Lambda}\, .
\end{gather}
Recall that $\epsilon=\epsilon_0(n) \tau^{2d-2}\leq \epsilon_0(n) C(n)^{-\Lambda}$. We repeat this argument $d^\star -1$ times, and obtain a covering of $\cS_{\tilde r}\cap B_{1/2}(0)$ made by balls $B_{t_i}(y_i)$ such that either $t_i=r$ or for all $x\in B_{t_i}(y_i)$, $N(x,\epsilon_0 t_i)\leq 3/2$. Thus if $t_i\geq \epsilon$, then $B_{t_i}(y_i)\cap \cS_{\tilde r} = \emptyset$. Otherwise, $\cS_{\tilde r}\cap B_{t_i}(y_i)$ can be easily covered by at most $C(n)\epsilon^{-n}$ balls of radius $r\geq \epsilon t_i$. 

By induction, it is easy to realize that at this last step we obtain a covering of $\cS_{\tilde r}\cap B_{1/2}(0)$ by at most $M$ balls of radius $r$, where $M r^{n-2}\leq C(n)^{\Lambda^2}$. Since ${\tilde r}\leq r C(n)^{-\Lambda}$, we obtain the thesis.

\end{proof}

\begin{remark}
 If we deal with functions in $\R^2$, we can obtain better estimates. Indeed, in this case the $\epsilon$-regularity theorem works for $\epsilon\leq \epsilon_0(n)$, without any dependence on $d$, and one does not need to use the cone splitting Lemma in \ref{lemma_epsd}. In the next two statements, we briefly describe how to modify the arguments stated previously in order to obtain these better estimates.
\end{remark}

For the next proposition, fix $\epsilon(n)$ to be the minimum of $\epsilon_0(n)$ given by Theorem \ref{t:eff_tan_con_uniq_harm} and Proposition \ref{prop_h2effcrit}.
\begin{proposition}
 Let $B_1(0)\subset \R^2$ be a good scale ball for $u$ relative to the degree $d$, fix $r>0$ and let $\tilde r = r r_0/d$, where $r_0=r_0(n)$ is the one in Proposition \ref{prop_h2effcrit}.
 
 There exists a single $x\in \cS_r$ such that $\cS_{\tilde r}\subset B_{r_x}(x)$ and either $r_x=r$ or for all $y\in B_{r_x}(x)$, $N(y,\epsilon r_x)\leq d-1+\epsilon$.
\end{proposition}
\begin{proof}
If for all $x\in \cS_r$, $r_x=\infty$, then we obtain our estimate just by considering $B_1(0)$ as a cover for itself.  

In the other cases, let $x\in \cS_r$ be (one of the) points for which $r_x$ is minimum. By definition of $r_x$ and good scale, we have that $N(x,1)-N(x,r_x)\leq \epsilon$, and thus, by Proposition \ref{prop_h2effcrit}, $B_1(x)\setminus B_{r_x}(x)$ has empty intersection with $\cS_{\tilde r}(u)$.
 
 Moreover, since $r_x$ has minimum value, the statement about frequency drop is trivial.
\end{proof}

With this proposition, using the exact same argument as before, we can prove the following improved estimate in dimension $2$.

\begin{theorem}
 Let $u:B_1(0)\subset \R^2\to \R$ be a harmonic function with $N(0,1)\leq \Lambda$. There exists a constant $C(n)$ such that
 \begin{gather}
  \Vol\ton{B_r\ton{\cS_r(u)} \cap B_{1/2}(0)}\leq C(n)^{\Lambda} r^2\, .
 \end{gather}
\end{theorem}


\section{More general elliptic equations}\label{s:general_elliptic}
Using the same technique as in the harmonic case, one can obtain similar results also for solutions of more general elliptic equations of the form \eqref{eq_Lu}. The most important tool in the estimates proved in the previous sections is Almgren's frequency and its monotonicity properties. For this reason, we start our analysis of elliptic equations by recalling the definition and basic properties of the generalized frequency. For convenience, we follow the notation used in \cite{ChNaVa}, which is a generalization of similar constructions given in \cite{galin1,galin2,HLrank,hanlin,hanhardtlin}.

\subsection{Generalized frequency}\label{ss:generalized_frequency}

 Fix an origin $\bar x$, and define the function $r^2$ by
\begin{gather}
 r^2=r^2(\bar x,x)=a_{ij}(\bar x) (x-\bar x)^i (x-\bar x)^j\, ,
\end{gather}
where $x=x^ie_i$ is the usual decomposition in the canonical basis of $\R^n$, and $a_{ij}$ is the inverse matrix of $a^{ij}$. Note that the level sets 
of $r$ are Euclidean ellipsoids centered at $\bar x$.
\begin{definition}\label{prop_gij}
 Given $a^{ij}$ satisfying \eqref{e:coefficient_estimates}, set
\begin{gather}
 \eta(\bar x,x)={a^{kl}(x)\frac{\partial r(\bar x,x)}{\partial x^k}\frac{\partial r(\bar x,x)}{\partial x^l}}
={a^{kl}(x)\frac{a_{ks}(\bar x)a_{lt}(\bar x)(x-\bar x)^s(x-\bar x)^t}{r^2}}\, ,\\
g_{ij}(\bar x,x)=\eta(\bar x,x)a_{ij}(x)\, .
\end{gather}
\end{definition}
\begin{remark}
 This metric has been introduced in the work \cite{toc}. It is important to underline that the geodesic distance $d_{\bar x}(\bar x,x)$ in the metric $g_{ij}(\bar x,x)$ is equal to $r(\bar x,x)$ for all $x,\bar x$. In particular, the geodesic ball $\cur{x\ \ s.t. \ \ d_{\bar x}(\bar x,x)<r}$ coincides with the Euclidean ellipsoid $\cur{x \ \ s.t. \ \ a_{ij}(\bar x) (x-\bar x)^i (x-\bar x )^j <r}=\bar x + Q_{\bar x} ^{-1}  (B_r(0))$.
\end{remark}

Now we are ready to define the generalized frequency function for a (weak) solution 
$u$ to \eqref{eq_Lu}. For ease of notation, we will keep using the symbol $N$ also for the generalized frequency.
\begin{definition}\label{deph_LN}
For a solution $u:B_1(0)$ to equation \eqref{eq_Lu}, for each $\bar x\in B_1(0)$ and
 $r\leq \lambda^{-1/2}(1-\abs {\bar x})$, define
\begin{gather}
 I(u,\bar x,g,r)=\int_{B(g(\bar x),\bar x,r)}\norm{\nabla u}_{g(\bar x)}^2
+ (u-u(\bar x))\Delta_{g(\bar x)} (u )dV_{g(\bar x)}\\
 D(u,\bar x,g,r)=\int_{B(g(\bar x),\bar x,r)}\norm{\nabla u}_{g(\bar x)}^2\\
H(u,\bar x,g,r)=\int_{\partial B(g(\bar x) ,\bar x,r)} 
\qua{u-u(\bar x)}^2 dS_{g(\bar x)}\\
N(u,\bar x,g,r) =\frac{rI(u,\bar x,g,r)}{H(u,\bar x,g,r)}\, ,
\end{gather}
where for convenience of notation we write $B(g,x,r)$ for the ball of radius $r$ centered at $x$ wrt the metric $g$.
\end{definition}
\begin{remark}
If the operator $\L$ in \eqref{eq_Lu} is the usual Laplace operator, then it is easily seen that this new definition coincides with the old one. This is why we call $N$ the \textit{generalized} frequency for solutions to \eqref{eq_Lu}.
\end{remark}
Note that $N$ has the same invariance properties than $u$. In particular,
\begin{gather}
 N(u,x,g,r) = N(T^u_{x,r},0,g_T,1)\, ,
\end{gather}
where $g_T$ is the metric defined according to the equation satisfied by $T$ (see \eqref{eq_dephT_Q}).

For convenience, from now on we will use the notation
\begin{gather}
 N(u,x,g,r) = N^u(x,r)\, 
\end{gather}
when there is no risk of confusion regarding the metric $g$.

\begin{proposition}\label{prop_genfreq}
 There exist constants $r_0,C$ depending only on $n$ and $\lambda$ such that
 \begin{enumerate}
  \item $N$ is almost monotone, in the sense that $e^{Cr} N(x,r)$ is monotone for all solutions $u$ and for all $r\in (0,r_0]$;
  \item $N$ controls the growth of $u$, in particular for $0<s<r \leq r_0$:
  \begin{gather}
   \abs{\frac{H(s)}{H(r)} \exp\ton{-2\int_s^r \frac{N(t)}{t} dt} -1} \leq C r\, ,
  \end{gather}
  \item $I$ and $D$ are almost equal, in particular
  \begin{gather}
   \frac{\abs{I(r)-D(r)}}{D(r)} \leq C r\, .
  \end{gather}

 \end{enumerate}
\end{proposition}
Moreover, by the existence and uniqueness of the tangent map proved in \cite{han_sing}, the limit $N(x,0)=\lim_{r\to 0} N(x,r)$ exists for all $x$ and it is the vanishing order of $u-u(x)$ at $x$ (thus, it is an integer $\geq 1$).

It is important to underline that a generalization of Theorem \ref{th_cN} is available also for general elliptic equations, although it is necessary to restrict ourselves to $r\leq r_0(n,\lambda)$.  The proof of the following is analogous to the proof of Theorem \ref{th_cN}, up to some minor technical modifications.

\begin{theorem}\label{th_ellcN}
 Let $u$ solve \eqref{eq_Lu} with \eqref{e:coefficient_estimates}, and assume for simplicity that $a^{ij}(0)=\delta^{ij}$. There exists $r_0=r_0(n,\lambda)$ and $C(n,\lambda)$ such that if $r_1\leq r_0$ and $N(0,r_1)\leq \Lambda$, then
 \begin{gather}
  N(y,r_1/2)\leq C\Lambda
 \end{gather}
for all $y\in (B_{r_1 }(0 ))$.
\end{theorem}

\subsection{Frequency pinching for elliptic equations: growth estimates}\label{sec_lipgrowth}
Generalizing the definition given for harmonic functions, we say that a solution $u$ to \eqref{eq_Lu} has frequency $\delta$-pinched on the scales $[r_2,r_1]$ around $x$ if
\begin{gather}
 \abs{N(x,s)-N(x,r_1)}\leq \delta\, \quad \forall\ s\in [r_2,r_1]\, .
\end{gather}
Given the almost monotonicity of $N$, a sufficient condition for pinching is
\begin{gather}
 N(x,r_1)-N(x,r_2)\leq \delta/2 \quad \text {  and  } \quad r_1\leq r_0(n,\lambda) \frac{\delta}{N(x,r_1)}\, .
\end{gather}

The aim of this section is to generalize in this context the properties enjoyed by harmonic functions with pinched frequency, with particular emphasis on the quantitative versions of these properties.

Throughout this section, we fix some $0<\delta<1/7$ and we will assume that $N(0,r_0)\leq \Lambda$ and that $r_1\leq \delta \min\cur{\frac{r_0}{\Lambda}, \frac{1}{C}}$.

\begin{lemma}\label{lemma_lipl2est}
Set for convenience $T=T^u_{x,r_1}$, and let $\ell$ be any real number, and suppose that for some $r_2\leq r_1$, $N(x,r_2)\geq \ell-\delta/2$ (or equivalently that $N(x,s)\geq \ell-\delta$ for all $s\in[r_2,r_1]$). Then 
\begin{gather}
 \int_{B_t} T^2dV \leq \frac{\omega_n}{n} (1+2\delta)\begin{cases}
                                             t^{n+2\ell-2\delta} & \text{ for } t\in\qua{\frac{r_2}{r_1},1}\, ,\\
                                             \ton{\frac{r_2}{r_1}}^{2\ell-2\delta-2} t^{n+2} & \text{ for } t\in\qua{0,\frac{r_2}{r_1}}\, ·
                                            \end{cases}
\end{gather}
\end{lemma}
\begin{proof}
The almost monotonicity of $N$ (or equivalently of $N^T$) and the fact that $N^T(x,0)\geq 1$ for any $x$ give the bound:
\begin{gather}
 N^T(0,r)\geq \begin{cases}
				      N^T(0,r_2)-\delta & \text{ for } r\in [r_2,r_1]\, ,\\
				      e^{-Cr}  & \text{ for } r\in [0,r_2]\, .
				     \end{cases}
\end{gather}
The lemma is an easy consequence of the $L^2$ estimates in Proposition \ref{prop_genfreq}. Indeed, we have
\begin{gather}
 \int_{B_t(0)} \abs{T(x)}^2 = \int_{0}^t ds \omega_n s^{n-1} \fint_{\partial B_s(0)} \abs{T(x)}^2 \leq \omega_n (1+Cr_1)\int_0^t ds s^{n-1} \exp\ton{-2\int_s ^1 \frac{N^T(0,r_1 s)} s ds }\, .
\end{gather}

\end{proof}

By standard elliptic estimates (see \cite[theorem 8.24]{GT}), we obtain the following corollary.
\begin{corollary}\label{cor_lipabsest}
 Under the assumptions of Lemma \ref{lemma_lipl2est}, we have that
 \begin{gather}
  \abs{T(x)} \leq C(n,\lambda)	2^{\ell/2}\begin{cases}
				 \abs x^{\ell-\delta} & \text{ for } \abs x \in [2^{-1/2}r_2/r_1,2^{-1/2}]\, ,\\
				 \ton{r_2\ r_1^{-1}}^{\ell-\delta-1} \abs x^{1} & \text{ for } \abs x \in [0,2^{-1/2}r_2/(r_1)]\, .
				\end{cases}
 \end{gather}
\end{corollary}

 \begin{proof}
 These estimates are an easy consequence of the standard elliptic estimates
  \begin{gather}
   \sup_{x\in B_{r}} \{\abs {T(x)}\}\leq C(n,\lambda) \norm T _{L^2(B_{2^{1/2}r})}r^{-n/2}\, .
  \end{gather}
 We refer the reader to \cite[theorem 8.24]{GT} for a proof of these estimates. 
\end{proof}

By $W^{2,p}$ elliptic estimates, we can easily use the previous Corollary to obtain $L^p$ estimates on the Laplacian of $u$.
\begin{lemma}\label{lemma_liplpest}
 For any fixed $p\in (1,\infty)$, under the assumptions of Lemma \ref{lemma_lipl2est}, the Laplacian of $T=T^u_{0,r_1}$ satisfies
 \begin{gather}\label{eq_lipdeltaest}
  \norm{\Delta T}_{L^p(B_t(0))}\leq C(n,\lambda,p)r_1 2^{\ell} \begin{cases}
				 \ton{r_2\ r_1^{-1}}^{\ell-\delta-1} t^{n/p}\   & \text{ for } t \in [0,r_2/(2r_1)]\\
				 t^{\ell-\delta-1} t^{n/p} & \text{ for } t \in [r_2/(2r_1),1/2]
				\end{cases}
 \end{gather}
\end{lemma}
\begin{proof}
 By $W^{2,p}$ elliptic estimates (see for example \cite[theorem 9.11]{GT}), we have
 \begin{gather}\label{eq_ellest}
  r^{2} \norm{\nabla^{2} T }_{L^p(B_r(0))} + r \norm{\nabla T }_{L^p(B_r(0))} \leq C(n,\lambda,p) \norm{T }_{L^p(B_{2^{1/2}r}(0))}
 \end{gather}
The estimates on the Laplacian are an easy consequence of the previous corollary and the fact that
 \begin{gather}\label{eq_estd}
  \Delta T = \Delta T - \tilde \L(T) = \ton{\tilde a^{ij}(0)-\tilde a^{ij}(x)}\partial_i \partial_j T + \ton{\tilde b^i(x) + \partial _j \tilde a^{ij}(x)} \partial_i T\, ,
 \end{gather}
where the coefficients $\tilde a^{ij}, \ \tilde b^i$ are the ones defined in \eqref{eq_LT}.
 
By the Lipschitz condition on $a^{ij}$ and the definition of $T$, we have that
\begin{gather}
 \abs{\tilde a^{ij}(0)-\tilde a^{ij}(x)}\leq  \lambda C(n,\lambda) r_1 \abs x\, .
\end{gather}
Moreover, the uniform bound on the coefficients $b^i$ and the Lipschitz bounds on $a^{ij}$ imply that
\begin{gather}
 \abs{\tilde b^i(x) + \partial _j \tilde a^{ij}(x) } \leq \lambda C(n,\lambda) r_1\, .
\end{gather}

Plugging in the estimates \eqref{eq_ellest} in \eqref{eq_estd} we obtain the result.
 
\end{proof}

\subsection{Frequency pinching and approximating harmonic functions}
Following \cite{han_sing}, we use the Green's kernel of the Laplacian in order to produce a harmonic function which approximates the solution $u$ under suitable pinching conditions.

\begin{theorem}\label{th_w}
 Let $u$ be a solution of \eqref{eq_Lu} with \eqref{e:coefficient_estimates}. Suppose that $r_1\leq r_0(n,\lambda)$, and that for some $r_2\leq r_1/4$, $N(x,r_2)\geq \ell$, where $\ell$ is any real number. Denote for simplicity $T=T^u_{x,r_1}$. Then there exists a function $w(x)$ such that for $\abs x \in [r_2/r_1,1/2]$ 
 \begin{gather}\label{eq_w0}
  \abs{w(x)}\leq C(n,\lambda) 8^\ell r_1 \abs x ^{\ell+1/3}\, , \quad \quad \abs{\nabla w (x)} \leq C(n,\lambda) 16^\ell r_1 \abs x ^{\ell-1+1/3}\, 
 \end{gather}
 and with $\Delta (w) = \Delta (T)$ on $B_{1/2}(0)$. Moreover, $w(0)=0$.
\end{theorem}
\begin{remark}
 As it will be clear from the proof, the function $w$ depends linearly on the function $T$, and in particular on its normalization. Recall that $T$ satisfies $\fint_{\partial B_1(0)} T^2 =1$.
\end{remark}

\begin{proof}
 We prove the theorem for $n\geq 3$. Similar estimates hold in the case $n=2$ by replacing $G(x,y)=c(n) \abs {x-y}^{2-n}$ with the Green's kernel in dimension $2$, i.e., $G(x,y)=c \log(\abs {x-y})$. 
 
\paragraph{Estimates on the Green's kernel.}

Let us recall some facts and estimates about the Green's kernel. Let $G(x,y)$ be the fundamental solution of the Laplace operator in $\R^n$, i.e., the Green's kernel. By standard theory, $G(x,y)= c(n) \abs {x-y} ^{2-n}$, where $c(n)$ is chosen in such a way that
 \begin{gather}
  \Delta_x G(x,y) = \delta(x-y)\, .
 \end{gather}
Fix some $y\neq 0$, and consider consider the function $G_y(x) = G(x,y)$ on the ball of radius $B_{\abs y}(0)$. $G_y$ is harmonic on this ball, and so we can write
\begin{gather}
 G_y(x) = \sum_k g_k(y) P_k(x)\, ,
\end{gather}
where $P_k(x)$ are homogeneous harmonic polynomials of degree $k$ normalized by $\fint_{\partial B_1(0)} P_k^2 =1$. For $r<\abs y$,
\begin{gather}
 \sum_k \qua{g_k(y) r^k}^2 = \fint_{\partial B_r(0)} G_y(x)^2\, .
\end{gather}
In particular, if we choose $r=2\abs y /3$, we obtain
\begin{gather}
 g_k(y) \leq \ton{\frac 3 2}^k \abs y ^{-k} \ton{\fint_{\partial B_{3\abs y /2}(0)} G_y(x)^2}^{1/2}\leq c(n) \ton{\frac 3 2}^k \abs y ^{-n+2-k}\, .
\end{gather}
Note also that the function
\begin{gather}
 S_{y,d}(x)=G(x,y)-\sum_{k=0}^d g_k(y) P_k(x)
\end{gather}
is a harmonic function with vanishing order $\geq d+1$ at the origin. Moreover, by the orthogonality properties of $P_k$, 
\begin{gather}
 \fint_{\partial B_{2\abs y /3}(0)} S_{y,d}(x)^2 \leq \fint_{\partial B_{2\abs y /3}(0)} G(x,y)^2 \leq \frac{c(n)}{\abs y^{2(n-2)}}\, .
\end{gather}
By growth conditions related to the frequency of $S$ and standard elliptic estimates, we have for $\abs x \leq \abs y /2$:
\begin{gather}\label{eq_d+1est}
 \abs{S_{y,d}(x)}\leq c(n) \ton{\fint_{\partial B_{4\abs x /3}(0)} S_{y,d}(x)^2}^{1/2} \leq c(n)2^{d+1} \frac{\abs x^{d+1}}{\abs y ^{d+1}}\ton{\fint_{\partial B_{2\abs y /3}(0)} S_{y,d}(x)^2}^{1/2}\leq c(n)2^{d+1} \frac{\abs x^{d+1}}{\abs y ^{n+d-1}}\, .
\end{gather}

\paragraph{Construction of $w$.}
Denote for convenience
\begin{gather}
 f(y)=\Delta T(y)\, .
\end{gather}
By Lemma \ref{lemma_liplpest}, $f\in L^p(B_1(0))$ for all $p<\infty$. Let $d$ be the closest integer to $\ell$ (so that $\abs{\ell-d}\leq 1/2$), and define the function $w$ by
\begin{gather}
 w(x) = \int_{\abs y \leq 1} \qua{G(x-y)-G(-y)} f(y)dy- \sum_{k=1}^d \int_{r_2/r_1\leq\abs y \leq 1} g_k(y) P_k(x) f(y)dy\, .
\end{gather}
Assuming that both integrals converge, the sum on the rhs is a harmonic polynomial of degree $\leq d$, and, by the properties of the Green's kernel, $\Delta (w) = \Delta (T)$. In order to finish the proof, we only need the $C^0$ estimates on $w$.

Rewrite $w$ as
\begin{gather}
 w(x)= \int_{\abs y < r_2/r_1} \ton{G(x-y)-G(-y)} f(y)dy + \int_{r_2/r_1\leq \abs y \leq 1} \qua{G(x-y) -\sum_{k=0}^d g_k(y)P_k(x) }f(y)dy\, .
\end{gather}
Fix $r_2/r_1<\abs x<1/2$, and split the integral in the following fashion:
\begin{gather}
 I_1 = \int_{\abs y \leq 2 \abs x} \ton{G(x-y)-G(-y)} f(y) dy\, ,\\
 I_2 = -\int_{r_2/r_1\leq \abs y \leq 2 \abs x} \sum_{k=1}^d g_k(y)P_k(x) f(y) dy\, ,\\
 I_3 = \int_{2\abs x \leq \abs y \leq 1} \qua{G(x-y) -\sum_{k=0}^d g_k(y)P_k(x) } f(y) dy\, .
\end{gather}
Let $p=n+1$ and let $p'=(n+1)/n$ be its conjugate exponent. Using the $L^p$ estimates on $f(y)$ we obtain
\begin{gather}
 \abs {I_1} \leq c(n)\abs x \ton{\int_{\abs y \leq 2 \abs x} \frac{1}{\abs{x-y}^{(n-1)p'}} dy }^{1/p'}\ton{\int_{\abs y \leq 2 \abs x} f(y)^p dy }^{1/p}\leq\\
 \leq c(n)\abs x \ton{\int_{\abs z \leq 3 \abs x} \frac{1}{\abs{z}^{(n-1)p'}} dz }^{1/p'} C(n,\lambda)2^{2\ell} r_1 {\abs x} ^{\ell-\delta-1} \abs x ^{n/(n+1)} \leq C(n,\lambda)r_1 2^{2\ell}\abs x^{\ell+1-\delta}\, .
\end{gather}
In order to estimate $I_2$ we write
\begin{gather}
 \abs {I_2} \leq c(n)\sum_{k=1}^d \ton{\frac 3 2}^k\abs{P_k(x)} \sum_{i=0}^{A} \int_{2^{-i}\abs x \leq \abs y \leq 2^{1-i}\abs x} \frac{\abs {f(y)}}{\abs y ^{n+k-2} }dy \leq\\
 \leq c(n)\sum_{k=1}^d 2^k \abs x^k \sum_{i=0}^{A} \ton{\int_{2^{-i}\abs x \leq \abs y \leq 2^{1-i}\abs x} \frac{1}{\abs y ^{(n+k-2)p'} }dy }^{1/p'} \ton{\int_{\abs y \leq 2^{1-i}\abs x} \abs{f(y)}^p dy }^{1/p}\, ,
\end{gather}
where $A$ is the smallest integer greater or equal to $\log_2 (r_1\abs x) - \log_2 (r_2)$.

We use the estimates in Lemma \ref{lemma_liplpest} in order to obtain
\begin{gather}
 \abs{I_2}\leq C(n,\lambda)2^{d+\ell}r_1  \sum_{k=1}^d \abs x ^k \sum_{i=0}^{A} \ton{\frac{\abs x}{2^i}}^{1-k}\ton{\frac{\abs x}{2^{i-1}}}^{\ell-\delta}\leq \\
 \leq C(n,\lambda)2^{d+2\ell} r_1 \abs x^{\ell+1-\delta}\sum_{k=1}^d \sum_{i=0}^{\infty} \ton{2^{k+\delta-1-\ell}}^i\leq C(n,\lambda)d2^{d+2\ell}r_1  \abs x^{\ell+1-\delta}\, .
\end{gather}
Note that the bounds on the infinite series are a direct consequence of $\abs{d-\ell}\leq 1/2$, which implies $2^{k+\delta-1-\ell}\leq 2^{-1/3}$. 

With a similar technique, we can estimate $I_3$. Indeed, by \eqref{eq_d+1est} we have
\begin{gather}
 \abs{I_3} \leq C(n) 2^d\abs x^{d+1} \sum_{i=0}^A \int_{2^{-1-i} \leq \abs y \leq 2^{-i} } \frac{\abs{f(y)}}{\abs y ^{n+d-1}} dy\, ,
\end{gather}
where $A$ is the first integer $\geq -\log_2(\abs x)-1$. Thus we have
\begin{gather}
 \abs{I_3} \leq C(n,\lambda)2^{d+\ell}r_1\abs x^{d+1} \sum_{i=0}^A \ton{\int_{2^{-1-i} \leq \abs y \leq 2^{-i} } \frac{1}{\abs y ^{(n+d-1)p'}} dy}^{1/p'}\ton{\int_{\abs y \leq 2^{-i} } \abs{f(y)}^p dy}^{1/p}\leq\\
 \notag \leq C(n,\lambda)2^{d+\ell}r_1\abs x^{d+1} \sum_{i=0}^A 2^{i(d+\delta-\ell)}\, .
\end{gather}
The last sum can be estimates via integrals. In particular, it is easily seen that
\begin{gather}
 \sum_{i=1}^N 2^{\eta i} \leq \int_0 ^ {-\log_2 (\abs x)} 2^{\eta s} ds \leq  \frac{1}{\eta \ln(2)}\ton{\abs x ^{-\eta}-1}\leq c \abs x ^{-\eta} (-\log(\abs x))\, .
\end{gather}

In the end we have
\begin{gather}
 \abs{I_3} \leq C(n,\lambda) r_1 2^{d+\ell} \abs x^{\ell+1-\delta} (-\log(\abs x))\, .
\end{gather}
Since $\abs {\ell-d}\leq 1/2$, we have proved the $C^0$ estimates for $\abs x \in \qua{\frac {r_2}{r_1}, \frac 1 2}$.

With analogous estimates, it is easy to prove that $\abs{w(x)} = O(\abs x)$ as $x\to 0$.

The estimate on the gradient of $w$ is a simple corollary of the elliptic estimate valid for any $p>n$ (see for example \cite[Theorem 9.11]{GT})
\begin{gather}
 \abs{\nabla w(x)}\leq C(n,\lambda,p) \abs x ^{-1} \ton{\norm{w}_{C^0(B_{2\abs x}(0)} + \abs x ^{2-n/p} \norm{\Delta w}_{L^p(B_{2\abs x}(0)} }\, .
\end{gather}
\end{proof}

As a corollary, we obtain the existence of an approximating harmonic function $h$ for $u$. For convenience of notation, we state the theorem with the function $T=T^u_{x,r}$.
\begin{corollary}\label{cor_ell1}
 Let $u$ be a solution of \eqref{eq_Lu} with \eqref{e:coefficient_estimates}, and $0<\delta<1/7$. Suppose that $r_1\leq r_0(n,\lambda)$, and that for some $r_2\leq r_1/4$, $u$ has frequency $\delta$-pinched on $[r_2,r_1]$, i.e., $\abs{N(x,s)-N(x,r_1)}\leq \delta$ for all $s\in [r_2,r_1]$. There exists a constant $C(n,\lambda)$ such that if $r_1\leq \delta \ton{C(n,\lambda) 16^\Lambda}^{-1}$, then there exists a harmonic function $h$ such that for $r_2/r_1\leq \abs y \leq 1/4$:
 \begin{gather}
  \abs{h(y)-T(y)}^2\leq \delta^2 \fint_{\partial B_{\abs y}(0)} T^2 dS \, \quad \text{ and }\quad \quad  \abs{N^h(0,\abs y)- N^T(0,\abs y)} \leq \delta\, .
 \end{gather}
\end{corollary}
\begin{proof}
 Define the function $h(y)=T(y)-w(y)$, which is evidently a harmonic function on $B_{1/2}(0)$, and let $\ell=N(x,r_2)$. By the estimates on $\abs {w(x)}$, the pinching of the frequency and the $L^2$ estimates in \ref{prop_genfreq}, we obtain immediately the first inequality. Indeed
 \begin{gather}\label{eq_wT}
  \abs{w(y)}\leq C(n,\lambda) 8^\ell r_1 \abs{y} ^{\ell+1/3} \ton{\fint_{\partial B_1(0)} T^2}^{1/2}\leq C(n,\lambda) 8^\ell r_1 \ton{\fint_{\partial B_{\abs y}(0)} T^2}^{1/2} (1+C(n,\lambda) r_1) \abs y ^{1/3-\delta}\, .
 \end{gather}
As for the frequency of $h$, consider the ratio between the frequency of $h$ and the generalized frequency of $T$
\begin{gather}
 \frac{N^h(0,r)}{N^T(0,r)} = \frac{D^T(r)}{I^T(r)}\frac{ \int_{B_r(0)} \abs{\nabla h}^2 dV }{ \int_{B_r(0)} \norm{\nabla T}_g^2 dV_g  } \frac{\int_{\partial B_r(0)} \abs{T}^2 dS_g }{\int_{\partial B_r(0)} \abs{h}^2 dS }\, ,
\end{gather}
where we have used the notation introduced in Proposition \ref{prop_genfreq}. Since $\abs{g^{ij}(y)-\delta^{ij}} \leq C(n,\lambda) \abs y$, we can easily replace all the integrals wrt $g$ with standard Euclidean integrals up to some small multiplicative constant. More precisely
\begin{gather}
 \abs{\frac{ \int_{B_r(0)} \abs{\nabla T}^2 dV }{ \int_{B_r(0)} \norm{\nabla T}_g^2 dV_g  }  -1} \leq C(n,\lambda) r_1\, , \quad \quad \abs{\frac{\int_{\partial B_r(0)} \abs{T}^2 dS_g }{\int_{\partial B_r(0)} \abs{T}^2 dS }-1}\leq C(n,\lambda)r_1\, .
\end{gather}
As for the ratio of the Dirichlet integrals, by definition
\begin{gather}
\int_{B_r(0)} \abs{\nabla h}^2 = \int_{B_r(0)} \abs{\nabla T}^2 + \int_{B_r(0)} \abs{\nabla w}^2 -2 \int_{B_r(0)} \ps{\nabla T}{\nabla w}\, .
\end{gather}
The $C^1$ estimate in the previous theorem and the $L^2$ growth estimates in \ref{prop_genfreq} give for all $r_2/r_1\leq r \leq 1/4$
\begin{gather}
 \fint_{B_r(0)} \abs{\nabla w}^2 \leq \ton{C16^\ell r_1 r ^{-1-\delta+1/3} }^2 \fint_{\partial B_r(0)} \abs T^2 \leq (1+Cr) \ton{C16^\ell r_1 r ^{-1-\delta+1/3} }^2 \frac{c(n) r^2}{\ell-\delta} \fint_{B_r(0)} \abs{\nabla T}^2\, .
\end{gather}
In a similar way, one can estimate also the second fraction.
\end{proof}

Corollary \ref{cor_ell1} also immediately leads to a generalization of Lemma \ref{lemma_Npinch}:
\begin{lemma}\label{lemma_ellNpinch}
 Let $u$ be a solution to \eqref{eq_Lu} with \eqref{e:coefficient_estimates}. There exists constants $r_0(n,\lambda)$ and $C(n,\lambda)$ such that if $N(0,r_0)\leq \Lambda$, $r_1\leq \delta C(n,\lambda)^{-\Lambda}$ and for some integer $d$, $N(0,r_1)\in [d+\delta,d+1-\delta]$, then $N(0,e^2 r_1/4)\leq N(0,r_1)-\delta/10$. Moreover, if $N(0,r_1)\leq d+1-\delta$, then $N\ton{0,\delta^{C(n,\lambda)} r_1}\leq d+\delta$. 
\end{lemma}
\begin{proof}
 This lemma is an easy consequence of the previous corollary and of Lemma \ref{lemma_Npinch} applied to the approximating harmonic function for $T$. 
\end{proof}

By applying Corollary \ref{cor_ell1} and the above Lemma, and combining with Theorem \ref{t:eff_tan_con_uniq_harm}, we immediately obtain our main result for the subsection, namely the {\it effective} tangent cone uniqueness statement.

\begin{theorem}\label{t:eff_tan_con_uniq}
 Let $u$ be a solution of \eqref{eq_Lu} with \eqref{e:coefficient_estimates}, $0<\epsilon<1/7$ and $N(x,r_1)\leq \ell+1$.  If $100 r_2\leq r_1\leq r_0(n,\lambda)\epsilon \ton{C(n,\lambda) 16^\ell}^{-1}$, and if $u$ has frequency $\epsilon$-pinched on $[r_2,r_1]$, i.e., $\abs{N(x,r_2)-N(x,r_1)}\leq \epsilon$, then there exists a {\it unique} homogeneous harmonic polynomial $P_d$ such that:
\begin{enumerate}
  \item[(i)] There exists an integer $d$ such that for all $t\in (r_2,r_1)$, $\abs{ N(0,t)-d}\leq 6\epsilon$,
  \item[(ii)] For all $t\in (4 r_2,r_1/16)$
   \begin{gather}
    \fint_{\partial B_1(0)} \abs{T_{0,t}^u - P_d}^2 \leq 14\epsilon\, ,
   \end{gather}
 \end{enumerate}
\end{theorem}
\begin{proof}
  Let $h$ be the harmonic approximation of $T^u_{0,r_1}$ built in Corollary \ref{cor_ell1}. This harmonic function satisfies
 \begin{gather}
  N^h(0,1/4)- N^h(0,r_2/r_1) \leq 2\epsilon\, .
 \end{gather}
As a consequence, we can apply Theorem \ref{t:eff_tan_con_uniq_harm} to $h$. The approximation properties proved in \ref{cor_ell1} immediately imply the thesis.
\end{proof}

Another application of Theorem \ref{th_w} leads us to the following:

\begin{lemma}\label{lemma_ell3/2}
 For any $\delta>0$, there exists $r_0(n,\lambda)$ such that if for some $r_1\leq \delta r_0$, $N(x,r_1)\leq 3/2$, then
 \begin{gather}
  \abs{\nabla T^u_{x,r_1/8} (x) }^2 \geq \frac{n}{2} (1+\delta) \, .
 \end{gather}
\end{lemma}
\begin{proof}
First note that if $h$ is harmonic with $N(x,r)<\frac{3}{2}$, then we have that
\begin{align}
|\nabla h|^2(x)\geq \frac{n}{2}\fint_{\partial B_r(x)}(u-u(x))^2\, .
\end{align}

Now in general let $h$ be any suitable harmonic $C^1$ approximation of $T$.  One possible way to obtain the approximation $h$ is to adapt the proof of Corollary \ref{cor_ell1}. Indeed, even though we are not assuming any pinching of the frequency, we can use theorem \ref{th_w} with $N=1$ and obtain the existence of a function $w$ such that $h=T-w$ is a good $C^{1,\alpha}$ approximation of $T$.  Then applying the above to $h$ and $r_0$ sufficiently small gives us our Lemma.

Although it is not necessary for the scope of this Lemma, it is worth noticing that equation \eqref{eq_wT} is still valid also in this context. Indeed, by the previous corollary, there exists a constant $c(n,\lambda)<1/2$ such that
\begin{gather}
 N(x,c^k r_1)\leq 1+2^{-k-1}\, .
\end{gather}
This and the $L^2$ growth conditions in Proposition \ref{prop_genfreq} imply that there exists a constant $C(n,\lambda)$ (not close to $1$, but still a constant) such that
\begin{gather}
 \fint_{\partial B_r(0)} \abs T^2 \leq C r^2 \fint_{\partial B_1(0)}\abs T^2\, .
\end{gather}
Using this estimate, it is easy to check that the proof of Corollary \ref{cor_ell1} carries over also in this context.
\end{proof}

\subsection{Almost cone splitting}\label{ss:almost_conesplitting_ell}
The aim of this section is to obtain a generalization of Lemma \ref{lemma_epsd} for elliptic equations which will allow us to extend Corollary \ref{cor_alcone} also in this context. The basic idea is quite simple: if we pick a solution $u$ to \eqref{eq_Lu} such that on a small enough ball its generalized frequency is pinched around some $x$, then by Theorem \ref{th_w} we obtain an approximating harmonic function $h$ with standard frequency pinched as well. All we need to prove is that if the frequency of $u$ is pinched also at some other point $x'$, then the frequency of $h$ is pinched as well around $x'$.

Throughout this section, we will assume that the solution $u$ of \eqref{eq_Lu} has frequency $\delta$-pinched on $[r_2,r_1]$, with $r_1\leq r_0(n,\lambda)$, $r_2\leq r_1/32$ and $N(0,r_1)\leq \Lambda$. Moreover, we fix the notation $T=T^u_{0,r_1}$ and $N=N(0,r_1)$.

\begin{proposition}\label{prop_ell_y_pinch}
 Suppose that $N(0,r_0)\leq \Lambda$, $r_1\leq \delta C(n,\lambda)^{-\Lambda}$ and $N(0,r_1)-N(0,r_1/32)\leq \delta/2$. Let $T=T^u_{0,r_1}$ and $h$ be its harmonic approximation, as in Corollary \ref{cor_ell1}. There exists a $\beta(n)>0$ such that if $\bar x\in B_{\beta(n)}(0)$, then
 \begin{gather}
  \abs{{N^T(\bar x,1/8) } - {N^h(\bar x,1/8)} }\leq \delta/2\, .
 \end{gather}
In particular, if for some $r\leq 1/16$ we have $N^T(\bar x,1)-N^T(\bar x,r/2)\leq \delta/2$, then $N^h(\bar x,1/8)- N^h(\bar x,r)<2\delta$.
\end{proposition}
\begin{proof}
Although philosophically very similar to the proof of Lemma \ref{lemma_ellNpinch}, the proof of this proposition has some minor (albeit annoying) technical details to be addressed. In particular, the fact that $a^{ij}(\bar x)\neq a^{ij}(0)$ will force us to deal with different ellipsoids instead of balls.
 
Consider the function $T=T^u_{0,r_1}$, which solves \eqref{eq_LT}.
By the invariance properties of the frequency and the Lipschitz assumption on $a^{ij}$, 
\begin{gather}
\abs{N^T(\bar x,s) - \frac{s \int_{E(\bar x,s)} \abs{\nabla T}^2 dV }{\int_{\partial E(\bar x,s)} \abs{T-T(\bar x)}^2 dS } } \leq C(n,\lambda) r_1 \Lambda\, , \quad \text{ where } \quad E(y,s) =\cur{z \ \ s.t. \ \ \tilde a_{ij}|_y (z-y)^i (z-j)^j < s^2  }\, .
\end{gather}
Since $\tilde a^{ij}(0)=\delta^{ij}$ by definition, there exists $\epsilon\leq C(n,\lambda) r_1$ such that for all $y\in B_1(0)$
\begin{gather}
 B_{(1-\epsilon) s}(y)\subset E(y,s)\subset B_{(1+\epsilon) s}(y)\, .
\end{gather}
Moreover, by standard elliptic estimates and the $L^2$ growth condition in Proposition \ref{prop_genfreq}, we can bound the gradient of $T$ by
\begin{gather}
 \norm{\nabla T}_{C^0(B_{1/4}(\bar x))}^2\leq C(n,\lambda) \fint_{\partial E(\bar x,1)} \abs{T-T(\bar x)}^2 \leq C(n,\lambda) 8^\Lambda \fint_{\partial E(\bar x,1/8)} \abs{T-T(\bar x)}^2     \, .
\end{gather}
We can use this bound to estimate the $L^2(\partial E(\bar x,1/8))$ norm. Indeed
\begin{gather}
 \abs{\frac{\int_{\partial B_{1/8}(\bar x)} \abs{T-T(\bar x)}^2 dS }{\int_{\partial E(\bar x,1/8)} \abs{T-T(\bar x)}^2 dS } -1 }\leq C(n,\lambda) \epsilon \frac{\norm{\nabla T}_{C^0(B_{1/4}(\bar x))} }{\ton{\int_{\partial E(\bar x,1)} \abs{T-T(\bar x)}^2 dS }^{1/2} } \leq C(n,\lambda)^\Lambda r_1\, .
\end{gather}
By using the estimates on $w=T-h$ from Theorem \ref{th_w}, we obtain
\begin{gather}
 \abs{{\fint_{\partial B_{1/8}(\bar x)} \abs{h-h(\bar x)}^2 } - {\fint_{\partial B_{1/8}(\bar x)} \abs{T-T(\bar x)}^2  }}\leq C(n,\lambda)^\Lambda r_1 \fint_{\partial B_1(0)} \abs{h}^2\leq C(n,\lambda)^\Lambda r_1 \fint_{\partial B_{1/16}(0)} \abs{h}^2\, .
\end{gather}
As shown in the proof of Theorem \ref{th_cN}, there exists $\beta(n)<1/8$ such that for $\abs {\bar x}\leq \beta(n)$:
\begin{gather}
 h(\bar x)^2\leq \frac 1 2 \fint_{\partial B_{1/8}(\bar x)} \abs {h}^2 \, \quad \Longrightarrow \quad \, \fint_{\partial B_{1/8}(\bar x)} \abs {h}^2 \leq 2\fint_{\partial B_{1/8}(\bar x)} \abs {h-h(\bar x)}^2 \, .
\end{gather}
Moreover simple geometric considerations and the growth estimates related to $N$ lead to
\begin{gather}
 \fint_{\partial B_{1/16}(0)} \abs h ^2 \leq c(n)\Lambda \fint_{\partial B_{1/8}(\bar x)} \abs {h}^2 \, .
\end{gather}
Putting together these estimates, we obtain
\begin{gather}
 \abs{\frac{\int_{\partial B_{1/8}(\bar x)} \abs{T-T(\bar x)}^2  }{\int_{\partial B_{1/8}(\bar x)} \abs{h-h(\bar x)}^2 }  -1}\leq C(n,\lambda)^\Lambda r_1
\end{gather}
Arguing as in the proof of Corollary \ref{cor_ell1}, we conclude
\begin{gather}
 \abs{\frac{N^T(\bar x,1/8) }{N^h(\bar x,1/8)} -1}\leq C(n,\lambda)^\Lambda r_1\, 
\end{gather}
as desired.
\end{proof}

As a corollary of this proposition, we can easily prove a generalization of Lemma \ref{lemma_epsd}. 
\begin{corollary}\label{cor_epsd}
 Fix some $0<\epsilon<\epsilon_0(n,\lambda)<<1$, and suppose that $N(x,r_0)\leq \Lambda$ for all $x\in B_{r_0}(0)$. Let $r_1\leq r_0 C(n,\lambda)^{-\Lambda}\epsilon$, and assume that for some integer $d$ we have
 \begin{itemize}
  \item $N(0,r_1)-N(0,\chi \beta(n) r_1)\leq \epsilon$, with $\chi = c(n,\lambda)\Lambda^{-1}$ and $\abs{N(0,r_1)-d}\leq \epsilon$,
  \item there exists $\tilde x \in B_{\chi r_1}(0)$ such that $N(\tilde x,r_1)-N(\tilde x,\chi \beta(n) r_1 )\leq \epsilon$, with $\abs{N(\tilde x,r_1)-d}\leq 1/3$.
 \end{itemize}
 Let $h$ be the harmonic approximation of $T_{0,\chi r_1}$ and $h'$ be the harmonic approximation of $T_{\tilde x,\chi r_1}$. Set also $\bar x = (\chi r_1)^{-1} Q_0(\tilde x)$. After rotating we may assume without loss that $\bar x = (t ,0,\cdots,0)$, with $\abs t\leq \beta(n)$. If $\epsilon \leq \epsilon_0(n)$, then $u$ is almost $\bar x$ invariant, in the sense that:
 \begin{enumerate}
  \item The $d$-th degree part in the Taylor expansion of $h$ is dominant. In particular if $h=\sum_k a_k P_k$, then
  \begin{gather}
   a_d^2\geq (1-12\epsilon) \norm{h}_{L^2(\partial B_1 (0))}^2\, .
  \end{gather}
  Similarly for $h'$.
  \item The $d$-th degree part in the Taylor expansion of $h$ is almost constant. In detail:
	  \begin{gather}
           \fint_{\partial B_1(0)} \abs{a_d(0) P_{d,0}(y) - a_d(\bar x) P_{d,\bar x}(y) }^2 dy \leq C(n)\epsilon t^2 \fint_{\partial B_1(0)} \abs{a_d(0) P_{d,0}(y) }^2 dy  =C(n)\epsilon t^2 a_d(0)^2\, .
	  \end{gather}
	  Similarly for $h'$.
  \item The two functions $h$ and $h'$ are almost the same. In particular
	  \begin{gather}
           \fint_{\partial B_1(0)} \abs{h(y) - h'(y) }^2 dy \leq C(n)\epsilon t^2\fint_{\partial B_1(0)} \abs{h(y)}^2 dy \, ;
	  \end{gather}
  \item The $\bar x$ derivative of $P_d$ is almost zero, more precisely
	   \begin{gather}
	    \norm{\partial_1 P_{d,0}}\leq C(n) t^{-1} \sqrt \epsilon \norm{\nabla P_{d,0}}=C(n) t^{-1} \sqrt \epsilon \sqrt{d(2d+n-2)}\norm{P_{d,0}}\, .
	   \end{gather}
 \end{enumerate}
 
\end{corollary}

\begin{proof}
 This Corollary is an easy consequence of the previous proposition and Lemma \ref{lemma_epsd}. Note that point $1$ is the content of Lemma \ref{lemma_ellNpinch}.
 
 Set for convenience $T=T_{0,r_2}$ and $T'=T_{\bar x,r_2}$, and set $h$ and $h'$ to be their harmonic approximations. By Corollary \ref{cor_ell1}, we can estimate
 \begin{gather}
  N^h(0,r_1/(2r_2))- N^h(0,1/5) \leq 2\epsilon\, .
 \end{gather}
 Let $\bar x = r_2^{-1}Q_0(\tilde x)$. In other words, $\bar x$ is the point in the domain of $T$ corresponding to $\tilde x$ in the domain of $u$. Note that $\bar x\in B_{\beta(n)}(0)$. Proposition \ref{prop_ell_y_pinch} guarantees that 
 \begin{gather}
  N^h(\bar x,r_1/(2r_2))- N^h(\bar x,1/5) \leq 2\epsilon\, .
 \end{gather}
As a consequence, we can apply Theorem \ref{t:eff_tan_con_uniq_harm} to $h$ and prove point $1$, and by Lemma \ref{lemma_epsd} applied to $h$, we immediately prove points $2$ and $4$.

As for point $3$, Lemma \ref{lemma_epsd} ensures that
 \begin{gather}
  \fint_{\partial B_1(0)} \abs{h(y) - h(\bar x +y) }^2 dy \leq C(n)\epsilon t^2\fint_{\partial B_1(0)} \abs{h(y)}^2 dy \, .
 \end{gather}
Arguing as in the proof of \ref{prop_ell_y_pinch}, we can easily prove that
 \begin{gather}
  \fint_{\partial B_1(0)} \abs{h(\bar x + y) - h'(y) }^2 dy \leq C(n)\epsilon t^2\fint_{\partial B_1(0)} \abs{h(y)}^2 dy \, ,
 \end{gather}
and this concludes the proof of point $3$.
\end{proof}

In turn, we are now in a position to generalize Corollary \ref{cor_alcone} to generic elliptic equations.
\begin{corollary}\label{cor_alcone_ell}\label{c:cone_splitting}
Fix some $0<\tau<1$ and some integer $d\geq 2$, and suppose that $N(x,r_0)\leq \Lambda$ for all $x\in B_{r_0}(0)$. Let $r_1\leq r_0 C(n,\lambda)^{-\Lambda} \epsilon$ and set $\cV$ to be set of points $x\in B_{r_0}(0)$ such that 
\begin{gather}
N(x,r_1)-N(x,r_x)\leq \epsilon\, \quad \text{ with } \quad r_x\leq c(n,\lambda)\Lambda^{-1} r_1 \quad \text { and } \quad \abs{\bar N(x,r_1)-d}\leq 1/3\, ,
\end{gather}
There exists $\epsilon_0(n,\tau)$ such that if $\epsilon<\epsilon_0$, then there exists a subspace $V$ of dimension at most $n-2$ such that for all $x\in \cV$
\begin{gather}
 \cV\cap B_{c(n,\lambda)\Lambda^{-1} r_1}(0)\subset x+B_\tau (V)\, .
\end{gather}
Note that the subspace $V$ is independent of $x$ and $r_x$
\end{corollary}

\subsection{Critical Points and Symmetry for Elliptic Equations}\label{ss:critical_symmetry_elliptic}

In order to complete the generalization to elliptic equations, we only have to show that given enough symmetry there cannot exist critical points away from an $n-2$-plane.  See Section \ref{ss:symmetric_criticalpoints} for the corresponding statements for harmonic functions.

Since the harmonic approximation given in Theorem \ref{th_w} is a $C^{1,\alpha}$ approximation for the function $T$, it is evident that Propositions \ref{prop_crit2} and \ref{prop_n-2} remain valid for elliptic equations under minor modifications. 

As for the effective critical points, we need to show that if $N^h(x,r)\geq 3/2$, then also $N^T(x,r)\geq 3/2$ (or something similar). One may think that as $r$ gets smaller, $h$ must be closer and closer to $T$ in order for this statement to be true. However, we will see that since we are concerned only with points of frequency $3/2$, and not points with generic frequency, the size of $r$ does not matter. Indeed, in some sense the condition $N^h(x,r)\geq 3/2$ is a $C^1$ condition on the function $h$, it states that the gradient of $h$ does not vanish in a quantitative way.

\begin{proposition}
Under the hypothesis and notation of Proposition \ref{prop_ell_y_pinch}, suppose that $x\in B_1(0)$ is such that $\abs{\nabla h(x)}^2\geq \alpha^2 \fint_{\partial B_1(0)} \abs h^2>0$. 
  If $r_1\leq C(n,\lambda)^{-\Lambda} \alpha$, then $r_c^T\geq c(n,\lambda) r_c^h $.
\end{proposition}
\begin{proof}
 By the arguments in the proof of Proposition \ref{prop_ell_y_pinch}, for $r_1\leq r_0(n,\lambda)^\Lambda$, $x\in B_1(0)$ and $r\leq 1$ we have
 \begin{gather}
  \abs{\frac{N^T(x,r)}{N^T(x,r)}-1}\leq \frac 1 {10}\, .
 \end{gather}
Since $h$ is harmonic, by the properties of its $L^2$ expansion for any $x$ and $r$
\begin{gather}
 \fint_{\partial B_r(x)} \abs{h-h(x)}^2 \geq r^2 \abs{\nabla h(x)}^2\, , \quad \quad \abs{\nabla h(x)}^2 \leq \fint_{B_r(x)} \abs{\nabla h}^2\, .
\end{gather}
This implies that
\begin{gather}
 \frac{\fint_{\partial B_r(x)} \abs{w-w(x)}^2 }{\fint_{\partial B_r(x)} \abs{h-h(x)}^2 }\leq \frac{\norm{\nabla w}_{\infty, B_r(x)}^2 }{\abs{\nabla h(x)}^2}\leq C(n,\lambda)^{\Lambda} r_1^2 \alpha^{-2}\, ,\\
 \frac{\fint_{B_r(x)} \abs{\nabla w}^2}{\fint_{B_r(x)} \abs{\nabla h}^2}\leq \frac{\norm{\nabla w}_{\infty, B_r(x)} ^2 }{\alpha^2\fint_{\partial B_1(0)} h^2} \leq C(n,\lambda)^\Lambda r_1^2 \alpha^{-2}\, .
\end{gather}
 With these estimates we can conclude that, for $r_0(n,\lambda)$ is sufficiently small,
 \begin{gather}
  \abs{\frac{N^T(x,r)}{N^h(x,r)} -1} \leq 1/7\, .
 \end{gather}
This in particular implies $N^T(x,r_c(x))\leq 1+5/7<2$. The thesis follows from Lemma \ref{lemma_ellNpinch}.
\end{proof}

As a Corollary, we obtain immediately the following generalizations of Propositions \ref{prop_h2effcrit} and \ref{prop_heffcrit}.
\begin{corollary}\label{cor_ell2effcrit}
 If $u:B_1(0)\subset \R^2\to \R$ is a  solution to \eqref{eq_Lu} such that for all $x\in B_{1/2}(0)$ $N(x,r_0)\leq \Lambda$, there exists $\epsilon_0(n,\lambda)$ such that $N(x,r_1)-N(x,r_1/10)\leq \epsilon_0$ implies $r_c(y)\geq c(n,\lambda)/\Lambda r_1$ for all $y\in \partial B_{r_1/5}(x)$.
\end{corollary}

\begin{corollary}\label{cor_elleffcrit}
  Let $u:B_1(0)\subset \R^n\to \R$ be a solution to \eqref{eq_Lu} such that $N(x,r_0)\leq \Lambda$ for all $x\in B_{1/2}(0)$ and $\tau<1$. Suppose also that $\epsilon\leq \epsilon_0(n,\lambda)\tau^\Lambda$ and $r_1\leq c(n,\lambda)^\Lambda \epsilon$.  Let $h$ be the approximating harmonic function for $T_{x,r}$, and suppose that
 \begin{gather}
  h=Q_d + \sum_k a_k P_k\, ,
 \end{gather}
where $P_k$ are normalized hhP's of degree $k$, and $Q_d$ is a normalized hhP of degree $d$ invariant wrt the $n-2$ dimensional plane $V$. If 
\begin{gather}
 \sum_k \abs{a_k}^2 e^{2\abs {k-d}}\leq \epsilon\, ,
\end{gather}
then for all $x\in B_1(0)\setminus B_\tau(V)$, $r_c^T(x)\geq c(n) \tau^{d}$. 
\end{corollary}

\subsection{Volume estimates for the effective critical set for elliptic equations}
The approximation theorem and the generalization proved in the previous sections allow us to extend the proof of the volume estimates also for elliptic equations. The proofs are the essentially same as in the case of harmonic functions, it is sufficient to replace the propositions and lemmas for harmonic function with their generalizations for generic elliptic equations proved in this section.

In particular, we need to
\begin{enumerate}
 \item replace the uniform control given by Theorem \ref{th_cN} with Theorem \ref{th_ellcN},
 \item replace the statement about frequency drops in Lemma \ref{lemma_Npinch} with Lemma \ref{lemma_ellNpinch},
 \item replace the tangent cone uniqueness of Theorem \ref{t:eff_tan_con_uniq_harm} with the one in Theorem \ref{t:eff_tan_con_uniq},
 \item replace the almost cone splitting in \ref{cor_alcone} with \ref{cor_alcone_ell},
 \item replace the $\epsilon$-regularity theorems in Propositions \ref{prop_h2effcrit} and \ref{prop_heffcrit} with Corollaries \ref{cor_ell2effcrit} and \ref{cor_elleffcrit} respectively.
\end{enumerate}

With these modifications, it is easy to prove that
\begin{theorem}
Let $R\leq r_0(n,\lambda)$ and consider $u:B_R(0)\subset \R^n\to \R$ be a solution to \eqref{eq_Lu} with \eqref{e:coefficient_estimates} such that $N(0,R)\leq \Lambda$. There exists a constant $C(n,\lambda)$ such that
 \begin{gather}
  \Vol\ton{B_r (\cS_r(u))\cap B_{R/2}(0)} \leq C(n,\lambda)^{\Lambda ^2} \ton{r/R}^2\, .
 \end{gather}
\end{theorem}
\begin{proof}
 The proof is almost identical to the one carried out for harmonic functions in \ref{th_vol_harm}.
 
 However, there is one point which needs a little attention. In many of the statements for elliptic equations, we require that $r_1\leq C(n,\lambda)^{-\Lambda} \epsilon$. In order to satisfy this request, we cover the ball $B_{r_0}(0)$, with a minimal covering of balls of radius $r_1=C(n,\lambda)^{-\Lambda} \epsilon$. It is evident that the number of these balls can be bounded above by $C(n,\lambda)^\Lambda \epsilon^{-1}\leq C(n,\Lambda)^{\Lambda^2}$.
 
 On every ball of this covering, we can apply all the statement proved for elliptic equations, and, by the same proof as in the harmonic case, obtain the desired estimate. Given the bound on the number of such balls, the estimate remains essentially unchanged also on $B_R(0)$.

\end{proof}

As it is easily seen, also the improved estimate in dimension $2$ can be generalized in a similar way.
\begin{theorem}
 Let $R\leq r_0(n,\lambda)$ and consider $u:B_R(0)\subset \R^2\to \R$ be a solution to \eqref{eq_Lu} with \eqref{e:coefficient_estimates} such that $N(0,R)\leq \Lambda$. There exists a constant $C(n,\lambda)$ such that
 \begin{gather}
  \Vol\ton{B_r (\cS_r(u))\cap B_{R/2}(0)} \leq C(\lambda)^\Lambda (r/R)^2\, .
 \end{gather}
As a corollary, the number of critical points of such a function is bounded above by $C(\lambda)^\Lambda$.
\end{theorem}

\clearpage

\begin{appendix}

\section{Nodal sets}\label{sec_nodal}
As anticipated in the introduction, with similar arguments one can obtain Minkowski estimates on the effective nodal set of a solution to an elliptic PDE. The results do not depend on whether (\ref{eq_Lu}) is critical or not, which is to say, whether $c\equiv 0$ or not.  One needs only change the symmetry results of Sections \ref{ss:symmetric_criticalpoints} and \ref{ss:critical_symmetry_elliptic} to reflect the nodal set as opposed to the critical set, and then, with exactly the same technique, it is possible to prove effective $n-1$ volume estimates on the tubular neighborhood of the set $u^{-1}(0)$. In particular, consider the two following simple propositions.
\begin{proposition}\label{prop_epsnod}
 If $u$ is a nonconstant harmonic function and $N(0,1)\leq 1/2$, then
 \begin{gather}
  u(0)^2\geq \frac 1 2 \fint_{\partial B_1(0)} u(x)^2 dS >0\, .
 \end{gather}
If $u$ solves \eqref{eq_Lu}, then there exist constants $c(n,\lambda),r_0(n,\lambda)$ such that if $N(0,r_1)\leq 1/2$ with $r_1\leq r_0$, then
\end{proposition}
\begin{gather}
 u(0)^2 \geq \frac 1 2 \fint_{\partial B_{cr_1}(0)} u(x)^2 dS >0\, .
\end{gather}
\begin{proof}
 The proof of the statement for harmonic function is analogous to the proof of Lemma \ref{l:frequency_comparison}. As for more general solutions, one can exploit the approximating harmonic function for $u$ to generalize the previous statement.
\end{proof}

In a completely similar way, one can prove the following proposition.
\begin{proposition}\label{prop_nod}
Let $u$ be a harmonic function. For every $\tau>0$, there exists $\epsilon(n,\tau)>0$ such that if $1-\epsilon \leq N(0,e^{-2})\leq N(0,e^2)\leq 1+\epsilon$, then $u$ does not have zeros in $B_1(0)\setminus B_\tau(V)$, where $V$ is some $n-1$ dimensional linear subspace of $\R^n$. Moreover, there exists $r_0=r_0(n)$ such for all $x\in B_1(0)\setminus B_r(V)$, $N(x,r_0)\leq \epsilon_0$.

In a similar way, if $u$ solves \eqref{eq_Lu}, there exists $r_0,\epsilon_0$ and $c$ depending only on $n,\lambda$ and $\tau$ such that if $1-\epsilon\leq N(0,r/c)\leq N(0,c r)\leq 1+\epsilon$, with $r\leq r_0$, then there exists an $n-1$ dimensional plane $V$ such that for all $x\in B_r(0)\setminus B_{\tau r} (V)$, $N(x, c r) \leq 1/2$.
\end{proposition}
\begin{proof}
If $u$ is harmonic, let $u= u(0) + a_1 P_1 +\sum_{k\geq 2} a_k P_k$. The pinching on $N$ implies that $u(0)^2\leq c \epsilon a_1^2$ and $\sum_{k\geq 2} a_k^2 e^{2k}\leq \epsilon$ as well. This in particular implies that
\begin{gather}
 \abs{u(x)-a_1P_1(x)}\leq \abs{u(0)} + \sum_{k\geq 2} \abs{a_k} \abs{P_k(x)}\leq \epsilon c(n)\ton{1+ \sum_{k\geq 2} e^{-2k} k^{n/2}}\leq c(n)\epsilon\, .
\end{gather}
Since $P_1(x)=\ps{L}{x}$, where $L$ is a vector of length $\sqrt n$, then we easily obtain the thesis. 

Again, for more general solutions, the proposition follows from an easy application of the approximation Theorem \ref{th_w}.
\end{proof}

With these ingredients, it is easy to generalize the estimate proved for the effective critical set and obtain the following
\begin{theorem}
 There exists $r_0=r_0(n,\lambda)$ such that if $u$ solves \eqref{eq_Lu} with \eqref{e:coefficient_estimates} on $B_{r_0}(0)\subset \R^n$ and if $N(0,r_0)\leq \Lambda$, then
 \begin{gather}
  \operatorname{Vol}\ton{B_r\ton{u^{-1}(0)} \cap B_{r_0/2}(0)} \leq \operatorname{Vol}\qua{B_r\ton{N(x,r)\leq\epsilon_0} \cap B_{r_0/2}(0)} \leq \ton{C(n,\lambda)\Lambda}^\Lambda r/r_0\, .
 \end{gather}

\end{theorem}
\begin{proof}
 The proof follows closely the proof of Theorem \ref{t:main_critical}. In this case however we are interested in $n-1$ Minkowski estimates, thus in the covering arguments for the good balls we make a distinction only between functions with $n-1$ or $n$ symmetries, and functions with at most $n-2$ symmetries. In this latter case, a simple covering argument of the whole good ball will do, while if a good ball is close to having at least $n-1$ symmetries, then by the cone splitting proved in \ref{prop_alcone} the dominant degree of this ball is either $0$ or $1$. Thus the $\epsilon$-regularity theorems just proved allow us to conclude the estimate.
\end{proof}

\section[Volume estimates on the critical and effective critical set for n=2]{Volume estimates on the critical and effective critical set for $n=2$}

In this appendix we give an alternate, simplified, proof of the main results for $n=2$ which allows for an easy improvement of the constants.  Namely, we prove that 
\begin{theorem}
Let $u:B_1(0)\subset \R^2 \to\dR$ solve (\ref{eq_Lu}) and satisfy (\ref{e:coefficient_estimates}). There exists $r_0=r_0(n,\lambda)>0$ with $r_0(n,0)=\infty$ and $C=C(n,\lambda)$ such that if $\Lambda\equiv N^u(0,s)$ for some $s\leq r_0$, and if \eqref{eq_Lu} is critical, then 
 \begin{gather}
  \#\cur{x \in B_{s/2}(0): |\nabla u|(x)=0} \leq e^{C \Lambda}\, .
 \end{gather}
 If (\ref{eq_Lu}) is not critical, then we have the estimate 
  \begin{gather}
  \#\cur{x \in B_{s/2}(0):|\nabla u|(x)=u(x)=0} \leq e^{C \Lambda}\, .
 \end{gather}
\end{theorem}
\begin{proof}
As before we will focus on the critical case, and we will assume $u$ is harmonic.  The technique is such that, verbatim as in Section \ref{s:general_elliptic} of the paper, with the appropriate approximation arguments the results all pass over to the general case.

By theorem \ref{th_cN}, $N(x,1/3)\leq C(n)\Lambda$ for all $x\in B_{1/2}(0)$. According to Proposition \ref{prop_h2effcrit}, there exists an $\epsilon_0$ independent of $\Lambda$ such that 
\begin{gather}
 N(x,re^2)- N(x,re^{-3}) \leq \epsilon_0 \ \ \ \ \text{for some} \ \ \ \ 0<r<r_0 \quad \Longrightarrow \quad \Cr(u) \cap B_r(x) \setminus B_{r/e}(x) = \emptyset\, .
\end{gather}

This, and monotonicity, means that every point can have at most
\begin{gather}
 K \leq \frac{4 c \Lambda}{\epsilon_0}
\end{gather}
critical scales. That is, for each critical point there are at most $K$ numbers $i$ such that $B_{e^{-i}}(x)\setminus B_{e^{-i-1}}(x)$ contains a critical point.

Now we proceed by induction on $i$. Let $A_0<\infty$ be the cardinality of $\Cr(u)\cap B_{1/2}(0)$. Define $T_i$ to be an infinite vector of zeros and ones, and let $\abs T = \sum_{i=1}^\infty T(i)$.

For $i=1$, consider all the balls of radius $e^{-i}$ centered at $x\in \Cru\cap B_{1/2}(0)$, and refine this covering of $\Cru$ by considering only a maximal subcovering such that $B_{e^{-i-1}}(x_j)$ are disjoint. This is obviously possible, and by simple volume estimates the number of balls in this covering are at most $c=e^4/4$.

Then refine further the covering by extracting a minimal subcovering with the property that each ball covers at least a point which is not covered by any other ball.

Now consider the ball in this covering that contains the largest number of critical points, say $B_{e^{-i}}(y_1)$, containing $A_i$ critical points. If $A_i = A_{i-1}$, then set $T_i=0$, otherwise evidently we have
\begin{gather}
A_{i-1}> A_i \geq c A_{i-1}\, .
\end{gather}
Moreover in this case (i.e., if $T(i)=1$) there also exists a critical point $x_i$ such that $e^{-i+1}\geq d(x_i,y_i)\geq e^{-i}$. Indeed, we assumed that the covering was minimal in this sense.

Now we repeat this process by induction and stop when $A_i =1$. Since the number or critical points is finite, the number of induction steps is finite. Set $\bar i$ to be the index relative to the last step. Evidently we have the estimate:
\begin{gather}
 A_0 \leq c^{\abs T} 
\end{gather}

In order to get a bound on $\abs T$, consider what happens if $T(i)=1$. As seen before, in this case there exists two critical points $x_i,y_i$ such that
\begin{gather}
 e^{-i+1}\geq d(x_i,y_i)\geq e^{-i}\, .
\end{gather}
Thus either $x_i$ or $y_i$ have the following property (call $z_i$ the one with the property): ALL the points in $B_{e^{-i-1}}(y_{i+1})$ have distance $\in[e^{-i-1},e^{i+1}]$ from $z_i$.

Now consider the critical point $y_{\bar i}$. Since it belongs to all the balls $B_{e^{-i}}(y_i)$, We know that this point has at least $\abs T$ critical scales, and now we can conclude $\abs T \leq K\leq 4c\Lambda/\epsilon_0$.
\end{proof}

With a similar argument, we can prove an effective version of this theorem.
\begin{theorem}
Let $u:B_1(0)\to\dR$ solve (\ref{eq_Lu}) and satisfy (\ref{e:coefficient_estimates}). There exists $r_0=r_0(n,\lambda)>0$ with $r_0(n,0)=\infty$ and $C=C(n,\lambda)$ such that if $\Lambda\equiv N^u(0,s)$ for some $s\leq r_0$, then 
 \begin{gather}
  \Vol\ton{B_r(\cS_r(u))\cap B_{s/2}(0)} \leq e^{C \Lambda}(r/s)^2\, .
 \end{gather}
\end{theorem}
\begin{proof}
Consider the set $\cS_r(u)$ and cover it with a Vitali covering of balls of radius $R=\bar \Lambda r/(5r_0)$. In detail
 \begin{gather}
  \cS_r(u)\subset \bigcup_{i=1}^M B_R(x_i) \, , \quad \quad x_i\in \cS_r(u)\, , \quad \quad B_{R_i/5}(x_i)\cap B_{R_j/5}(x_j)\neq \emptyset\, .
 \end{gather}
 Let $y\not \in B_{R_i/5}(x_i)$, and suppose that $\abs{y-x_i}$ is a good scale for $x_i$, meaning that
 \begin{gather}
N\ton{x_i,e^2 \abs{y-x_i}} -N\ton{x_i,e^{-2}\abs{y-x_i} }\leq \epsilon\, .  
 \end{gather}
Since we know that $N(x,1/3)\leq \bar \Lambda$ for all $x\in B_{1/2}(0)$, then $r_c(y)\geq \frac{Rr_0}{\bar \Lambda }\geq r$. 
 
By using the argument of the previous theorem, one proves that the number $M$ of centers of the covering has a uniform upper bound. Thus we obtain that
 \begin{gather}
  \Vol\ton{B_{R}(\cS_r(u))\cap B_{1/2}(0)} \leq e^{c \Lambda}R^2\, \quad \Longrightarrow \quad \, \Vol\ton{B_{r}(\cS_r(u))\cap B_{1/2}(0)} \leq C \Lambda^2 e^{c \Lambda}r^2\leq e^{c \Lambda}r^2\, .
 \end{gather}
\end{proof}

\end{appendix}

\bibliographystyle{aomalpha}
\bibliography{nava_bib}

\end{document}